\documentclass[11pt]{amsart}

\usepackage{graphicx} 
\usepackage{url,tabularx,array,geometry}
\usepackage{amssymb}
\usepackage{fancyhdr}
\usepackage{wrapfig}
\usepackage{epstopdf}
\usepackage{amsmath}
\usepackage[labelfont=bf]{caption}
\usepackage[hidelinks]{hyperref}
\usepackage{amsthm}
\usepackage{bbm}
\usepackage[T1]{fontenc}
\usepackage{yfonts}
\usepackage{breqn}
\usepackage{harpoon}
\usepackage[all]{xy}
\usepackage{tikz}
\usepackage{tikz-cd}
\usepackage{hyperref}
\usepackage{comment}
\usepackage{mathtools}
\usepackage{quiver}
\usepackage{multicol}
\usepackage{float}
\usepackage{blkarray}

\theoremstyle{plain}
\newtheorem{theorem}{Theorem}
\newtheorem{conjecture}{Conjecture}

\newtheorem{theoremN}{Theorem}[section]
\newtheorem{propositionN}[theoremN]{Proposition}
\newtheorem{corollaryN}[theoremN]{Corollary}
\newtheorem{lemmaN}[theoremN]{Lemma}
\newtheorem{conjectureN}[theoremN]{Conjecture}
\theoremstyle{definition}
\newtheorem{definitionN}[theoremN]{Definition}
\theoremstyle{remark}

\newtheorem*{note}{Note}

\newtheorem*{example}{Example}

\usepackage[maxbibnames = 50]{biblatex}
\renewbibmacro{in:}{}
\DeclareFieldFormat{pages}{#1}
\title[Geometric Approach to Links-Quivers Correspondence]{A Geometric Approach to the Links-Quivers Correspondence II: Rational Links}

\author{Jonathan A. Higgins}
\thanks{The
author was partially supported by NSF grant DMS-2405402.
}

\begin{document}

\address{Jonathan Higgins, Department of Mathematics, University of Illinois Urbana-Champaign, 1409 W Green St, Urbana, IL 61801, USA}
\email{jh110@illinois.edu}

\begin{abstract}
   Originally proposed in \cite{KRSS17,K19}, the Links-Quivers Correspondence predicts that the generating function for the symmetric (or antisymmetric) colored HOMFLY-PT polynomials for links can be put in a ``quiver form,'' so that the generating function is expressed in terms of a quadratic form and two linear forms. This was originally proved for rational links by Sto\v si\'c and Wedrich \cite{SW21}, but here we give a direct geometric description of the linear and quadratic forms in terms of the first and second configuration spaces of the 3-punctured plane.
\end{abstract}


\maketitle

\tableofcontents

\section{Introduction}
\label{introsec}


The colored HOMFLY-PT polynomials of a link $L$ are an infinite family of invariants known to generalize many other link polynomials. Consequently, computing them is generally tedious and difficult, but the Links-Quivers Correspondence of Kucharski, Reineke Sto\v si\'c, and Sułkowski \cite{KRSS17,K19} predicts that the generating function for the symmetric-colored HOMFLY-PT polynomials (or the antisymmetric ones) can be expressed in terms of a finite amount of data. As the name suggests, given any link $L$, the conjecture predicts that we can associate a symmetric quiver $Q_L$ with $L$ such that the generating function of the colored HOMFLY-PT polynomials for $L$, which we call $P(L)$, can be expressed in terms of a generating function for $Q_L$. Since its conception, there has been extensive work on studying the structure of these quivers \cite{CKNPSS25,EKL20,JKLN21,KLNS26} and proving the conjecture for families of knots \cite{SW21,SW21ii}. See also \cite{EKL23,JGKM23,KRSS20,PSS18,SS25} for further developments and \cite{EGGKPSS22,K20} for a proposed extension to knot complements. In \cite{SW21}, Sto\v si\'c and Wedrich proved the conjecture for rational links; in this paper, we provide a new geometric interpretation of this result. We will focus on the antisymmetric-colored HOMFLY-PT polynomials, but our results will also extend to the symmetric-colored ones. Following \cite{JHpI26}, where we did so for rational tangles, we will now determine the quiver for rational links by using intersection models of Lagrangians in the first and second configuration spaces of the punctured plane. 

Originally formulated by Hoste, Ocneanu, Millett, Freyd, Lickorish, and Yetter in \cite{HOMFLY85} and Przytcycki and Traczyk in \cite{PT87}, the (uncolored) HOMFLY-PT polynomial was defined as a two-variable generalization of the Jones and Alexander polynomials. In his seminal paper \cite{J87}, Jones described the HOMFLY-PT polynomial of a link as the Ocneanu trace of the image in the Hecke algebra of a braid whose closure yields the link. The colored HOMFLY-PT polynomials, which are indexed by the Young diagrams (in their fullest generality), may then be defined carefully by cabling the braid representation of the link, inserting a ``projector'' associated with the Young diagram, and taking the trace.

In \cite{RT90}, Reshetikhin and Turaev showed that we get invariants of tangles by coloring their strands with irreducible representations of quantum groups associated with simple Lie algebras. In particular, the Type-A invariants are obtained by using irreducible representations of $U_q(\mathfrak{sl}_N)$. For our purposes, given a fixed $N\in \mathbb{N}$, we will think of the $j$-colored $\mathfrak{sl}_N$ polynomial of a link $L$ as the one we get from coloring all strands of $L$ with the miniscule representation $\Lambda^j V_N$, where $V_N$ is the vector representation for $U_q(\mathfrak{sl}_N)$ and $0\leq j\leq N$. If we denote this polynomial invariant for $L$ as $P_{L,N}^{\bigwedge^j}(q)$, it is known that they stabilize for $N\gg 0$ in the sense that 
\[
P_{L,N}^{\bigwedge^j}(q)=P_L^{\bigwedge^j}(q,a)\bigg|_{a\mapsto q^N}
\]
results in no cancellations for large $N$, where $P_L^{\bigwedge^j}(q,a)$ is what we will call the $j$-colored HOMFLY-PT polynomial. In terms of the previous paragraph, this is the colored HOMFLY-PT polynomial indexed by the single column Young diagram with $j$ boxes.

Given a symmetric quiver $Q$ with $m$ vertices, the generating function for its cohomological Hall algebra is given by
\begin{equation}
P_Q(x_1,...,x_m) = \sum_{\textbf{d}=(d_1,..,d_m)\in\mathbb{N}^m} (-q)^{-\langle \textbf{d},\textbf{d}\rangle_Q}\prod_{i=1}^m \frac{x_1^{d_1}...x_m^{d_m}}{(1-q^{-2})...(1-q^{-2d_i})},
\end{equation}
where $Q$ also denotes its adjacency matrix, $\langle \textbf{d},\textbf{e}\rangle_Q = \sum_{i,j}(\delta_{ij}- Q_{ij})d_ie_j$ is its Euler form, and $x_1,...,x_m$ are formal variables. The Links-Quivers Correspondence relates generating functions of this form with the generating functions of colored HOMFLY-PT polynomials for links. If $L$ is a link with $|L|$ components, we may define the generating function
\begin{equation}
    \widehat{P}(L)=\sum_{j\geq 0}\left(\prod_{i=1}^j(1-q^{2i})\right)^{|L|-2}P^{\bigwedge^j}_L(q,a)x^j
\end{equation}
which is a rescaled version of the standard colored HOMFLY-PT generating function,
\begin{equation}
    P(L)=\sum_{j\geq 0}P^{\bigwedge^j}_L(q,a)x^j,
\end{equation}
which is what we will be working with mostly in this paper. However, the Links-Quivers Correspondence is best stated in terms of $\widehat{P}(L)$.

\begin{conjecture}[Links-Quivers Correspondence]
\label{LQC1intro}
    For any link $L$, there is a symmetric quiver $Q_L$ such that 
   \[
   \widehat{P}(L)=P_{Q_L}(\overline{x})\bigg|_{x_i\mapsto (-1)^{Q_{ii}-S_i}q^{S_i-1}a^{A_i}x},
   \]
   where $S$ and $A$ are vectors of length given by the number of vertices in $Q_L$ with adjacency matrix $Q$.
\end{conjecture}

In this paper, we will think of the adjacency matrix $Q$ as a quadratic form and the vectors $S$ and $A$ as linear forms on a free $\mathbb{Z}$-module equipped with a preferred basis. For a reduced fraction $u/v\in\mathbb{Q}^{\geq 1}$, we get a rational knot $K_{u/v}$ if $u$ is odd and a 2-component rational link $L_{u/v}$ if $u$ is even; associated with these links are two types of Lagrangians in the 3-punctured plane whose intersections provide the preferred bases for these free $\mathbb{Z}$-modules. For a rational knot $K_{u/v}$, we will use $\mathcal{D}(K_{u/v})$ to denote the picture in the punctured plane consisting of these two Lagrangians, and $\mathcal{D}(L_{u/v})$ is defined similarly. See Figure \ref{Dfigintro}. Now, we state the main theorems for this paper.

\begin{theorem}
\label{knotthmintro}
    For a rational knot $K_{u/v}$, we may express $P(K_{u/v})$ as 
    \begin{equation}
    \label{PKuv}
        P(K_{u/v})=\sum_{\textbf{d}=(d_1,...,d_u)\in \mathbb{N}^u}(-q)^{S\cdot \textbf{d}}a^{A\cdot \textbf{d}}q^{\textbf{d} \cdot Q\cdot \textbf{d}^t} {d_1+...+d_u\brack d_1,...,d_u},
    \end{equation}
    where $S$ and $A$ are computed by winding numbers of loops based at the $u$ Lagrangian intersections in $\mathcal{D}(K_{u/v})$ about the three punctures. The quadratic form $Q$ may be computed similarly, but by passing to the second configuration space of the 3-punctured plane, where we also need to consider winding numbers of loops around the diagonal $\Delta$.
\end{theorem}

The theorem for 2-component links is similar, but the colored HOMFLY-PT invariants are no longer Laurent polynomials.

\begin{theorem}
\label{linkthmintro}
    For a 2-component rational link $L_{u/v}$, we may express $P(L_{u/v})$ as 
    \begin{equation}
    \label{PLuv}
   P(L_{u/v})=\sum_{\textbf{d}=(d_1,...,d_{2u})\in \mathbb{N}^{2u}} \frac{(-q)^{S\cdot \textbf{d}}a^{A\cdot \textbf{d}}q^{\textbf{d}\cdot Q \cdot\textbf{d}^T}}{\prod_{i=1}^{2u}(q^2;q^2)_{d_i}}x^{d_1+...+d_{2u}},
\end{equation}
    where $S$ and $A$ are computed by winding numbers of loops based at the $2u$ Lagrangian intersections in $\mathcal{D}(L_{u/v})$ about the three punctures. The quadratic form $Q$ may be computed similarly, but by passing to the second configuration space of the 3-punctured plane, where we also need to consider winding numbers of loops around the diagonal $\Delta$.
\end{theorem}

In \cite{SW21}, Sto\v si\'c and Wedrich proved the Links-Quivers Correspondence for rational links by showing that $P(K_{u/v})$ and $P(L_{u/v})$ may be written as in (\ref{PKuv}) and (\ref{PLuv}), which were called the ``quiver forms'' for the generating functions. Our contribution here is to see how the quiver forms relate to the geometry of the links, thus providing a ``geometrically preferred'' quiver.

The statement of Theorems \ref{knotthmintro} and \ref{linkthmintro} will be formalized and proved in Sections \ref{knotssec} and \ref{linkssec}, respectively. The proofs will heavily rely on \cite{JHpI26}, the prequel to this paper, which built upon the work of Wedrich \cite{W16} for rational tangles. In particular, following the procedure of Sto\v si\'c and Wedrich in \cite{SW21}, we will think of rational links as closures of rational tangles; then, the $j$-colored HOMFLY-PT polynomial of the closure of a rational tangle $\tau_{u/v}$ is computed by taking the closure of the skein module evaluation $\langle\tau_{u/v}\rangle_j$ for the tangle in an appropriately defined ``colored HOMFLY-PT'' skein module. In particular, this can be thought of as the colored $\mathfrak{sl}_N$ skein module for $N\gg 0$, where the dependence on $N$ is removed by setting $a=q^N$. See \cite{JHpI26} for further details.

\begin{figure}
    \begin{tikzpicture}
        \node at (0,0) {\includegraphics[height=3.5cm]{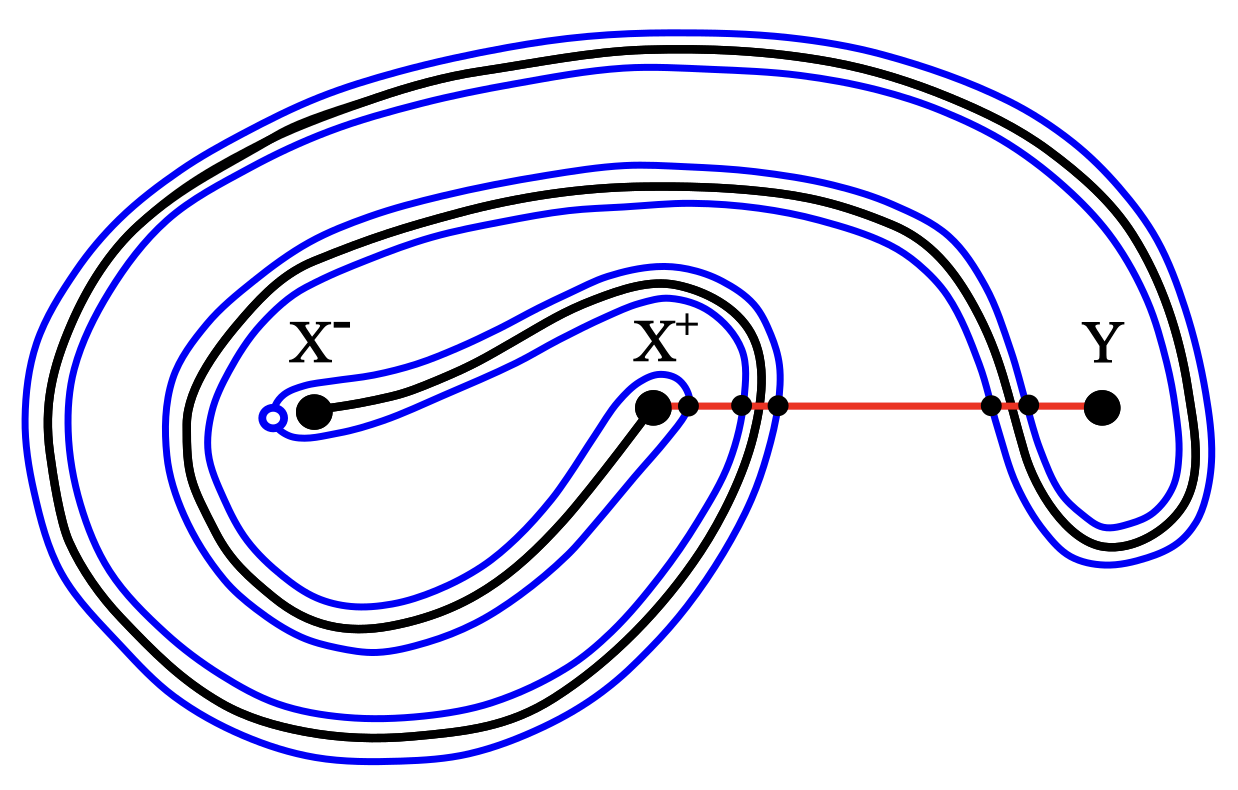}};
        \node at (7,0) {\includegraphics[height=4cm]{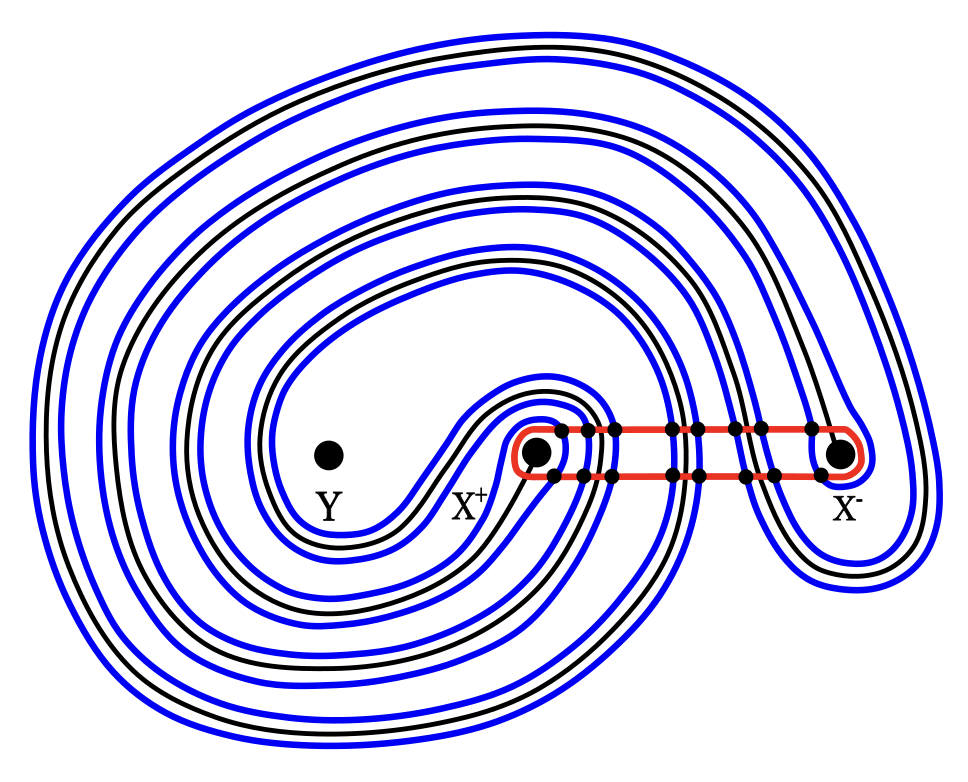}};
    \end{tikzpicture}
    \caption{Left: $\mathcal{D}(K_{5/2})$. Right: $\mathcal{D}(L_{8/3}).$}
    \label{Dfigintro}
\end{figure}

Especially important to contemporary work in the study of knot invariants is categorification. After Khovanov's original categorification of the Jones polynomial in \cite{Kh00} (see also \cite{BN02}), mathematicians started seeking categorifications of other quantum invariants. Khovanov's original result was extended to tangles by Bar-Natan in \cite{BN05} and to the (uncolored) $\mathfrak{sl}_N$ invariants by Khovanov and Rozansky in \cite{KR08a}. Around the same time, Khovanov and Rozansky categorified the HOMFLY-PT polynomial in \cite{Kh07} and \cite{KR08b}, and this was then extended to colored HOMFLY-PT homology in \cite{MSV11} by Makaay, Sto\v si\'c, and Vaz, and by Webster and Williamson in \cite{WW17}. An excellent exposition to knot polynomials and categorification can be found in \cite{R21} by Rasmussen. 

Colored HOMFLY-PT homology is triply graded, so if we let $C\mathcal{H}_{i,j,k}^{\bigwedge^r}$ denote the portion of the $r$-colored HOMFLY-PT chain complex, defined up to homotopy, in $(q,a,t)$-grading $(i,j,k)$, where $t$ is the homological grading, then the Poincar\'e polynomial is
\begin{equation}
\label{Poincarepoly}
    \mathcal{P}_L^{\bigwedge^r}(q,a,t)=\sum_{i,j,k}q^ia^jt^k \text{dim}(C\mathcal{H}_{i,j,k}^{\bigwedge^r}).
\end{equation}
Naturally, setting $t=-1$ corresponds to taking the graded Euler characteristic, which gives the $r$-colored HOMFLY-PT polynomial $P_L^{\bigwedge^r}(q,a)$. An extension to Conjecture \ref{LQC1intro} for Poincar\'e polynomials is also given in \cite{K19}, which predicts that the colored HOMFLY-PT homology of $L$ can be represented by a chain complex whose homological gradings are determined by the diagonal of $Q_L$. If
\begin{equation}
    \mathcal{P}(L)=\sum_{r\geq 0} \mathcal{P}_L^{\bigwedge^r}(q,a,t)x^j
\end{equation}
is the generating function for the Poincar\'e polynomials and
\begin{equation}
    \widehat{\mathcal{P}}(L)= \sum_{r\geq 0} \left(\prod_{i=1}^r (1-q^{2i})\right)^{|L|-2}\mathcal{P}_L^{\bigwedge^r}(q,a,t)x^j
\end{equation}
is the rescaled version, then we may express the extension of the Links-Quivers Correspondence to Poincar\'e polynomials by Conjecture \ref{Poincarepolyconj} below, which is a reformulation of Conjecture 4.4 in \cite{K19}. In particular, we are applying a change of variables to the quadruply graded Poincar\'e polynomials in the the conjecture from \cite{K19} to give an appropriate form for the antisymmetric-colored HOMFLY-PT homology.

\begin{conjecture}
\label{Poincarepolyconj}
    Given a link $L$ with the exponential growth property, the colored HOMFLY-PT homologies may be represented by chain complexes such that
    \begin{equation}
        \widehat{\mathcal{P}}(L)=\sum_{\textbf{d}=(d_1,...,d_m)\in \mathbb{N}^m} \frac{q^{S\cdot \textbf{d}+\textbf{d}\cdot Q \cdot\textbf{d}^T}a^{A\cdot \textbf{d}}t^{T\cdot\textbf{d}}}{\prod_{i=1}^m(q^2;q^2)_{d_i}}x^{d_1+...+d_m}
    \end{equation}
    For some $m$, where $T=(t_1,...,t_m)\in\mathbb{Z}^m$ such that $t_i=-Q_{ii}$.
\end{conjecture}

In \cite{SW21}, Sto\v si\'c and Wedrich showed that the uncolored HOMFLY-PT homology of rational knots can be represented by a chain complex whose homological gradings are given by the diagonal entries of $-Q$. Naturally, being able to address the conjecture requires an understanding of the whole colored HOMFLY-PT chain complexes, involving the differentials. We will not investigate this here, but we will address the conjecture in terms of our geometric framework when possible. 



\bigskip
\noindent
\textbf{Acknowledgments} The author would like to thank his PhD advisor, Jacob Rasmussen, for his helpful comments and support while writing this paper. He would also like to thank Piotr Kucharski, Marko Sto\v si\'c, Piotr Su\l{}kowski, and Paul Wedrich for helpful comments on an earlier draft.

\bigskip
\noindent
\textbf{Structure of Paper} Section \ref{bgrd} will contain some background information necessary for this paper. This will include some further details on quantum algebra, rational tangles and links, colored HOMFLY-PT invariants, and the Links-Quivers Correspondence. Section \ref{LQCTsec} will be devoted to providing a quick review of the necessary constructions and results from \cite{JHpI26}, the prequel to this paper. In sections \ref{knotssec} and \ref{linkssec}, we will extend the results of \cite{JHpI26} to knots and links, respectively, describing in detail the geometric setup in the process. Examples will be provided at the end of sections \ref{knotssec} and \ref{linkssec} to demonstrate the theory just presented. We will conclude with section \ref{sympolysec}, where we will say a few words about the symmetric-colored invariants.

\section{Background}
\label{bgrd}



\subsection{Quantum Algebra}

We begin by introducing some basic definitions and notation from quantum algebra that will be needed later.

\begin{definitionN}
    The \textit{Pochhammer symbols}, in their most general form, may be defined as $(x;y)_i=\prod_{j=0}^{i-1}(1-xy^j)$.

\end{definitionN}

The most common type of Pochhammer symbols that we will use are when $x=y=q^2$, or
\[
(q^2;q^2)_i=\prod_{j=1}^i(1-q^{2i}).
\]
We will, however, encounter other Pochhammer symbols as well, including
\[
(a^2q^{2-2j-2k};q^2)_k=\prod_{i=1}^k(1-a^2q^{2i-2j-2k})
\]

We can use Pochhammer symbols to define quantum multinomials.

\begin{definitionN}
    Suppose that $d_1,...,d_m\in\mathbb{Z}^{\geq0}$ are such that $d_1+...+d_m=d$. Then, we may define the \textit{quantum multinomial} as the polynomial
    \begin{equation*}
    {d\brack d_1,...,d_m}=\frac{(q^2;q^2)_d}{(q^2;q^2)_{d_1}...(q^2;q^2)_{d_m}}\in \mathbb{Z}[q^2].
\end{equation*}
\end{definitionN}

Naturally, the quantum binomials are the case where $m=2$, and we adopt the notation of \cite{SW21} to write
\begin{equation*}
    {j\brack d_1}_+={j\brack d_1,j-d_1}=\frac{(q^2;q^2)_j}{(q^2;q^2)_{d_1}(q^2;q^2)_{j-d_1}}.
\end{equation*}

Finally, we present the following lemma from \cite{K19}, which is a helpful tool for many computations in quantum algebra. 

\begin{lemmaN}
\label{algidentity}
    For $d_1,...,d_m\geq 0$, the following holds:
    \begin{equation*}
        \frac{(x^2;q^2)_{d_1+...+d_m}}{(q^2;q^2)_{d_1}...(q^2;q^2)_{d_m}}=\sum_{\substack{\alpha_1+\beta_1=d_1\\...\\\alpha_m+\beta_m=d_m}}\frac{(-x^2q^{-1})^{\alpha_1+...+\alpha_m}q^{\alpha_1^2+...+\alpha_m^2+2\sum_{i=1}^{m-1}\alpha_{i+1}(d_1+...+d_i)}}{(q^2;q^2)_{\alpha_1}...(q^2;q^2)_{\alpha_m}(q^2;q^2)_{\beta_1}...(q^2;q^2)_{\beta_m}}.
    \end{equation*}
\end{lemmaN}

\subsection{Rational Tangles}


The \textit{(positive) rational tangles} are an infinite family of oriented tangles with two strands. They can be defined inductively by starting with the trivial tangle consisting of two untangled and upwards-oriented strands, denoted $\tau_{0/1}$, and then applying a finite sequence of top and right twists. See Figure \ref{trivialandtwists}. 

\begin{figure}
    \raisebox{0pt}{\includegraphics[height=3cm, angle=0]{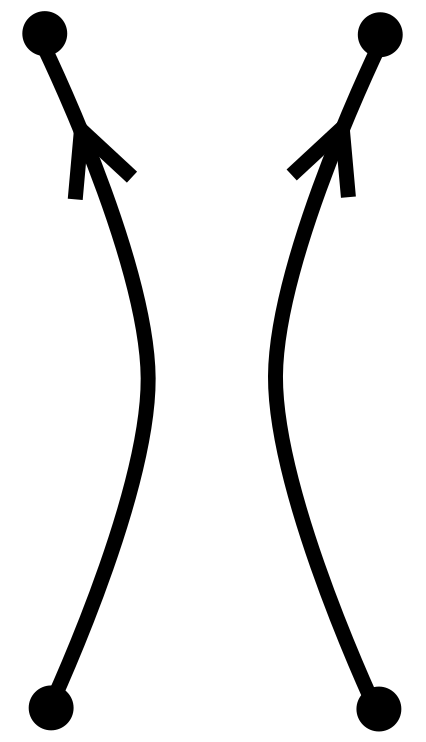}}\qquad \qquad \raisebox{0pt}{\includegraphics[height=3cm, angle=0]{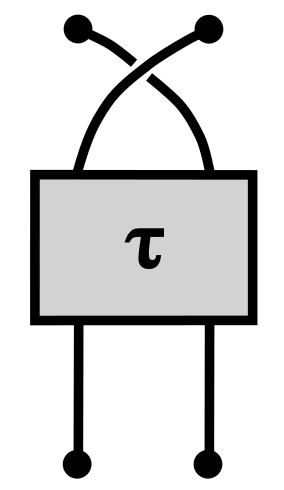}} \qquad \qquad \raisebox{0pt}{\includegraphics[height=3cm, angle=0]{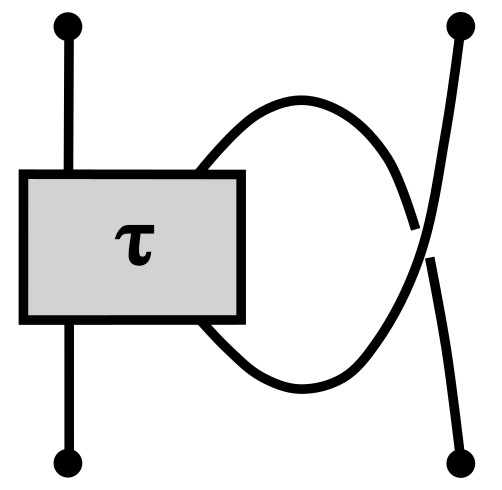}}
    \caption{Left: the trivial tangle, $\tau_{0/1}$. Center: top twist rule. Right: right twist rule.}
    \label{trivialandtwists}
\end{figure}

Associated to a positive rational tangle is a reduced fraction $u/v\geq 0$, which is known to be an isotopy invariant for the tangle. This fraction can be computed from the sequence of top and right twists applied to $\tau_{0/1}$. We will write $\tau_{u/v}$ for the rational tangle associated with $u/v$; then, we may write the effects of the twists on $u/v$ as follows:
\[
T\tau_{u/v}=\tau_{(u+v)/v}
\qquad\qquad 
R\tau_{u/v}=\tau_{u/(u+v)},
\]
where $T$ and $R$ denote the top and right twists, respectively. 

To construct a rational tangle $\tau_{u/v}$, we can always start by applying a top twist $T$ to $\tau_{0/1}$ since $R\tau_{0/1}=\tau_{0/1}$. If $u/v\geq 1$, then the final twist is a top twist, so we can write $T^{a_n}R^{a_{n-1}}...R^{a_2}T^{a_1}\tau_{0/1}=\tau_{u/v}$ with each $a_i>0$ and $n$ odd, where the rules for $T$ and $R$ give 
\[
\frac{u}{v}=a_n+\frac{1}{a_{n-1}+\frac{1}{a_{n-2}+...+\frac{1}{a_2+\frac{1}{a_1}}}}.
\]
Alternatively, the final twist is a right twist if $0<u/v<1$, so we have $R^{a_n}T^{a_{n-1}}...R^{a_2}T^{a_1}\tau_{0/1}=\tau_{u/v}$ in this case, so $n$ is even and
\[
\frac{u}{v}=\frac{1}{a_n+\frac{1}{a_{n-1}+...+\frac{1}{a_2+\frac{1}{a_1}}}}.
\]
Thus, given a non-negative rational number, we can express it in terms of a continued fraction expansion in one of these forms, which then determines the particular sequence of top and right twists needed to get $\tau_{u/v}$. In particular, the continued fraction expansion should have odd length if and only if $u/v\geq 1$, which can always be guaranteed.

Below are a couple of examples.
\[
\begin{tikzpicture}
    \node at (0,0) {$\tau_{3/1}=T^3\tau_{0/1}=$};
    \node at (2.8,0) {\includegraphics[height=2.5cm]{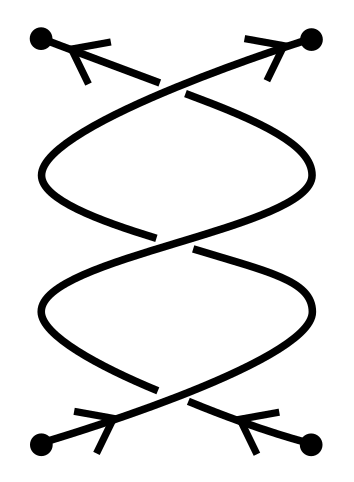}};
    \node at (7,0) {$\tau_{5/2}=T^2RT\tau_{0/1}=$};
    \node at (10,0) {\includegraphics[height=2.5cm]{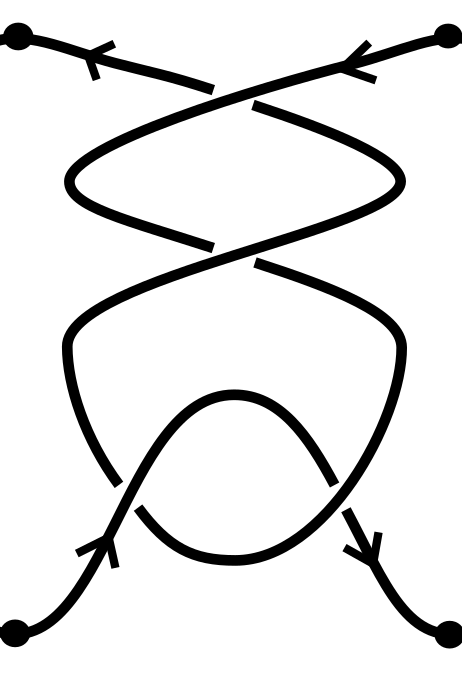}};
\end{tikzpicture}
\]
These examples demonstrate how different sequences of top and right twists can result in different orientations of the strands at the boundary. There are only three possible orientations,  and we will denote them by $UP, OP,$ and $RI$ to follow the conventions of Sto\v si\'c and Wedrich \cite{SW21}. The three possible orientations are shown below.

\[
\vcenter{\hbox{\includegraphics[height=3cm,angle=0]{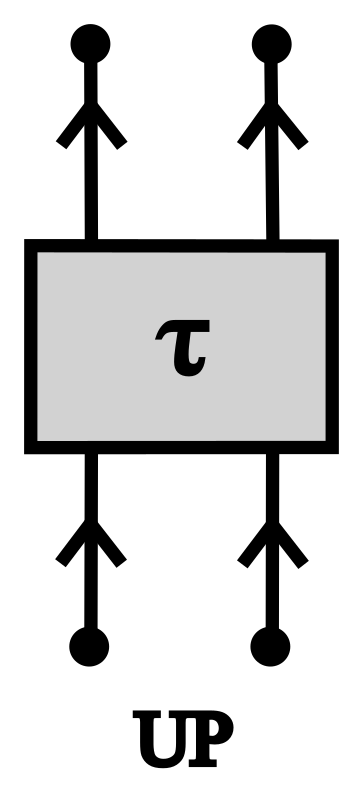}}} \qquad \qquad 
\vcenter{\hbox{\includegraphics[height=3cm,angle=0]{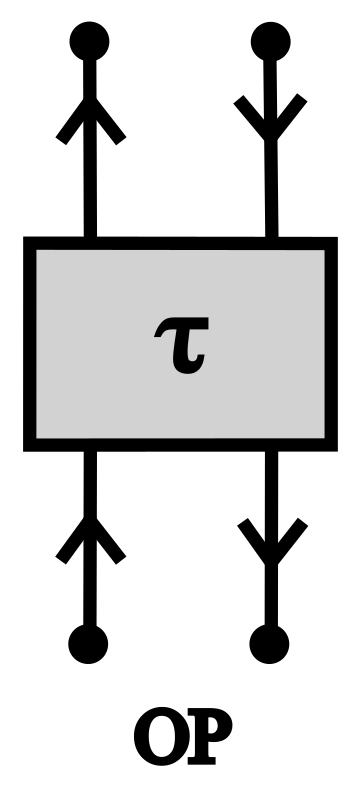}}} \qquad \qquad
\vcenter{\hbox{\includegraphics[height=3cm,angle=0]{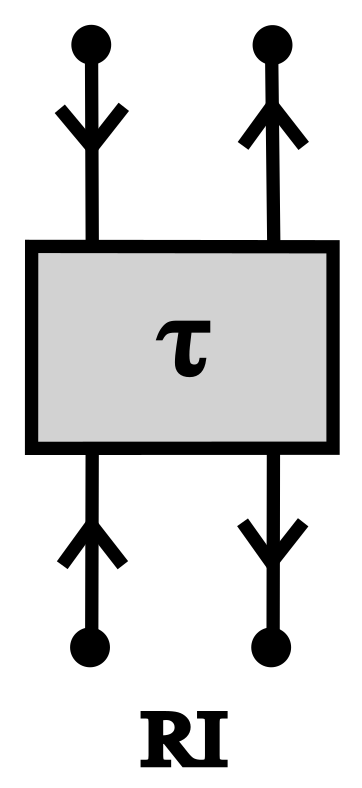}}}
\]
Thus, we see that $\tau_{3/1}$ has $UP$ orientation and $\tau_{5/2}$ has $OP$ orientation. For more details on rational tangles, see Section 2.2 of \cite{JHpI26}.

\subsection{Closures of Rational Tangles}

There are two closure operations on rational tangles that give the rational links. The first, known as the \textit{numerator closure}, is where the NW/SW and NE/SE points are glued together. This may also be called the \textit{N/S closure} because the tangle endpoints are being attached to each other vertically. The other closure operation is known as the \textit{denominator closure}, or the \textit{E/W closure}, where the NW/NE and SW/SE endpoints are glued together.

\begin{definitionN}
    A \textit{(positive) rational link} is a link obtained by taking a closure of a (positive) rational tangle.
\end{definitionN}

Naturally, orientations impose restrictions on which closure operations can be applied to particular rational tangles. In particular, $\tau_{u/v}$ must have $UP$ or $OP$ orientation in order take its numerator closure; similarly, $\tau_{u/v}$ must have $RI$ or $OP$ orientation to take its denominator closure. 

Furthermore, if we ignore orientations for a moment, it is easy to see that taking closures of rational tangles often results in nugatory crossings. If $u/v>1$, so that $\tau_{u/v}$'s final twist is a top twist, then taking its denominator closure will result in nugatory crossings which could be resolved by applying Reidemeister I moves. Similarly, if $u/v<1$, then taking the numerator closure results in nugatory crossings. Naturally, we wish to avoid creating these nugatory crossings; consequently, we only take the numerator closure when $u/v\geq 1$ (if the orientation allows it), and we only take the denominator closure when $u/v<1$ (orientation permitting).

This discussion illustrates how rational links may be given by closures of many different rational tangles. We will not say more about this now, but the curious reader can consult \cite{Ku96} to learn more. However, we will need the following lemma on rational links, which was proved in \cite{SW21}.

\begin{figure}
\begin{tikzpicture}
    \node at (-2.5,0) {\includegraphics[height=5cm]{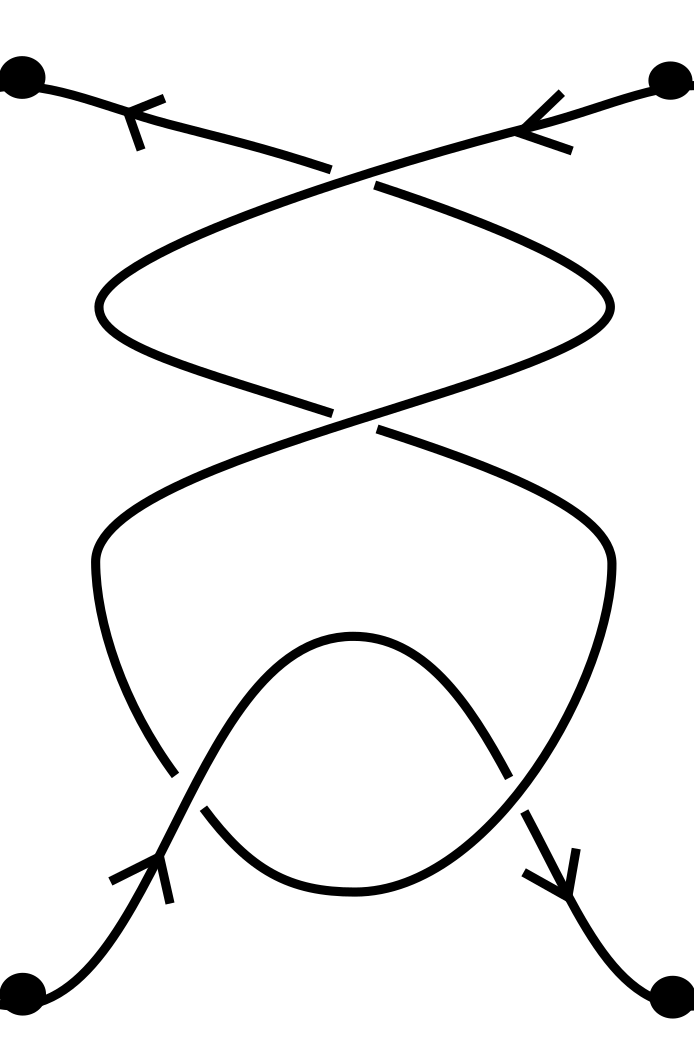}};
    \node at (0,0) {$\xrightarrow{\text{Cl}(-)}$};
    \node at (4,0) {\includegraphics[height=5cm]{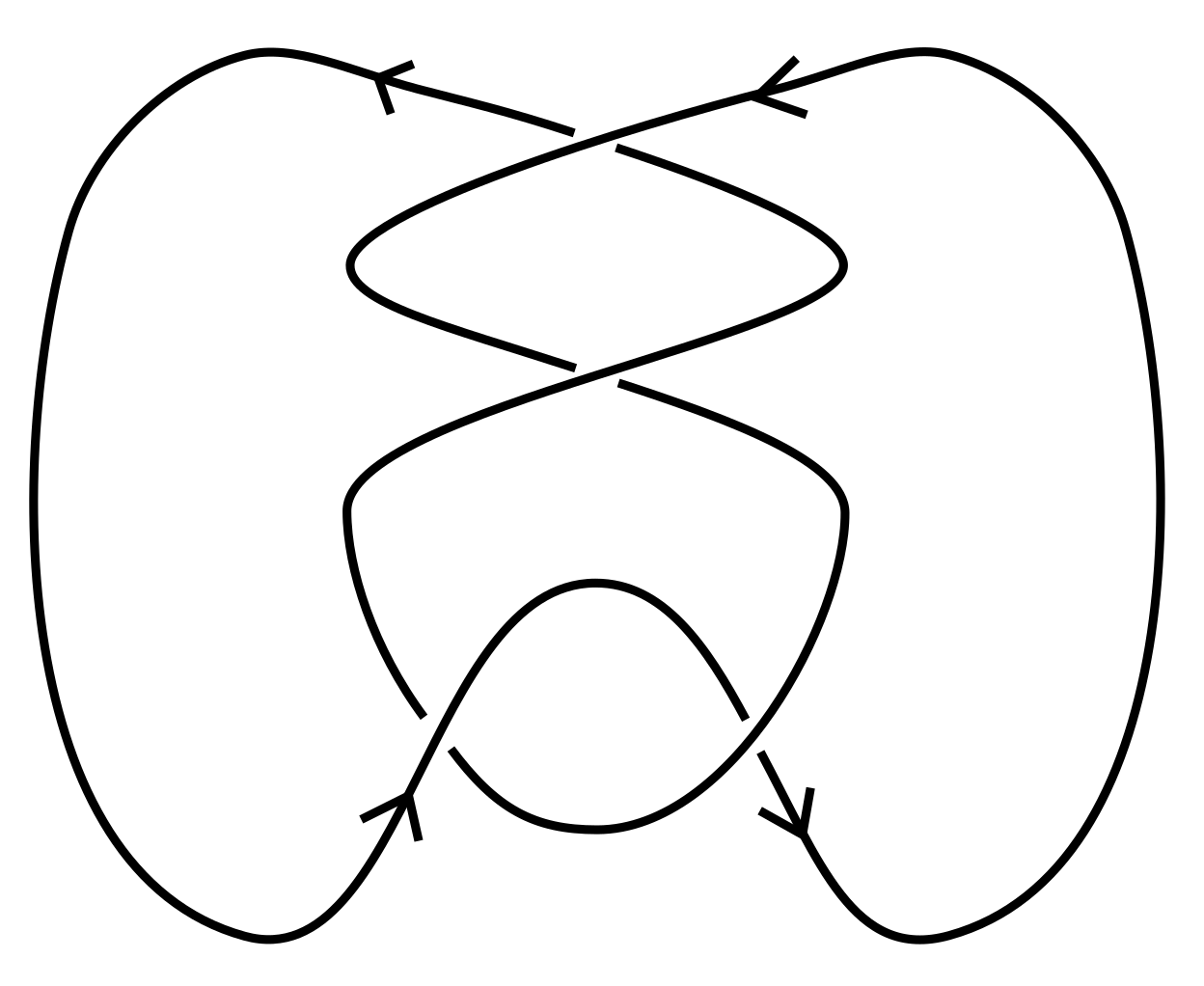}};
\end{tikzpicture}
    \caption{The figure-8 knot, $K_{5/2}$, taken as $\text{Cl}(\tau_{5/2})$}
    \label{K52closurefig}
\end{figure}

\begin{lemmaN}
\label{ratlinkslem}
    Any nontrivial positive rational link may be obtained by taking the numerator closure of a rational tangle $\tau_{u/v}$ with $u/v>1$ having $UP$ or $OP$ orientation. Furthermore, the closure is a knot if and only if $u$ is odd, and in this case, $\tau_{u/v}$ has $UP$ orientation if $v$ is odd and $OP$ orientation if $v$ is even. 
\end{lemmaN}

This lemma conveniently restricts the types of tangles and closures we need to concern ourselves with when working with rational links.

\begin{definitionN}
    Given a rational tangle $\tau_{u/v}$ with $u/v>1$ and $UP$ or $OP$ orientation, let $\text{Cl}(\tau_{u/v})$ denote its numerator closure. If $\text{Cl}(\tau_{u/v})$ is a knot (so $u$ is odd), write $K_{u/v}$ for $\text{Cl}(\tau_{u/v})$; otherwise, $u$ is even and the closure is a 2-component link, so we write $L_{u/v}$ for $\text{Cl}(\tau_{u/v})$.
\end{definitionN}

Next, we will discuss the colored HOMFLY-PT polynomials for rational links and establish the conventions we will use for working with them.

\subsection{Colored HOMFLY-PT Polynomials of Rational Links}
\label{Cpolysec}

First, we need to briefly review the skein module evaluations of rational tangles. We will ultimately think of the colored HOMFLY-PT polynomials as the closures of skein module evaluations in an appropriately defined ``colored HOMFLY-PT skein module.'' For a rational tangle $\tau_{u/v}$, we write $\langle \tau_{u/v}\rangle_j$ for its $j$-colored HOMFLY-PT skein module evaluation, so, in particular, the colored HOMFLY-PT polynomial of $\text{Cl}(\tau_{u/v})$ will be computed by taking the closure of $\langle \tau_{u/v}\rangle_j$. This will ultimately evaluate to a Laurent polynomial for knots and a rational expression for links with multiple components.

The skein module evaluation $\langle \tau_{u/v}\rangle_j$ is a linear combination over $\mathbb{C}[q^{\pm1},a^{\pm 1}]$ of basis webs $X[j,k]$, where $X\in\{UP,OP,RI\}$ denotes the orientation of $\tau_{u/v}$. Figure \ref{basiswebs} shows how the webs are drawn for each $X$. We have $0\leq k \leq j$ and $0$-labeled edges are deleted. 

\begin{figure}
    \raisebox{0pt}{\includegraphics[height=2.5cm, angle=0]{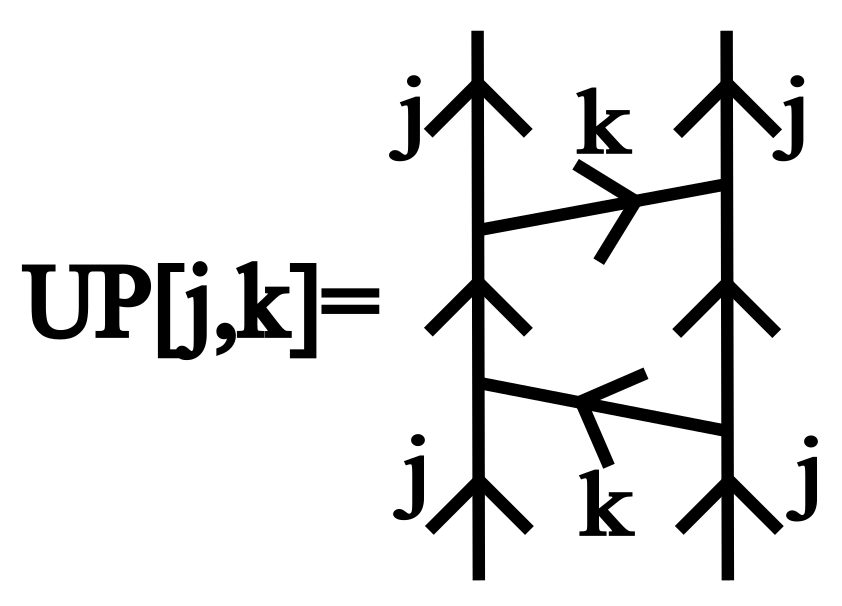}}\qquad \qquad \raisebox{0pt}{\includegraphics[height=2.5cm, angle=0]{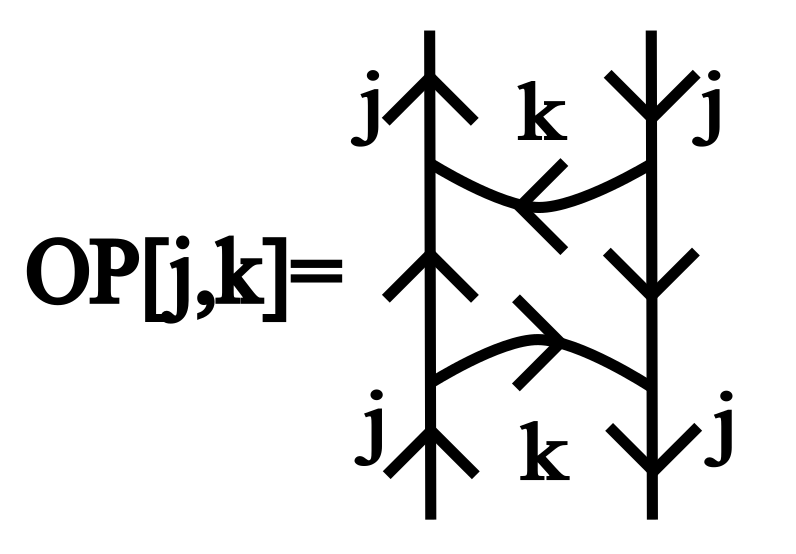}} \qquad \qquad \raisebox{0pt}{\includegraphics[height=2.5cm, angle=0]{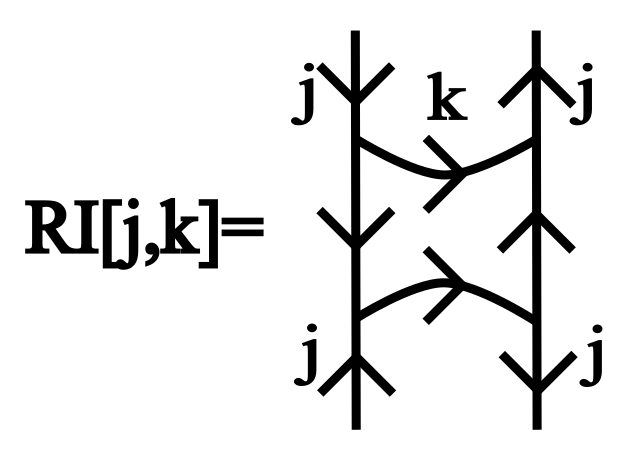}}
    \caption{The three types of basis webs}
    \label{basiswebs}
\end{figure}

As we discussed in \cite{JHpI26}, $\langle \tau_{u/v}\rangle_j$ can be computed inductively by applying the twist rules of Wedrich \cite{W16}. In particular, one can start with $\langle \tau_{0/1}\rangle _j= UP[j,0]$, and then apply the same sequence of top and right twists used to build $\tau_{u/v}$, applying the following rules at each step:
\begin{enumerate}
  \item $\text{TUP}[j,k]= \sum_{h=k}^j (-q)^hq^{k^2}{h\brack k}_+\text{UP}[j,h]$ \label{TUPrule}
\item  $\text{TOP}[j,k]= \sum_{h=k}^j (-q)^ha^kq^{k(k-2j)}{h\brack k}_+\text{RI}[j,h]$
\item $\text{TRI}[j,k]= \sum_{h=k}^j (-q)^ha^hq^{q^2-2hj}{h\brack k}_+\text{OP}[j,h]$
\item $\text{RUP}[j,k]= \sum_{h=0}^k (-q)^ha^hq^{k(2j-k)-2hj}{j-h\brack k-h}^-\text{OP}[j,h]$
\item $\text{ROP}[j,k]= \sum_{h=0}^k (-q)^ha^kq^{-k^2}{j-h\brack k-h}^-\text{UP}[j,h]$
\item $\text{RRI}[j,k]= \sum_{h=0}^k (-q)^hq^{-k(k-2j)}{j-h\brack k-h}^-\text{RI}[j,h].$
\end{enumerate}

In \cite{JHpI26}, we presented this information in a slightly different form. In particular, we were primarily interested in the Poincar\'e polynomials $\langle\langle -\rangle\rangle_j$ associated with rational tangles, which were linear combinations of the $X[j,k]$ over $\mathbb{C}[q^{\pm1},a^{\pm1},t^{\pm1}]$. However, $\langle\tau_{u/v}\rangle_j=\langle\langle\tau_{u/v}\rangle\rangle_j|_{t=-1}$, so the twist rules above follow from the ones in \cite{JHpI26} by substituting $t=-1$. Refer to \cite{JHpI26} for more details on skein module evaluations. 

There are many things one can say about colored HOMFLY-PT polynomials of links, but we will say here only what is necessary. Although they may be defined directly using representation theory, we will compute them by taking closures of skein module evaluations for rational tangles, as we already mentioned. In particular, if
\[
\langle \tau_{u/v}\rangle_j=\sum_{k=0}^j C_k\,X[j,k], \qquad\qquad C_k\in\mathbb{C}[q^{\pm1},a^{\pm1}],
\]
then taking the closure of $\langle \tau_{u/v}\rangle_j$ means
\[
\text{Cl}(\langle\tau_{u/v}\rangle_j) = \sum_{k=0}^j C_k \text{Cl}(X[j,k]),
\]
where $\text{Cl}(X[j,k])\in \mathbb{C}(q)[a^{\pm1}]$ is computed by using the skein module structure. Thus, in light of Lemma \ref{ratlinkslem}, in order to compute the $j$-colored HOMFLY-PT polynomial of a rational link (one or two components), it suffices to know $\langle\tau_{u/v}\rangle_j$, $\text{Cl}(UP[j,k])$, and $\text{Cl}(OP[j,k])$. In \cite{SW21}, the authors give formulas for these, but it turns out to be computationally helpful to consider formulas such as $\text{Cl}(TX[j,k])$ instead, which we can do since $u/v\geq1$. If $\tau_{u/v}$ has $UP$ orientation, undoing the final top twist gives $\tau_{(u-v)/v}$, also with $UP$ orientation. On the other hand, if $\tau_{u/v}$ has $OP$ orientation, then undoing the final top twist gives $\tau_{(u-v)/v}$ with $RI$ orientation. Then, we have the following closure formulas, proved in \cite{SW21}, where $\sim$ denotes ``up to framing shift'' (the formulas in \cite{SW21} are written differently, but are related to these by a framing shift of $\pm 1$), which corresponds to multiplying the $j$-colored HOMFLY-PT polynomials by $(-q)^{\mp j}a^{\mp j}q^{\pm j^2}$.

\begin{lemmaN}
\label{WSclosureformulas}
We have the following:
    \begin{align}
    &\text{Cl}(TUP[j,k])\sim (-q)^kq^{2k^2}{j\brack k}_+\frac{(a^2q^{2-2j-2k};q^2)_k}{(q^2;q^2)_k}\\
    &\text{Cl}(TRI[j,k])\sim (-q)^{k-j}a^{2(k-j)}q^{2(k-j)^2}{j \brack k}_+ \frac{(a^2q^{2-2j-2(j-k)};q^2)_{j-k}}{(q^2;q^2)_{j-k}}.
    \end{align}    
\end{lemmaN}

Now, we have everything we need to define the colored HOMFLY-PT polynomials for rational links in this context.

\begin{definitionN}
    Given a rational tangle $\tau_{u/v}$ with $u/v>1$ and $UP$ or $OP$ orientation,the \textit{$j$-colored HOMFLY-PT polynomial} of its numerator closure, denoted $P_{u/v}^{\bigwedge^j}(q,a)$, is given by
    \[
    P_{u/v}^{\bigwedge^j}(q,a):=\text{Cl}(\langle T\tau_{(u-v)/v}\rangle_j).
    \]
\end{definitionN}

It should be emphasized that, although this is not the standard way of defining these invariants, this still agrees with the standard definition (up to some framing shift). We have written this as the definition to establish notation and because this is how we will compute them. Additionally, the $\bigwedge^j$ in $P_{u/v}^{\bigwedge^j}$ is used to denote that we are working with antisymmetric colored HOMFLY-PT polynomials, rather than the symmetric ones. 

Now, we need to introduce the notation we will use for the generating functions of these colored HOMFLY-PT invariants.

\begin{definitionN}
    Given a rational knot $K_{u/v}$, we will use $P(K_{u/v})$ to denote the generating function of its colored HOMFLY-PT polynomials:
    \[
    P(K_{u/v})=\sum_{j\geq 0} P_{u/v}^{\bigwedge^j}(q,a)x^j.
    \]
    Similarly, one defines $P(L_{u/v})$ for a 2-component rational link $L_{u/v}$.
\end{definitionN}

\subsection{The Links-Quivers Correspondence}
\label{LQCsec}

Given a link $L$, the Links-Quivers Correspondence predicts that one can find a symmetric quiver $Q_L$ such that the generating function of the colored HOMFLY-PT polynomials of $L$ (possibly rescaled by Pochhammer symbols) are related to a generating function associated with $Q_L$. As stated in the introduction, for any symmetric quiver $Q$, the generating function for its cohomological Hall algebra (the generating function we want), is of the form
\begin{equation}
    P_Q(x_1,...,x_m) = \sum_{\textbf{d}=(d_1,..,d_m)\in\mathbb{N}^m} (-q)^{-\langle \textbf{d},\textbf{d}\rangle_Q}\prod_{i=1}^m \frac{x_1^{d_1}...x_m^{d_m}}{(1-q^{-2})...(1-q^{-2d_i})},
\end{equation}
where $\langle \textbf{d},\textbf{e}\rangle_Q = \sum_{i,j}(\delta_{ij}- Q_{ij})d_ie_j$ and $x_1,...,x_m$ are formal variables.

In Section \ref{Cpolysec}, we gave one way to think about (antisymmetric) colored HOMFLY-PT polynomials for rational links, but as we alluded to, these invariants exist for arbitrary links. Let $P_L^{\bigwedge^j}(q,a)$ denote the $j$-colored HOMFLY-PT polynomial for a link $L$. If $L$ has $|L|$ components, recall the following generating function of (rescaled) colored HOMFLY-PT polynomials, now written in terms of Pochhammer symbols.

\begin{definitionN}
   We define $\widehat{P}(L)$ to be the generating function
    \begin{equation}
        \widehat{P}(L)=\sum_{j\geq0}(q^2;q^2)_j^{|L|-2}P_L^{\bigwedge^j}(q,a)x^j.
    \end{equation} 
\end{definitionN}

We can now state the Links-Quivers Correspondence Conjecture.

\begin{conjectureN}[Links-Quivers Correspondence \cite{K19}]
    Given a link $L$, there is a symmetric quiver $Q_L$ with adjacency matrix $Q$ such that
    \begin{equation}
    \widehat{P}(L)=P_{Q_L}(\overline{x})|_{x_i\mapsto (-1)^{Q_{ii}-S_i}q^{S_i-1}a^{A_i}x},
   \end{equation}
   Where $S$ and $A$ can be thought of as linear forms and $Q$ as a quadratic form on an appropriately defined free $\mathbb{Z}$-module of dimension given by the number of vertices in $Q$.
\end{conjectureN}

Importantly, the Links-Quivers Correspondence predicts that one can describe the infinitely many colored HOMFLY-PT polynomials for a link $L$ with a finite amount of data. In particular, one only needs two linear forms and a quadratic form. If one wants the Poincar\'e polynomials of the colored HOMFLY-PT chain complexes instead of just the colored HOMFLY-PT polynomials, the expectation is that one would only need one additional linear form (as predicted by the ``refined exponential growth'' conjecture of Gorsky, Gukov, and Sto\v si\'c \cite{GGS13,GS12}). 

In \cite{SW21}, Sto\v si\'c and Wedrich prove the Links-Quivers Correspondence for rational links by showing that $P(K_{u/v})$ and $P(L_{u/v})$ can be put into ``quiver form.'' . 

\begin{definitionN}
    For a rational knot $K_{u/v}$, $P(K_{u/v})$ is in \textit{quiver form} if it can be written as 
    \begin{equation}
    \label{quivformknot}
        P(K_{u/v})=\sum_{\textbf{d}=(d_1,...,d_u)\in\mathbb{N}^u}(-q)^{S\cdot \textbf{d}}a^{A\cdot \textbf{d}}q^{\textbf{d}\cdot Q\cdot \textbf{d}^t}{j \brack d_1,...,d_u}x^j,
    \end{equation}
    where $S$ and $A$ may be thought of as linear forms and $Q$ as a quadratic form on a free $\mathbb{Z}$-module of dimension $u$. Similarly, for a 2-component rational link $L_{u/v}$, $P(L_{u/v})$ is in \textit{quiver form} if it can be written as 
    \begin{equation}
    \label{quivformlink}
        P(L_{u/v})=\sum_{\textbf{d}=(d_1,...,d_{2u})\in \mathbb{N}^{2u}}(-q)^{S\cdot \textbf{d}}a^{A\cdot \textbf{d}}q^{\textbf{d}\cdot Q\cdot \textbf{d}^t}\frac{x^j}{\prod_{i=1}^{2u}(q^2;q^2)_{d_i}},
    \end{equation}
     where $S$ and $A$ may be thought of as linear forms and $Q$ as a quadratic form on a free $\mathbb{Z}$-module of dimension $2u$.
\end{definitionN}

Thus, one may state the main theorems of \cite{SW21} as given below.

\begin{theoremN}[\cite{SW21}]
    For $K_{u/v}$ a rational knot and $L_{u/v}$ a 2-component rational link, $P(K_{u/v})$ and $P(L_{u/v})$ may be expressed in quiver form.
\end{theoremN}

In this paper, we will provide a geometric way to compute $S,A,$ and $Q$. One fact that we will need to prove our theorems is that the quadratic form $Q$ used in the quiver form is highly non-unique.

\begin{definitionN}
    Given a symmetric matrix $Q$ satisfying the role of the quadratic form in Equation (\ref{quivformknot}) or (\ref{quivformlink}), we say $Q\sim Q'$, for $Q'$ a symmetric matrix of the same size, if replacing $Q$ with $Q'$ gives the same generating function. In terms of the Links-Quivers Correspondence, this means
    \begin{equation}
     P_Q(x_1,...,x_u)\bigg|_{x_i\mapsto (-1)^{Q_{i,i}-S_i}q^{S_i-1}a^{A_i}x} = P_{Q'}(x_1,...,x_u)\bigg|_{x_i\mapsto (-1)^{Q'_{i,i}-S_i}q^{S_i-1}a^{A_i}x}.   
    \end{equation}
    Importantly, $S$ and $A$ are assumed to be fixed.
\end{definitionN}

In the proofs found later in this paper, we will need to permute certain entries in our quadratic forms, viewed as symmetric matrices.
The following theorem by Jankowski, Kucharski, Larragu\'ivel, Noshchenko, and Su\l{}kowski in \cite{JKLN21} provides some conditions for permuting entries of $Q$ in such a way that $Q'$, the result of these permutations, satisfies $Q\sim Q'$. Also, we will use the notation $\Lambda_i=(-1)^{Q_{i,i}-S_i}q^{S_i-1}a^{A_i}$ and $Q_0=\{1,...,|Q|\}$, where $|Q|$ denotes the number of rows (or columns) in $Q$.

\begin{theoremN}
\label{permthm}
    (Theorem 6, \cite{JKLN21}) If $Q$ and $Q'$ are two symmetric matrices of the same size along with the data $\Lambda_i=\Lambda_i'$ for all $i\in Q_0$, then if $Q$ and $Q'$ are related by a disjoint sequence of transpositions of non-diagonal elements
    \[
    Q_{ab}\leftrightarrow Q_{cd}, \qquad \qquad Q_{ba}\leftrightarrow Q_{dc}
    \]
    for pairwise different $a,b,c,d\in Q_0$ with
    \[
    \Lambda_a\Lambda_b=\Lambda_c\Lambda_d
    \]
    and 
    \[
    Q_{ab}=Q_{cd}-1, \qquad \qquad Q_{ai}+Q_{bi}=Q_{ci}+Q_{di}-\delta_{ci}-\delta_{di}, \qquad \forall i\in Q_0,
    \]
    or
    \[
    Q_{cd}=Q_{ab}-1, \qquad \qquad Q_{ci}+Q_{di}=Q_{ai}+Q_{bi}-\delta_{ai}-\delta_{bi}, \qquad \forall i\in Q_0,
    \]
    then $Q\sim Q'$.
\end{theoremN}

Thus, Theorem \ref{permthm} states some sufficient conditions for permuting pairs of off-diagonal entries in $Q$ while still yielding the same generating function.

\section{Links-Quivers for Rational Tangles}
\label{LQCTsec}

This paper builds on the work of \cite{JHpI26}, so now we review the constructions and results that will be needed for the remainder of the paper. Refer to \cite{JHpI26} for further details.

\subsection{Construction of Lagrangians}
\label{TangleLsection}

First, we describe the geometric construction used in \cite{JHpI26}. In order to compute the linear and quadratic forms for tangles, we studied loops defined along Lagrangians in the first and second configuration spaces of the 3-punctured plane, $M=\mathbb{C}\setminus\{X^+,X^-,Y\}$.

The first type of Lagrangian is an arc between the punctures $X^+$ and $X^-$. For the rational tangle $\tau_{u/v}$ we denoted this Lagrangian by $\alpha_{u/v}$. In \cite{JHpI26}, we gave a few different ways to construct $\alpha_{u/v}$, and here we will recall two of them. The first method is direct, whereas the second is inductive. Naturally, it is convenient to be able to construct $\alpha_{u/v}$ directly for any $u/v$, but it is also beneficial to understand the inductive method, as an important part of the proofs in this paper will involve undoing the final top twist and seeing how this affects the Lagrangians. 

Now, we present the first method in the form of a Lemma. This is essentially a reformulation of the beginning of Section 4.1 of \cite{W16} to fit our conventions.

\begin{lemmaN}
\label{bridgerecipe}
    Up to isotopy, $\alpha_{u/v}$ may be constructed as follows:
    Draw the intervals $[-2,-1]$ and $[1,2]$ on the real axis and partition them into $u$ parts of equal size. If $u\geq v$, starting at the points $\pm1$, label the ends of the partitions on both line segments by $0,1,...,2u-1\in \mathbb{Z}/(2u)\mathbb{Z}$ in a clockwise manner, as in Figure \ref{4_3_tangle} on the left. Starting at the 0-labeled point $\text{par}(u)\in\{\pm1\}$, the parity of $u$, draw the arc $\alpha_{u/v}$ by proceeding to the point labeled $v$ on the opposite side, then going to $2v$ on the same side you started (mod $2v$), and proceed until you hit $\{-2,-1,1,2\}$ again. The three punctures in the plane are the points at $x=\pm1, 2$ on the real axis. If $u<v$, reverse the roles of $u$ and $v$, label the partition points in a counter-clockwise manner, and proceed in a similar manner as in the previous case, but start at the 0-labeled point $-\text{par}(v)\in\{\pm1\}$. (Equivalently, reflect the picture for $\tau_{v/u}$ across the imaginary axis.) In this case, the three punctures in the plane are at $x=\pm1,-2$ on the real axis.
\end{lemmaN}

Now that we know how to construct $\alpha_{u/v}$, we need to define the other types of Lagrangians. As Lemma \ref{bridgerecipe} says, we can think of the three punctures $X^+,X^-,$ and $Y$ as points on the real axis of $\mathbb{C}$, so they can be ordered from left to right. Let $l_A$ be a vertical Lagrangian (parallel to the imaginary axis) running between the left and middle punctures, and let $l_I$ be a vertical Lagrangian running between the middle and right punctures, such that $l_A$ and $l_I$ have the minimal number of intersections with $\alpha_{u/v}$. Then, one can check
\[
|\alpha_{u/v}\cap l_A|=u
\qquad\text{and}\qquad
|\alpha_{u/v}\cap l_I|=v.
\]
Next, we provide some names for these Lagrangians and their intersections with $\alpha_{u/v}.$

\begin{definitionN}
    We will call $l_A$ the \textit{active axis} and $l_I$ the \textit{inactive axis}. The points in $\alpha_{u/v}\cap l_A$ are called the \textit{active intersection points} and the points in $\alpha_{u/v}\cap l_I$ are called the \textit{inactive intersection points}.
\end{definitionN}

In \cite{JHpI26}, we needed to use the picture in the 3-punctured plane consisting of these Lagrangians, so we introduced the following notation.

\begin{definitionN}
    Given a rational tangle $\tau_{u/v}$, let $\mathcal{D}(\tau_{u/v})$ be the picture in $M=\mathbb{C}\setminus\{X^+,X^-,Y\}$ consisting of the Lagrangians $\alpha_{u/v},l_A,$ and $l_I$.
\end{definitionN}

See the figure below for an example. In this paper, we will use something similar for knots and links, but the Lagrangians will be different. 

\begin{figure}
   \begin{tikzpicture}
       \node at (0,0) {\includegraphics[height=3cm]{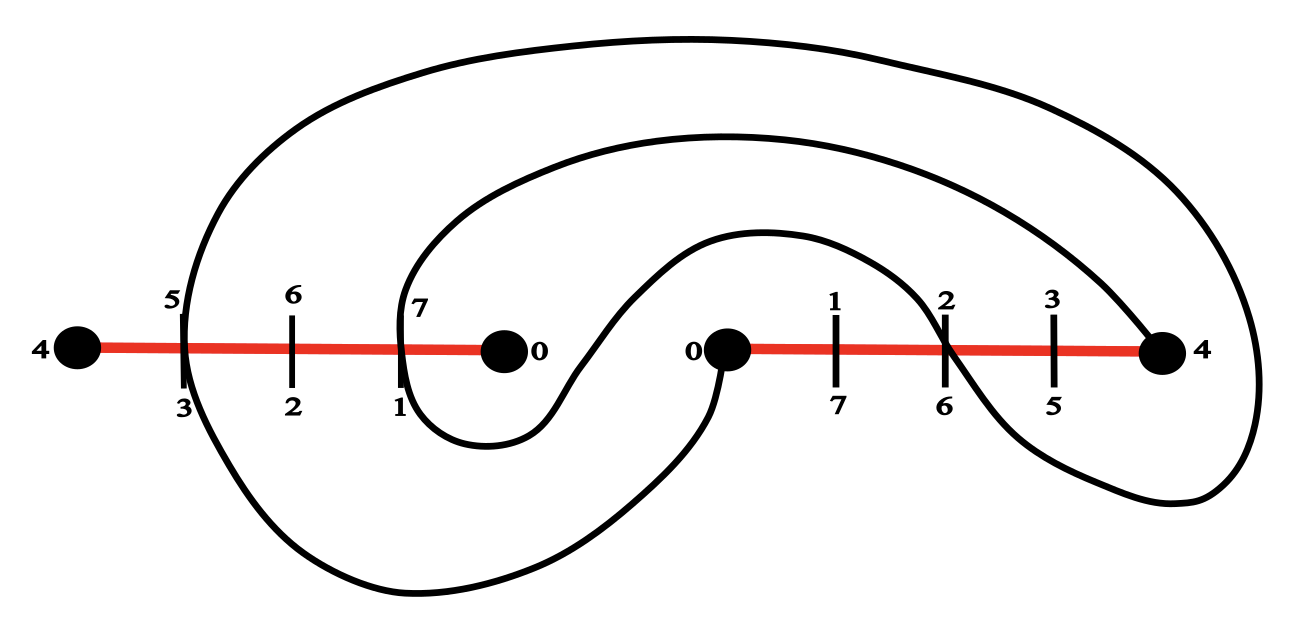}};
       \node at (3.8,0) {$\longrightarrow$};
       \node at (7,0) {\includegraphics[height=3cm]{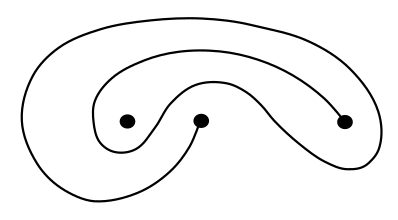}};
   \end{tikzpicture}
\caption{An application of Lemma \ref{bridgerecipe} to $\tau_{4/3}$}
\label{4_3_tangle}
\end{figure}
\[
 \begin{tikzpicture}
        \node at (0,0) {$\mathcal{D}(\tau_{4/3})=$};
        \node at (3,0) {\includegraphics[height=3cm]{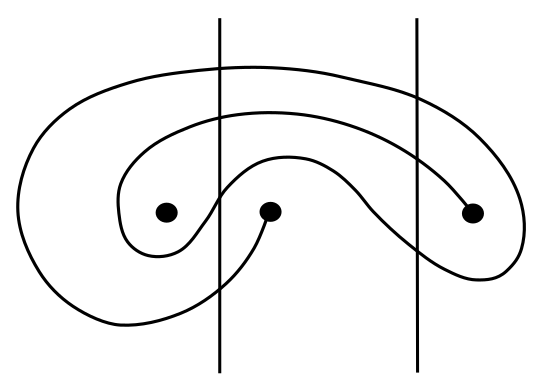}};
\end{tikzpicture}
\]

Next, we have the inductive procedure to obtain $\mathcal{D}(\tau_{u/v})$, also taken from \cite{W16}. In the figures below, the blue point is $Y$ and the red points are $X^+$ and $X^-$, whose exact positions are determined by the particular sequence of twists. Importantly, though, the point in the middle for $\tau_{0/1}$ is $X^-$ and the red point on the right is $X^+$.

\begin{lemmaN}[Lemma 4.1, \cite{W16}]
\label{twistfx}
$\mathcal{D}(\tau_{u/v})$ can be built inductively using $l_A$ and $l_I$ as follows. Starting with the diagram corresponding to $\tau_{0/1}$, shown below, apply top twists by bending $l_A$ towards $l_I$ and straightening it out as shown below on the right. Right twists are performed in a similar way by bending $l_I$ towards $l_A$ and then straightening.
\[
\begin{tikzpicture}
\node at (-5,0) {$\mathcal{D}(\tau_{0/1})=$};
\node at (-3,0) {\includegraphics[height=2cm]{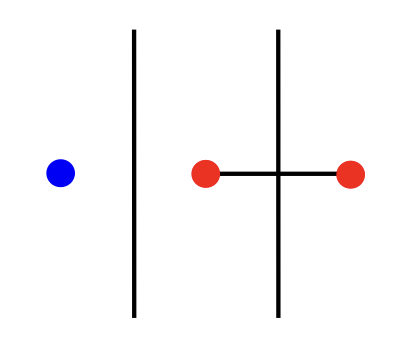}};
\node at (1,0) {$\mathcal{D}(T\tau_{1/1})=\mathcal{D}(\tau_{2/1})\colon$};
\node at (4,0) {\includegraphics[height=2cm]{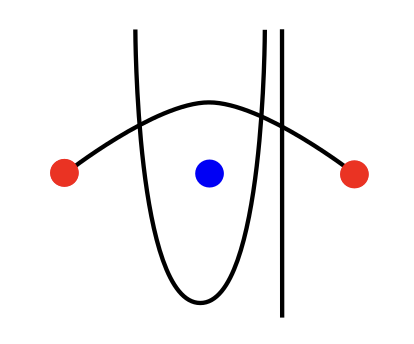}};
\node at (5.5,0) {$\longrightarrow$};
\node at (7,0) {\includegraphics[height=2cm]{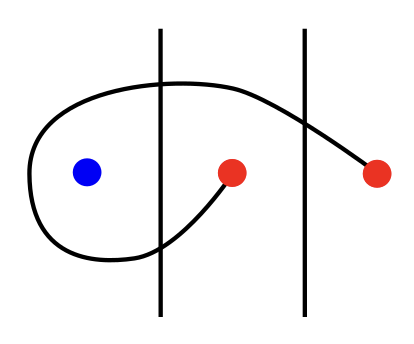}};
\end{tikzpicture}
\]
\end{lemmaN}

Given that $\tau_{4/3}=TR^2T\tau_{0/1}$, it is not difficult to check that Lemma \ref{twistfx} gives the same $\mathcal{D}(\tau_{4/3})$ that we get from Lemma \ref{bridgerecipe} after adding $l_A$ and $l_I$ (up to isotopy).

The computations in \cite{JHpI26}, as well as the ones in this paper, involve keeping track of the three punctures. Continuing with the $\tau_{4/3}$ example, the three punctures in $\mathcal{D}(\tau_{4/3})$ are $Y,X^-,$ and $X^+$ from left to right. We encode this information with the notation $Y|X^-|X^+$. In addition to knowing the order of the named punctures, we want to know the orientation of $\tau_{u/v}$. We will call the pair of data consisting of the orientation of $\tau_{u/v}$ and the ordering of the punctures the \textit{state} of $\tau_{u/v}$, and we write it as $(X, A|B|C)$, where $X$ is the orientation and $A,B,$ and $C$ are the puncture names. For example, the state for $\tau_{4/3}$ is $(UP, Y|X^-|X^+)$, which is the same state as $\tau_{0/1}$. The following lemma shows how the top and right twists affect states.

\begin{lemmaN}[Corollary 4.3, \cite{W16}]
\label{twistdiagram}
    The following diagram shows the effects of top and right twists on the orientations of the arcs and the configurations of the points $X^-,X^+$, and $Y$ (read left to right). The underlined state represents the state of the trivial tangle $\tau_{0/1}$.
  
\[\begin{tikzcd}[sep=small]
   & \underline{(UP,\,\,\,\, Y\vert X^-\vert X^+)} \arrow[rr,"T"]\arrow[ld] && (UP,\,\,\,\, X^-\vert Y\vert X^+) \arrow[ll]\arrow[rd,"R"] &\\
   (OP,\,\,\,\, Y\vert X^+\vert X^-) \arrow[ru,"R"]\arrow[rd]&&&& (OP,\,\,\,\,X^-\vert X^+\vert Y) \arrow[lu]\arrow[ld,"T"]\\
    & (RI,\,\,\,\, X^+\vert Y\vert X^-) \arrow[lu,"T"]\arrow[rr] && (RI,\,\,\,\, X^+\vert X^-\vert Y) \arrow[ll,"T"]\arrow[ru] &
\end{tikzcd}\]
\end{lemmaN}

Next, we need to introduce notation for the Lagrangian intersections.

\begin{definitionN}
    Given a rational tangle $\tau_{u/v}$, let $\xi_1,...,\xi_u$ denote the active intersection points and $\xi_{u+1},...,\xi_{u+v}$ denote the inactive intersection points in $\mathcal{D}(\tau_{u/v})$. We use $\mathcal{G}_{u/v}^T$ to denote the set of intersection points, i.e. $\mathcal{G}_{u/v}^T=\{\xi_1,...,\xi_{u+v}\}$.
\end{definitionN}

Before proceeding to define the loops and homomorphisms used in \cite{JHpI26}, we provide an important ordering of the $\xi_i$'s.

\begin{definitionN}
\label{precorder}
    Let $\textbf{r}:[0,1]\to\alpha$ be the constant-speed parametrization of $\alpha$ with no critical points, such that $\mathbf{r}(0)=X^-$ and $\mathbf{r}(1)=X^+$. Given two intersection points $\xi_i$ and $\xi_j$, we say $\xi_i\prec \xi_j$ if $\xi_i=\mathbf{r}(t_0)$ and $\xi_j=\mathbf{r}(t_1)$ with $t_0<t_1$.
\end{definitionN}

This ordering on the $\xi_i$'s will be used later to help define orderings of the Lagrangian intersections for rational links.

\subsection{Loops and Homomorphisms}
\label{loophomsec}

There are two different types of loops in \cite{JHpI26}: one is defined in $M$ and the other in $\text{Conf}^2(M)$. However, both types are indexed by pairs of intersection points in $\mathcal{G}_{u/v}^T$. First, we will describe the loops in $M$. In the following discussion, it should be noted that when we refer to a path being unique, we mean that it is unique up to homotopy.

Given two intersections $\xi_i$ and $\xi_j$, there is a unique path from $\xi_i$ to $\xi_j$ along $\alpha$. Denote this path by $\eta_{i,j}$. If $\xi_i$ and $\xi_j$ are both active or inactive intersections, then there is a unique path from $\xi_j$ back to $\xi_i$ along the appropriate (in)active axis. Otherwise, $\xi_i$ and $\xi_j$ lie on different vertical Lagrangians. If $\xi_i\in l_A$ and $\xi_j\in l_I$, then we can determine a unique path from $\xi_j$ back to $\xi_i$ by first traveling along $l_I$ to the top intersection point, sliding along $\alpha$ to the top intersection point on $l_A$, and then traveling back down to $\xi_i$ along $l_A$. We can define the unique path from $\xi_j$ to $\xi_i$ in an analogous way if $\xi_j\in l_A$ and $\xi_i\in l_I$. Let $\bar\eta_{j,i}$ be the path from $\xi_j$ to $\xi_i$ just described, depending on whether $\xi_i$ and $\xi_j$ are each active or inactive.

\begin{definitionN}
    Given Lagrangian intersection points $\xi_i$ and $\xi_j$ in $\mathcal{D}(\tau_{u/v}),$ let $\gamma_{i,j}^T:[0,1]\to M$ be the loop $\gamma_{i,j}^T=\eta_{i,j}\cdot \bar\eta_{j,i}$.
\end{definitionN}

\begin{figure}
    \centering
    \includegraphics[height=4cm]{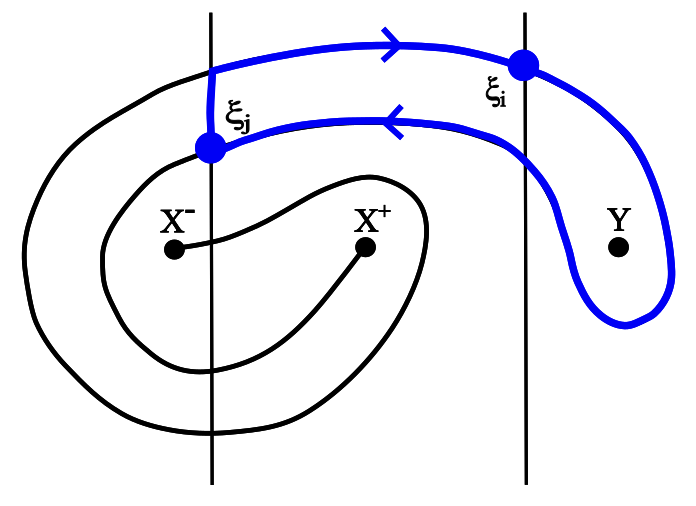}
    \caption{An example loop $\gamma_{i,j}^T$ for $\tau_{5/2}$}
    \label{tangleloopex}
\end{figure}

In contrast to \cite{JHpI26}, we include the $T$ superscript to indicate that this is a loop defined for a rational tangle. We will use $K$ and $L$ superscripts for the appropriate loops for knots and 2-component links later.

There is a maximum element in $\mathcal{G}_{u/v}^T$ with respect to $\prec$, given by the intersection point immediately before $X^+$. Many of the loops used in \cite{JHpI26} are of the form $\gamma_{i,j}^T$ where $\xi_j$ is this particular intersection, so we use the following notation.

\begin{definitionN}
    Given a rational tangle $\tau_{u/v}$, let $\xi_\omega$ denote the maximum element of $\mathcal{G}_{u/v}^T$ with respect to $\prec.$ Now, define $\gamma_i^T:=\gamma_{i,\omega}^T$.
\end{definitionN}

Next, we must consider the loops in $\text{Conf}^2(M)$. Recall that a loop $\tilde \gamma:[0,1]\to\text{Conf}^2(M)$ is given by a pair $(\sigma^1,\sigma^2)$, where $\sigma_1$ and $\sigma_2$ are both loops or paths that swap starting/ending points in $M$, such that $\sigma^1(t)$ and $\sigma^2(t)$ are distinct for all $t\in[0,1]$.

\begin{figure}
\begin{tikzpicture}
    \node at (0,0) {\includegraphics[height=4cm]{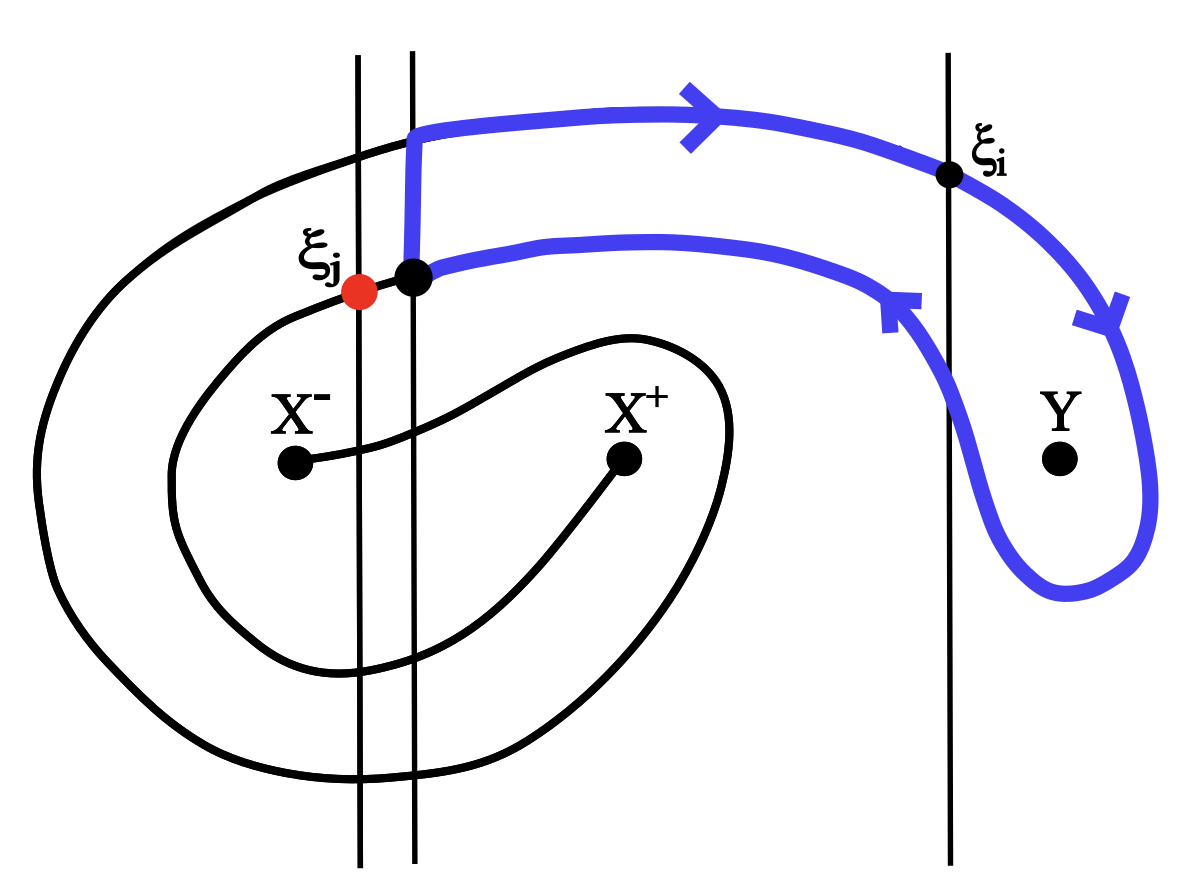}};
    \node at (8,0) {\includegraphics[height=4cm]{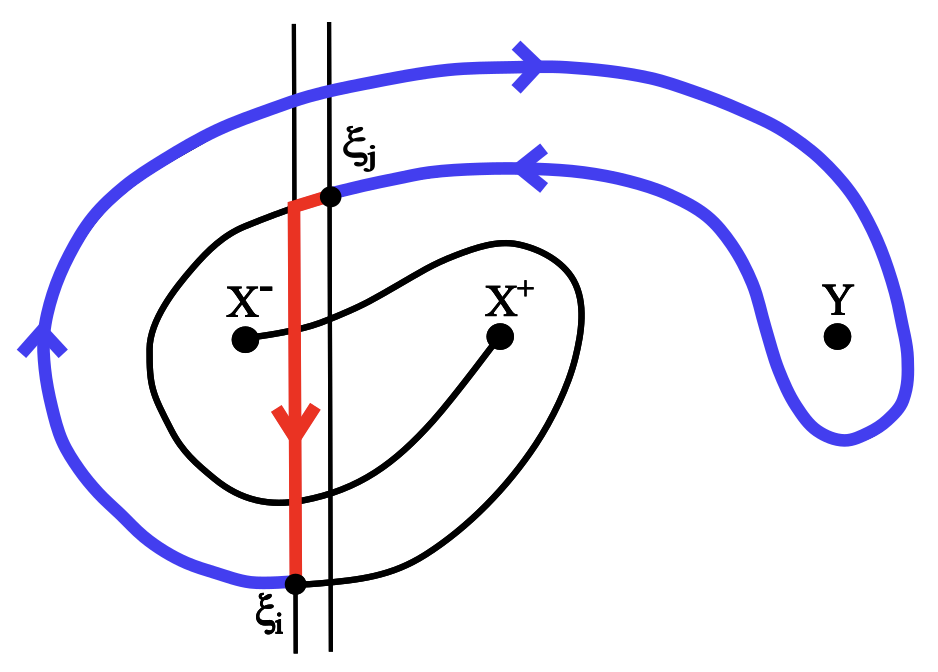}};
\end{tikzpicture}
    \caption{Two example loops $\tilde\gamma_{i,j}^T$ for $\tau_{5/2}$}
    \label{T52loops}
\end{figure}

Although the loops in $\text{Conf}^2(M)$ are similar to the ones in $M$, it takes considerably more work to define them carefully, so we refer the reader to the end of Section 3.1 in \cite{JHpI26} for a careful treatment of them. Here, we will only say a couple of facts about these loops. They are indexed by (ordered) pairs of Lagrangian intersections $\xi_i$ and $\xi_j$ such that $\tilde \gamma_{i,j}^T=(\sigma_{i,j}^1,\sigma_{i,j}^2):[0,1]\to\text{Conf}^2(M)$ associated with these intersection points satisfies $\gamma_{i,j}^T=\sigma_{i,j}^1+\sigma_{i,j}^2$ in $H_1(M)$. The basepoint for $\tilde \gamma_{i,j}^T$ is determined by taking parallel copies of the vertical Lagrangians and placing whichever of $\xi_i$ or $\xi_j$ is below the other (if both are active or inactive intersections) or whichever is active (if one is active and the other inactive) on the left-most vertical. Two examples are shown in Figure \ref{T52loops}, but a detailed explanation can be found in \cite{JHpI26}. 

We must also define the homomorphisms used in \cite{JHpI26}. These are the same homomorphisms we will need later in this paper. Let $V\in\{X^+,X^-,Y\}$. The first type of homomorphism we need is defined in terms of $\Psi_V^1:\pi_1(M)\to\mathbb{Z}$, which is given by the composition
\[
\Psi_V^1:\pi_1(M)\xrightarrow{\iota_*} \pi_1(\mathbb{C}\setminus V)\to \pi_1(\text{Conf}^2(\mathbb{C}))\cong \text{Br}_2\xrightarrow[\sim]{\text{ab}} \mathbb{Z},
\]
where the second map takes a loop $\gamma\subset \mathbb{C}\setminus V$ to $(\gamma, c_V)\subset \text{Conf}^2(\mathbb{C})$, where $c_V$ is the constant loop at $V$, and the other maps are the obvious ones (aside from the choice of the particular isomorphism $\pi_1(\text{Conf}^2(\mathbb{C}))\cong \text{Br}_2$). Then, for $W\subseteq \{X^+,X^-,Y\}$, we define $\Psi_W:\pi_1(M)\to\mathbb{Z}$ as
\[
\Psi_W:=\frac{1}{2}\sum_{V\in W}\Psi_V^1.
\]
It can easily be checked that dividing by $2$ still results in an integer. The other type of homomorphism is $\Phi:\pi_1(\text{Conf}^2(M))\to \mathbb{Z}$, which is given by the composition
\[
\Phi:\pi_1(\text{Conf}^2(M))\xrightarrow{\iota_*}\pi_1(\text{Conf}^2(\mathbb{C}))\cong \text{Br}_2\xrightarrow[\sim]{\text{ab}} \mathbb{Z}.
\]

For all the homomorphisms just described, the isomorphism $\pi_1(\text{Conf}^2(\mathbb{C}))\cong \text{Br}_2$ is defined by choosing a perspective for viewing the loops in $\text{Conf}^2(\mathbb{C})$. The convention in \cite{JHpI26} and in this paper is that the loop should be viewed by tilting the top of the page towards you such that the printed side of the page is facing down. In other words, you are viewing the loop from above, and when the loop moves to the left or right on the page, you see it move in the same direction from this perspective. See the end of Section 2.4 in \cite{JHpI26} for an example.

\subsection{Geometric Formulas for Rational Tangles}

Before we can state the main results we will need from \cite{JHpI26}, we need to recall a few more notational choices.

\begin{definitionN}
    Given a rational tangle $\tau_{u/v}$, let $P(\tau_{u/v})$ denote the generating function of the colored HOMFLY-PT skein module evaluations of $\tau_{u/v}$, i.e.
    \[
    P(\tau_{u/v})=\sum_{j\geq 0}\langle\tau_{u/v}\rangle_j.
    \]
\end{definitionN}

Additionally, the formulas below involve certain $\delta$ terms. For example, $\delta_{i,A}$ should be read as
\[
\delta_{i,A}=
\begin{cases}
    1, \qquad \xi_i \,\text{is an active intersection point}\\
    0, \qquad \text{otherwise}
\end{cases}
\]
and $\delta_{i,I}$ should be read similarly, but with ``inactive'' instead of ``active'' in the definition. Additionally, we used $Z$ to denote the middle of the three punctures, so $\delta_{Z,X^+}$ should be read as the Kronecker delta that gives $1$ if $Z=X^+$ and $0$ otherwise.

Now, we can state the first main result that we will need from \cite{JHpI26}.

\begin{propositionN}
\label{tangleHOMFLYpoly}
    We can express $P(\tau_{u/v})$ as
    \begin{multline}
    \label{HOMFLYthmeq}
        \sum_{\bf{d}=(d_1,...,d_{u+v})\in \mathbb{N}^{u+v}}(-q)^{S\cdot\bf{d}}a^{A\cdot \bf{d}}q^{\bf{d}\cdot Q\cdot \bf{d}^t}{d_1+...+d_u\brack d_1,...,d_u}{d_{u+1}+...+d_{u+v}\brack d_{u+1},...,d_{u+v}}\\ \times X[d_1+...+d_{u+v},d_1+...+d_u],
    \end{multline}
    where $S,A,$ and $Q$ are defined on the free $\mathbb{Z}$-module $\mathbb{Z}\mathcal{G}_{u/v}^T$ with basis given by $\mathcal{G}_{u/v}^T$ and are computed by the following formulas:
    \newline
    \[
    \begin{cases}
        S_i=\Psi_{\{X^\pm,Y\}}([\gamma_i])+\delta_{i,A}\delta_{\omega,I}-\delta_{i,I}\delta_{\omega,A}\\
        A_i=2\Psi_{X^+}([\gamma_i])+\delta_{Z,X^+}(\delta_{i,A}\delta_{\omega,I}-\delta_{i,I}\delta_{\omega,A})\\
        \begin{cases}
            Q_{ii}=(\Psi_{\{X^-,Y\}}-3\Psi_{X^+})([\gamma_i])+2\delta_{Z,X^+}(\delta_{i,I}\delta_{\omega,A}-\delta_{i,A}\delta_{\omega,I})\\
            Q_{ij}=Q_{ii}+(\Phi-2\Psi_{X^+})([\gamma_{j,i}])+\delta_{Z,X^+}(\delta_{i,A}\delta_{j,I}-\delta_{j,A}\delta_{i,I}).
        \end{cases}
    \end{cases}
    \]
\end{propositionN}

The other result we need from \cite{JHpI26} involves reducing the number of indices in the sum. In particular, we can partition $\mathcal{G}_{u/v}^T$ as $\mathcal{G}_{u/v}^T=\mathfrak{X}\sqcup \mathfrak{Y}\sqcup \mathfrak{Z}$ in a particular way such that we only need to compute $S,A,$ and $Q$ on $\mathfrak{Y}\sqcup\mathfrak{Z}$. We can think of this as restricting $S,A,$ and $Q$ to free $\mathbb{Z}$-submodule with basis given by $\mathfrak{Y}\sqcup\mathfrak{Z}$. Once we do this, we need to introduce a new linear form on this smaller $\mathbb{Z}$-module, denoted $K$, such that
\begin{equation}
\label{Kform}
    K_i=K(\xi_i)=\begin{cases}
    1, \qquad\qquad \xi_i\in\mathfrak{Y}\\
    0,\qquad\qquad \xi_i\in\mathfrak{Z}.
\end{cases}
\end{equation}

We partition the intersection points in terms of pairs connected by arcs that wrap around $Y$ as shown in Figure \ref{Ydisk}. If we consider all pairs of intersections connected in this way, we set $\mathfrak{X}$ to be the types of points labeled $\xi_i$ and $\mathfrak{Y}$ to be the types of points labeled $\xi_j$ in Figure \ref{Ydisk}. The set $\mathfrak{Z}$ consists of all the remaining intersection points. See \cite{JHpI26} for examples. Given this setup, we have the following proposition.

\begin{figure}
   
 \raisebox{0pt}{\includegraphics[height=3.5cm, angle=0]{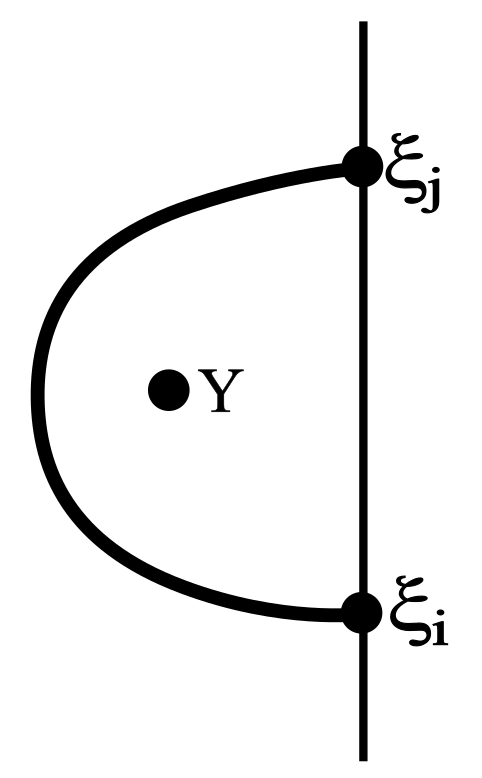}}\qquad \qquad \raisebox{0pt}{\includegraphics[height=3.5cm, angle=0]{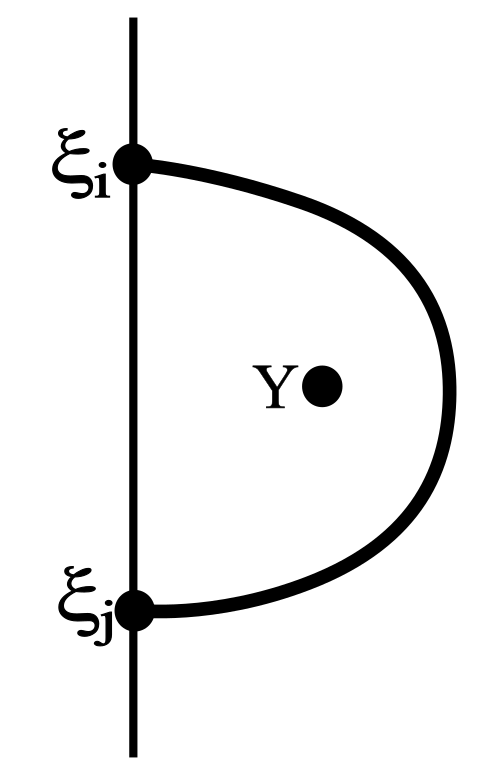}}
\caption{The two cases where consecutive $\xi_i,\xi_j\in \alpha\cap l$ (for $l\in\{l_
A,l_I\}$) wrap around $Y$.}
\label{Ydisk}
\end{figure}

\begin{propositionN}
\label{indpartprop}
    Given $\tau_{u/v}$ and $\mathcal{G}_{u/v}^T=\mathfrak{X}\sqcup\mathfrak{Y}\sqcup\mathfrak{Z}$ as described above, $P(\tau_{u/v})$ can be expressed as
    \begin{multline*}
        P(\tau_{u/v})=\sum_{\textbf{d}\in \mathbb{N}^{m+n}}(-q)^{S\cdot\textbf{d}}q^{\textbf{d}\cdot Q\cdot \textbf{d}^t}a^{A\cdot \textbf{d}} (q^2;q^2)_{K\cdot \textbf{d}}{d_1+...+d_m\brack d_1,...,d_m}{d_{m+1}+...+d_{m+n}\brack d_{m+1},...,d_{m+n}}\\ \times X[d_1+...+d_{m+n},d_1+...+d_m],
    \end{multline*}
    where $m\leq u$, $n\leq v$, $m+n=|\mathfrak{Y}\sqcup\mathfrak{Z}|$, $K$ is given by Equation (\ref{Kform}), and $S,A,$ and $Q$ are defined on the free $\mathbb{Z}$ module with basis given by $\mathfrak{Y}\sqcup \mathfrak{Z}$ 
 and are computed by the same formulas as in Proposition \ref{tangleHOMFLYpoly}, but restricted to $\mathfrak{Y}\sqcup\mathfrak{Z}$.
\end{propositionN}

\section{Rational Knots}
\label{knotssec}

\subsection{Setting Up the Theorem for Knots}
\label{knotssetup}

We will now provide a method for computing the quiver for rational knots in a geometric way. The entirety of this section will, thus, be dedicated to proving Theorem \ref{knotthmintro} from Section \ref{introsec}. In the next section, we will consider Theorem \ref{linkthmintro} for two-component links; the geometric setup will be different there even though the formulas are essentially the same.

Recall from Section \ref{bgrd} that all rational knots may be obtained by taking the numerator closure of a rational tangle $\tau_{u/v}$ with $u$ odd and orientation $UP$ or $OP$. As Wedrich and Sto\v si\'c observe in \cite{SW21}, this can be realized by applying a top twist $T$ followed by the numerator closure of a rational tangle $\tau_{(u-v)/v}$ with $u-v$ even, $v$ odd, and orientation $UP$, or we can have $u-v$ odd, $v$ even, and orientation $RI$. This will factor strongly into the proof of this section's theorem. However, before we can state the theorem, we need to make our geometric set-up more precise.


Given a rational knot $K_{u/v}=\text{Cl}(\tau_{u/v})$, such as $K_{5/2}$ on the left in Figure \ref{fig8knotfig}, we can start with the arc Lagrangian $\alpha_{u/v}$ for $\tau_{u/v}$. In dealing with the knot rather than the tangle, we will want to modify the Lagrangians being used, but $\alpha_{u/v}$ will serve as a ``skeleton'' for the picture we want. In particular, one of the Lagrangians may be thought of as a band wrapped tightly around the arc $\alpha_{u/v}$. We will denote it by $\overline{\alpha_{u/v}}$ and draw it in blue for the remainder of the paper. The other sort of Lagrangian that we want will serve as a replacement for the vertical Lagrangians $l_A$ and $l_I$ used in $\mathcal{D}(\tau_{u/v})$. This new Lagrangian will be given by a horizontal arc between $Y$ and $X^+$ (think of these as strings with ends tied at the two points). We will denote this Lagrangian by $\beta$, and it will be drawn in red. Of course, these ``horizontal'' arcs and the band around $\alpha$ can be wiggled around via isotopies, and we will do precisely this throughout the proof of the theorem, but we can think of it this way (shown in Figure \ref{fig8knotfig}) as the ``standard'' picture.

\begin{definitionN}
    Given a rational knot $K_{u/v}$, let $\mathcal{D}(K_{u/v})$ be the picture in the 3-punctured plane $M$ consisting of $\alpha_{u/v}, \overline{\alpha_{u/v}},$ and $\beta$.
\end{definitionN}

Although $\alpha_{u/v}$ is not necessary here, we include it because it makes the rest of $\mathcal{D}(K_{u/v})$ easier to draw. Importantly, though, the linear and quadratic forms in the quiver form of the generating function $P(K_{u/v})$ that we want will be defined on a free $\mathbb{Z}$-module with preferred basis given by the Lagrangian intersections $\overline{\alpha_{u/v}}\cap \beta$ in $\mathcal{D}(K_{u/v})$. For the numerator closure of $\tau_{u/v}$, observe that $\overline{\alpha_{u/v}}\cap \beta$ consists of $u$ points.  

\begin{definitionN}
    Let $\kappa_1,...,\kappa_u$ denote the $u$ intersection points in $\overline{\alpha_{u/v}}\cap \beta$. Furthermore, let $\mathcal{G}_{u/v}^K$ be the set of these points in $\overline{\alpha_{u/v}}\cap\beta$, i.e. $\mathcal{G}_{u/v}^K=\{\kappa_1,...,\kappa_u\}$, and we take $\mathbb{Z}\mathcal{G}_{u/v}^K$ to be the free $\mathbb{Z}$-module with basis $\mathcal{G}_{u/v}^K$. 
\end{definitionN}

\begin{figure}
\raisebox{0pt}{\includegraphics[height=4cm, angle=0]{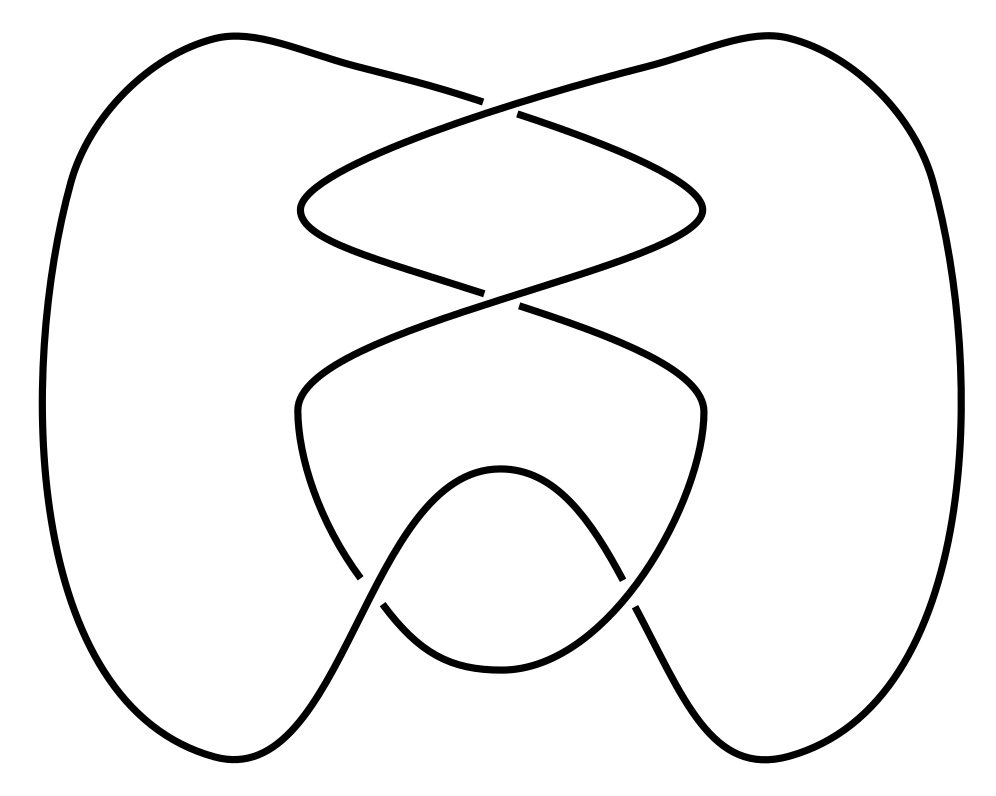}}\qquad \qquad \raisebox{0pt}{\includegraphics[height=5cm, angle=0]{fig8geopic.png}}
\caption{Left: The figure-8 knot drawn as $K_{5/2}=\text{Cl}(\tau_{5/2})$. Right: $\mathcal{D}(K_{5/2})$.}
\label{fig8knotfig}  
\end{figure}

Just as we did for rational tangles, we will define our loops used in the geometric formulas to be in $M=\mathbb{C}\setminus\{X^+,X^-,Y\}$ or $\text{Conf}^2(M)$, where the latter is only necessary for computing the off-diagonal entries of $Q$. Thus, if we write $\beta_1,...,\beta_j$ for parallel copies of $\beta$ in $M$, then we only need to consider intersections between
\[
\mathbb{L}_{u/v}^j:=\text{Sym}^j(\overline{\alpha_{u/v}})\setminus \Delta \subset \text{Conf}^j(M)
\]
and 
\[
\mathbb{L}^j:=\pi(\beta_1\times ...\times \beta_j) \subset \text{Conf}^j(M)
\]
for $j\in \{1,2\}$, where $\pi:M^j\setminus \Delta \to \text{Conf}^j(M)$ is the obvious map. Our loops will, once again, be defined in terms of intersections of these Lagrangians, and loops in $\text{Conf}^2(M)$ will be drawn in the plane by taking parallel copies $\beta_1$ and $\beta_2$ of $\beta$. Special care will be needed to keep them distinguished, so we will call the top one $\beta_1$ and the bottom one $\beta_2$.

\begin{definitionN}
    Let $\mathcal{D}^2(K_{u/v})$ be the configuration of Lagrangians in $\text{Conf}^2(M)$ consisting of  $\mathbb{L}_{u/v}^2$ and $\mathbb{L}^2$. This is modeled in $M$ by taking $\overline{\alpha_{u/v}}, \beta_1$, and $\beta_2$, where points in $\text{Conf}^2(M)$ are represented by pairs of distinct points on the Lagrangians and the loops are given by pairs of disjoint paths with basepoints $(x_1,x_2)$ such that $x_1\in \overline{\alpha_{u/v}}\cap \beta_1$ and $x_2\in \overline{\alpha_{u/v}}\cap \beta_2$.
\end{definitionN}

At last, we may define the loops that we will feature in our theorem. In fact, these will be easier to define than they were for tangles in \cite{JHpI26}, as we no longer have the active/inactive distinction. In the case of knots and links, we will use the same homomorphisms $\Psi_W$ (for $W\subset \{X^+,X^-,Y\}$) and $\Phi$ that were used for tangles (defined in Section \ref{loophomsec}); the only thing different is the loops. As in Section \ref{loophomsec}, ``unique'' should be read as ``unique up to homotopy'' in the following discussion.

For now, the precise ordering of these points does not matter, but we will label them more carefully later. Given two of these intersections $\kappa_i$ and $\kappa_j$, there are two distinct paths from $\kappa_i$ to $\kappa_j$ along the blue curve $\overline{\alpha_{u/v}}$; let $\eta_{i,j}$ be the path that does not traverse the portion of $\overline{\alpha_{u/v}}$ tightly wrapped around $X^-$.  There is a unique path from $\kappa_j$ back to $\kappa_i$ along $\beta$ which we will call $\bar \eta_{j,i}$. 

\begin{definitionN}
We define the loop $\gamma_{i,j}^K:[0,1]\to M$ by the concatenation of paths $\eta_{i,j}\cdot\bar\eta_{j,i}$, which is well-defined up to homotopy.   
\end{definitionN}

Furthermore, there is a unique intersection point coming from $\overline{\alpha_{u/v}}$ and $\beta$ near $X^+$, which is always the right-most intersection in the $UP$ case and the left-most intersection in the $RI$ case (so it is the left-most intersection point in the image on the right in Figure \ref{fig8knotfig}). We will label this intersection point $\kappa_\omega$. 

\begin{definitionN}
 Define $\gamma_i^K:[0,1]\to M$ to be the loop $\gamma^K_i:=\gamma^K_{i,\omega}$.   
\end{definitionN}

See Figure \ref{gammaiknotex} for an example of one of these loops in the case of the figure-8 knot, $K_{5/2}$.

\begin{figure}
    \centering
    \includegraphics[height=6cm, angle=0]{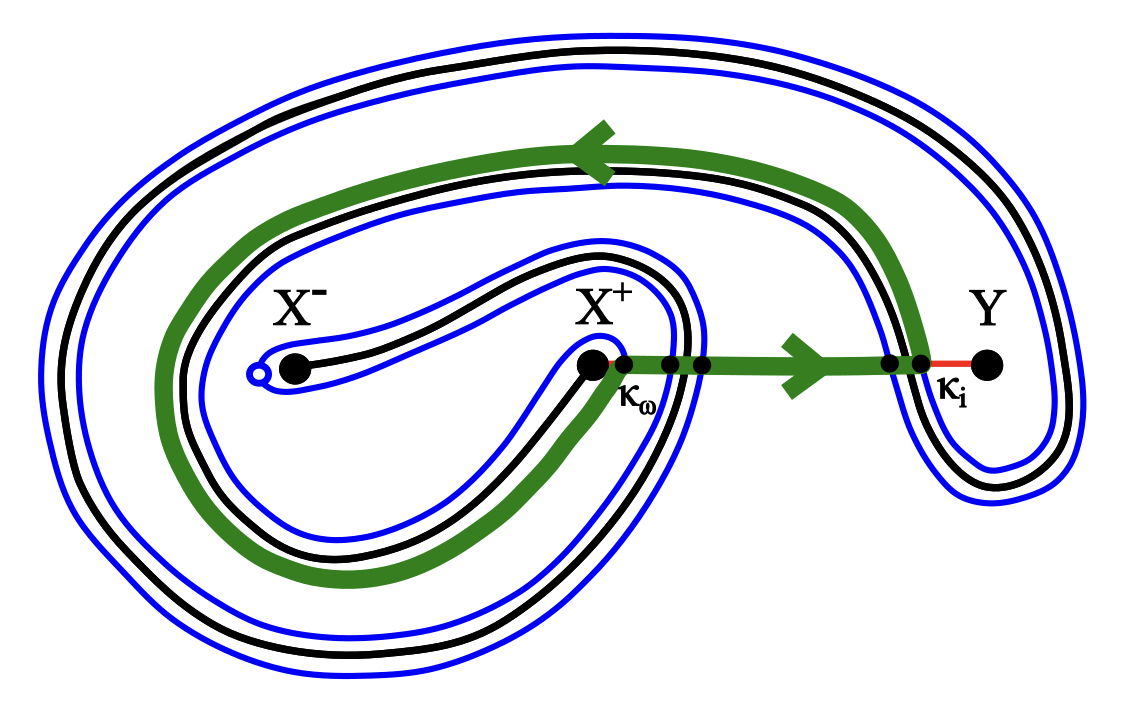}
    \caption{An example path $\gamma_i^K$ in $\mathcal{D}(K_{5/2})$}
    \label{gammaiknotex}
\end{figure}

Next, we will define our loops in $\text{Conf}^2(M)$. As indicated above, we will draw our picture in the plane by taking two parallel copies of $\beta$. To draw $\tilde{\gamma}_{i,j}^K=(\gamma_{i,j}^1,\gamma_{i,j}^2):[0,1]\to \text{Conf}^2(M)$ in $\mathcal{D}^2(K_{u/v})$, first consider which of $\kappa_i$ and $\kappa_j$ is further to the left in $\mathcal{D}(K_{u/v})$. This should be placed at the corresponding intersection on $\beta_1$ and the other on $\beta_2$. In other words, the basepoint for $\tilde{\gamma}_{i,j}^K$ always looks as shown below in a neighborhood of $\beta_1$ and $\beta_2$, where we have only drawn the parts of $\overline{\alpha_{u/v}}$ giving $\kappa_i$ and $\kappa_j$.
\[
\includegraphics[height=2cm]{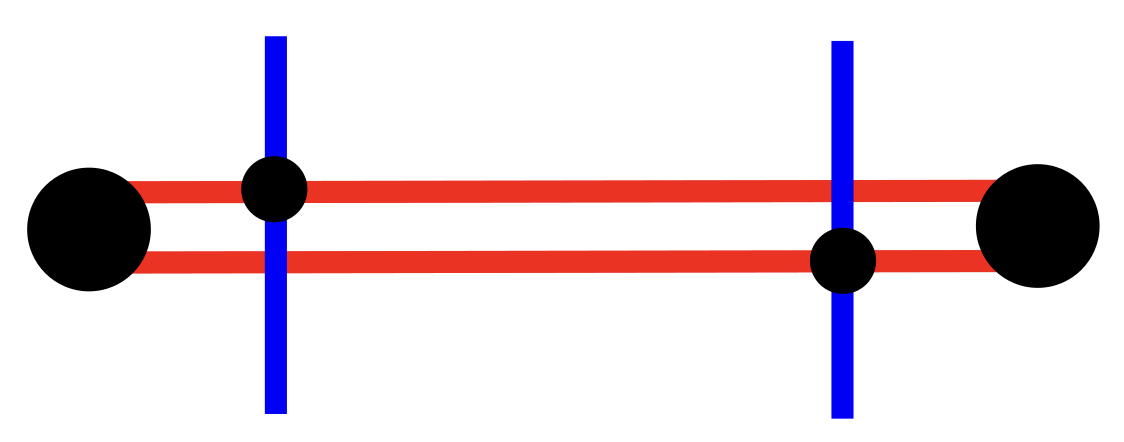}
\]

For the sake of our definition below, we will use $\beta_i$ to denote the horizontal on which the $\kappa_i$ point is placed, and similarly use $\beta_j$ for $\kappa_j$'s horizontal. The component path $\gamma_{i,j}^1$ is defined by first following $\overline{\alpha_{u/v}}$ (avoiding the end by $X_-$) until it returns to $\beta_i$ at the intersection for $\kappa_j$. If the path does not pass through the $\kappa_j$ point on $\beta_j$, then $\gamma_{i,j}^1$ concludes by following the path along $\beta_i$ back to $\kappa_i$ and $\gamma_{i,j}^2$ is the constant loop at $\kappa_j$; if it passes through the $\kappa_j$ point on $\beta_j$ first, $\gamma_{i,j}^1$ stops here and $\gamma_{i,j}^2$ moves in the same direction along $\overline{\alpha_{u/v}}$ from $\beta_j$ to $\beta_i$ and then back to $\kappa_i$ along $\beta_i$. See Figure \ref{wedge2paths} for examples of both cases. Note that we could also define these loops in terms of path concatenations after carefully defining additional notation.

\begin{note}
    We will occasionally abuse notation and write $\gamma_{i,j}^K$ for both $\gamma_{i,j}^K\subset M$ and $\tilde\gamma_{i,j}^K\subset \text{Conf}^2(M)$. It will be clear from context which of the loops we are referring to.
\end{note}



\begin{figure}
    \centering
    \includegraphics[height=6.5cm, angle=0]{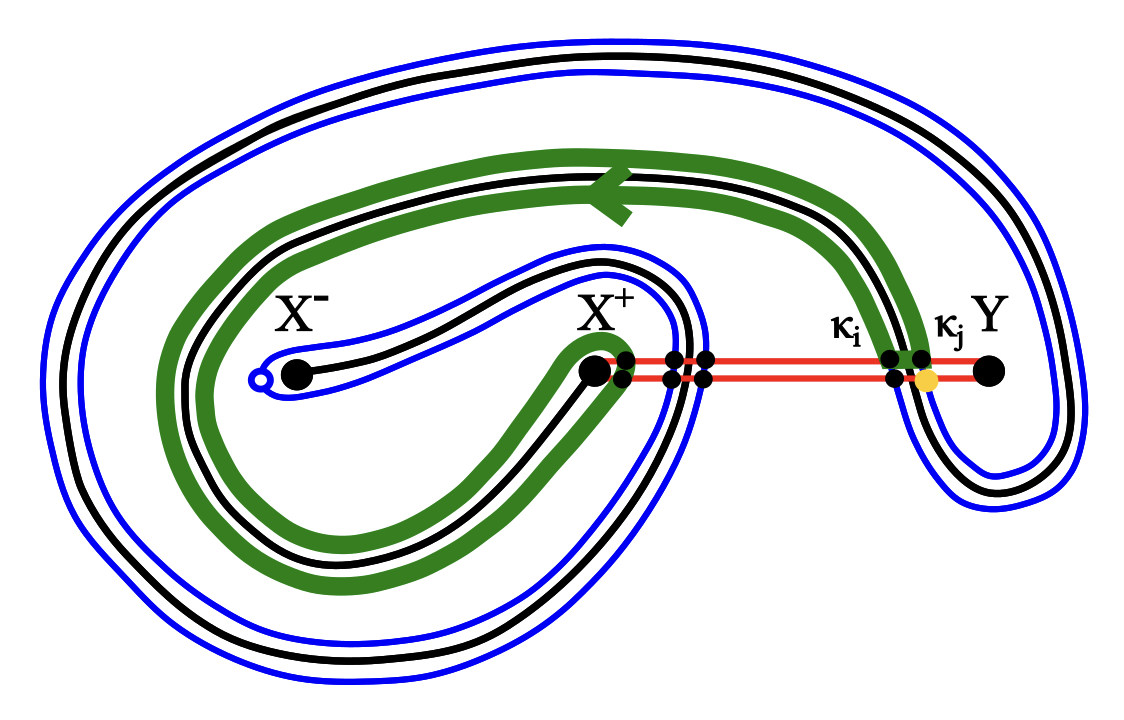}
    \qquad \qquad
    \includegraphics[height=6.5cm, angle=0]{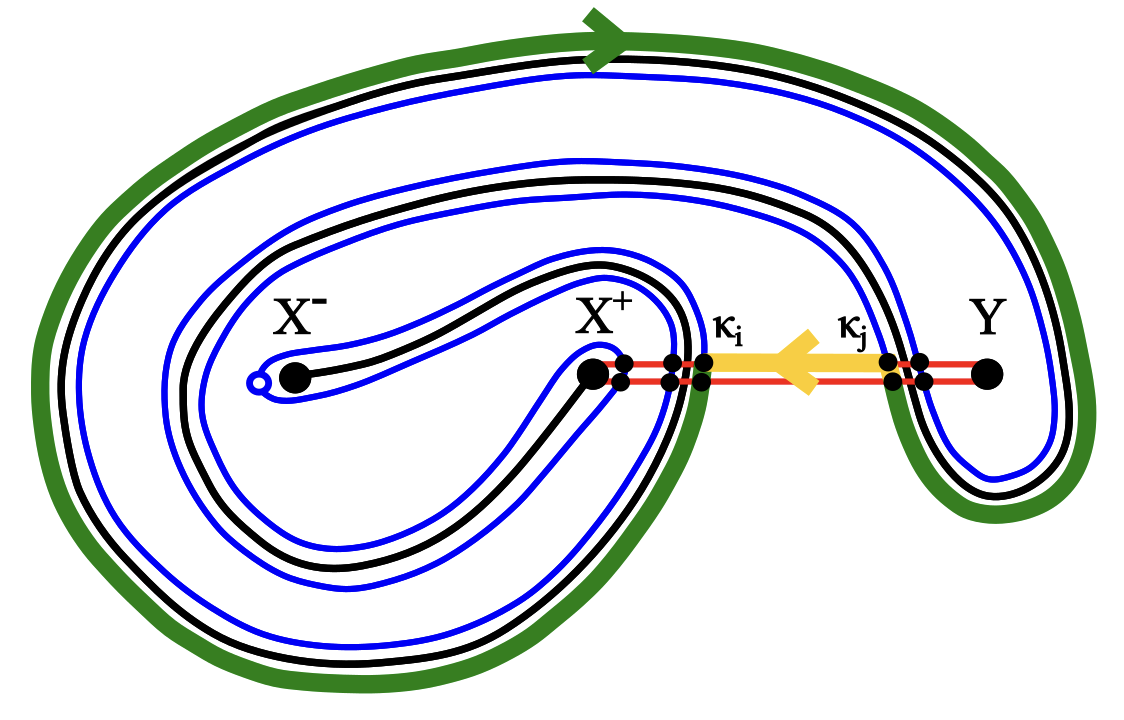}
    \caption{Examples of the two types of $\tilde{\gamma}_{i,j}$ paths in $\mathcal{D}^2(K_{5/2})$}
    \label{wedge2paths}
\end{figure}

Since all of our loops are defined to miss the portion of $\overline{\alpha_{u/v}}$ next to $X^-$, we will we use a small open circle in our figures to denote a hole or ``no cross zone'' for our loops. In fact we will include this hole in our diagrams $\mathcal{D}^j(K_{u/v})$. This can be seen, for example, in the past few figures for $K_{5/2}$. In light of this, there is no longer any choice for which paths our loops can take along $\overline{\alpha_{u/v}}$, so we could also define the loops in terms of this modified picture---adding the hole determines our loops up to homotopy, which is all we need for homomorphisms on fundamental groups.

Now that we have defined our loops, we can state the theorem we will prove for the remainder of this section. Notice that the formulas are the exact same as for the tangles, but with correction terms corresponding to overall shifts of $S,A,$ and $Q$, instead of the $\delta$ terms. We will notice the same phenomenon for links in the next section, but we will need to modify the Lagrangians once again.

\begin{theoremN}
\label{Knotsthm}
    For a rational knot $K_{u/v}$, $P(K_{u/v})$ can be written as
    \begin{equation}
        P(K_{u/v})=\sum_{\textbf{d}=(d_1,...,d_u)\in \mathbb{N}^u}(-q)^{S\cdot \textbf{d}}a^{A\cdot \textbf{d}}q^{\textbf{d} \cdot Q\cdot \textbf{d}^t} {d_1+...+d_u\brack d_1,...,d_u}
    \end{equation}
    where $S$, $A$, and $Q$ are defined on the free $\mathbb{Z}$-module $\mathbb{Z}\mathcal{G}_{u/v}^K$ and are computed by the following formulas:
    \[
    \begin{cases}
        S_i=\Psi_{\{X^\pm,Y\}}([\gamma_i^K])+\mu_1(K_{u/v})\\
        A_i=2\Psi_{X^+}([\gamma_i^K])+\mu_2(K_{u/v})\\
        \begin{cases}
            Q_{ii}=(\Psi_{\{X^-,Y\}}-3\Psi_{X^+})([\gamma_i^K])+\mu_3(K_{u/v})\\
            Q_{ij}=Q_{ii}+(\Phi-2\Psi_{X^+})([\gamma_{j,i}^K]),
        \end{cases}
    \end{cases}
    \]
    where the $\mu_i(K_{u/v})$ are integers that depend on the types of twists used to build the knot. They are defined in Section \ref{cortermsec}.
\end{theoremN}

The $\mu_i(K_{u/v})$ serve as correction terms for making sure that everything is scaled properly. By the work of Rasmussen \cite{R15}, we know that the (uncolored) HOMFLY-PT homology of $K_{u/v}$ should satisfy
\begin{equation}
\label{deltaeq1}
2\text{gr}_t-2\text{gr}_{a}-\text{gr}_q=\sigma(K_{u/v}),
\end{equation}
or, equivalently, the HOMFLY-PT homology of $K_{u/v}$ is concentrated in a $\delta$-grading determined by the signature of the knot. Furthermore, Sto\v si\'c and Wedrich used this fact in \cite{SW21} to show that the homological gradings are determined by the diagonal of $-Q$, which means we should have
\begin{equation}
\label{deltaeq2}
S_i-Q_{ii}-2A_i=\sigma(K_{u/v})
\end{equation}
for all $1\leq i\leq u$. Without the $\mu_i(K_{u/v})$ terms in Theorem \ref{Knotsthm}, we clearly have $S_i-Q_{ii}-2A_i=0$ for all $i$ for any $K_{u/v}$.

The proof of Theorem \ref{Knotsthm} will consist of three main steps. The first two involve showing that the linear and quadratic forms, respectively, are correct aside from the shifts given by the $\mu_i$'s. After this, we will determine what the shifts must be in order to ensure that Equation (\ref{deltaeq2}) holds.  

\subsection{Symmetry of $Q$}

Before proceeding to the proof of Theorem \ref{Knotsthm}, it should also be noted that, just as we have for tangles in \cite{JHpI26}, we can write a formula for computing $Q$ that only involves loops in $\text{Conf}^2(M)$. The way it is written in Theorem \ref{Knotsthm} is to simplify calculations by considering loops in $M$ as much as possible (rather than $\text{Conf}^2(M)$). As was done for tangles, if we take $\Phi^2=\Phi$ and the same homomorphism $\Psi_{X^+}^2:\text{Conf}^2(M)\to\mathbb{Z}$  as before, given by the composition
\[
\Psi_{X^+}^2:\pi_1(\text{Conf}^2(M))\xrightarrow{\iota_*}\pi_1(\text{Conf}^2(\mathbb{C}\setminus X^+))\to \pi_1(\text{Conf}^3(\mathbb{C}))\cong \text{Br}_3\xrightarrow{\text{ab}}\mathbb{Z},
\]
then we can compute $Q$ by the formula
\begin{equation}
\label{Conf2Qeq}
Q_{ij}=(2\Phi^2-\Psi_{X^+}^2)([\gamma_{j,i}^{2,K}]),
\end{equation}
where $\gamma_{j,i}^{2,K}$ is defined in a way similar to the loops $\gamma_{j,i}^2$ in Section 3.2 of \cite{JHpI26}, but with the new Lagrangians. In particular, $\gamma_{j,i}^{2,K}$ can be defined in terms of the concatenation of three different types of paths. The first path involves sliding $\kappa_j$ (on $\beta'_j$) along $\overline{\alpha_{u/v}}$ to the point in $\text{Conf}^2(M)$ where $\kappa_i$ is on both copies of $\beta'$. The second path involves sliding the two $\kappa_i$ points together along $\overline{\alpha_{u/v}}$ until they reach the point in $\text{Conf}^2(M)$ with both points at $\kappa_\omega$ on the two copies of $\beta'$. The third path involves returning to the starting point along $\beta'_1$ and $\beta'_2$. This is the analogous construction to the one given in more detail for tangles in \cite{JHpI26}, but adapted to the new Lagrangians.

Then, following a similar argument to the one used to prove Proposition 3.17 from \cite{JHpI26} gives the next proposition for knots.

\begin{figure}
   \begin{tikzpicture}
       \node at (0,0) {\includegraphics[height=7cm]{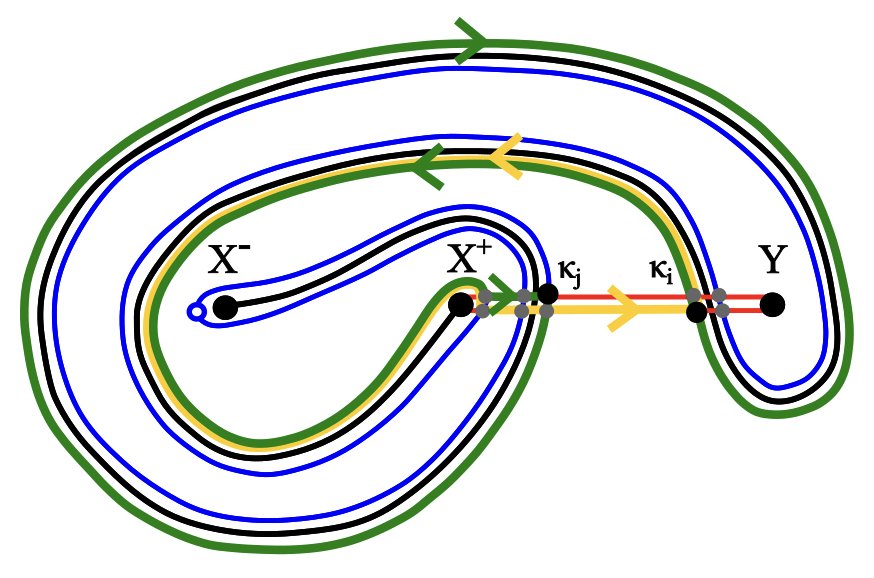}};
       \node at (8,0) {\includegraphics[height=4cm]{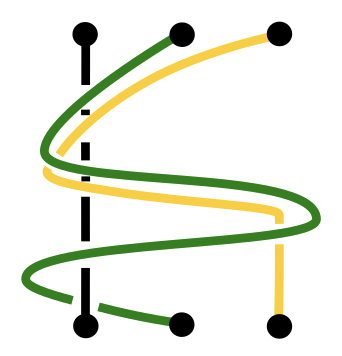}};
   \end{tikzpicture}
    \caption{Left: an example $\gamma_{j,i}^{2,K}$ loop for $K_{5/2}$. Right: the braid corresponding to $\gamma_{j,i}^{2,K}$, where the black strand comes from the point $X^+$. Note that computing $Q_{ij}$ from this loop gives $-2$, which agrees with Theorem \ref{Knotsthm}.}
    \label{gammaji2Kfig}
\end{figure}

\begin{propositionN}
\label{KnotQsymprop}
    The matrix $Q$ computed by Theorem \ref{Knotsthm} agrees with (\ref{Conf2Qeq}). Furthermore, $Q$ is indeed symmetric, i.e. $Q_{ij}=Q_{ji}$.
\end{propositionN}

In particular, one first proves that $\gamma_{i,j}^{2,K}$ and $\gamma_{j,i}^{2,K}$ are homologous in $\text{Conf}^2(M)$, which gives $Q_{ij}=Q_{ji}$ when computed by Equation (\ref{Conf2Qeq}). This follows from cancellation arguments similar to the ones used in the proof of Lemma 3.18 in \cite{JHpI26}. Then, following \cite{JHpI26}, one shows 
\[
\gamma_{j,i}^{2,K}
=\gamma_{i,i}^{2,K}+\tilde{\gamma}_{j,i}
\]
in $H_1(\text{Conf}^2(M))$, after which one only needs to establish that
\[
(2\Phi^2-\Psi_{X^+}^2)([\gamma_{i,i}^{2,K}])=(\Psi_{\{X^-,Y\}}-3\Psi_{X^+})([\gamma_i])
\]
and 
\[
(2\Phi^2-\Psi_{X^+}^2)([\tilde\gamma_{j,i}])=(\Phi_{\{X^-,Y\}}-2\Psi_{X^+})([\gamma_{j,i}]).
\]
We do not include the full proofs of these statements here because they are similar to the ones used in Section 3.2 of \cite{JHpI26}, but adapted to the new Lagrangians.

Now we will prove Theorem \ref{Knotsthm}, beginning with the linear forms.

\subsection{Step 1: The Linear Forms}
\label{pfknotssec}

To prove Theorem \ref{Knotsthm}, we will compare what the formulas give with what we know to work from \cite{SW21}. This will require us to study how entries in $S$, $A$, and $Q$ relate to each other and to the entries we get from the algebraic methods of \cite{SW21}, up to some shift. 

As we stated in the second paragraph of Section \ref{knotssetup}, rather than thinking of the rational knot $K_{u/v}$ as $\text{Cl}(\tau_{u/v})$, it helps to think of it as $\text{Cl}(T\tau_{(u-v)/v})$, where $\tau_{(u-v)/v}$ has $UP$ or $RI$ orientation (in the former case, $u-v$ is even and $v$ is odd; these are reversed for $RI$). The following lemma shows how the $\text{Cl}(T-)$ formulas from \cite{SW21} give $S,A,$ and $Q$ for $K_{u/v}$ in terms of these linear and quadratic forms for $\tau_{(u-v)/v}$, computed by Proposition \ref{indpartprop}.
This subsection and the next will consist of comparing the $S,A,$ and $Q$ computed by Theorem \ref{Knotsthm} with the ones below.

\begin{lemmaN}
\label{WSclosureformula}
    If $K_{u/v}=\text{Cl}(T\tau_{(u-v)/v})$ where $\tau_{(u-v)/v}$ has $UP$ orientation, then the $\text{Cl}(T-)$ operation has the effect
   \begin{multline*}
      UP, \left[\begin{array}{ c| c | c  }
    1 & S_+ & A_+  \\
    \hline
    0 & S_- & A_- 
  \end{array}\right],  
  \left[\begin{array}{ c |c }
    Q_{++} & Q_{+-} \\
    \hline
    Q_{-+} & Q_{--}
  \end{array}\right]
  \\
    \xrightarrow{\text{Cl}(T-)}
    \left[\begin{array}{ c | c  }
    S_++1 & A_+  \\
    \hline
    S_++2 & A_+ +2 \\
    \hline
    S_- & A_- 
  \end{array}\right], 
  \left[\begin{array}{ c | c | c}
    Q_{++}+2 & Q_{++}+L & Q_{+-} \\
    \hline
    Q_{++}+U & Q_{++}-1 & Q_{+-}-1\\
    \hline
    Q_{-+} & Q_{-+}-1 & Q_{--}
  \end{array}\right],
   \end{multline*}
   on the generating function data, where $\mathfrak{Y}$ consists of $\frac{u-v}{2}$ active intersections and $\mathfrak{Z}$ consists of the $v$ inactive intersections.

  In the case that $\tau_{(u-v)/v}$ has $RI$ orientation, then $\text{Cl}(T-)$ has the effect
\begin{multline*}
      RI, \left[\begin{array}{ c| c | c  }
    0 & S_+ & A_+  \\
    \hline
    1 & S_- & A_- 
  \end{array}\right],  
  \left[\begin{array}{ c |c }
    Q_{++} & Q_{+-} \\
    \hline
    Q_{-+} & Q_{--}
  \end{array}\right]
  \\
    \xrightarrow{\text{Cl}(T-)}
    \left[\begin{array}{ c | c  }
    S_+  & A_+  \\
    \hline
    S_- -1 & A_- -2 \\
    \hline
    S_-  & A_- 
  \end{array}\right], 
  \left[\begin{array}{ c | c | c}
    Q_{++} & Q_{+-} & Q_{+-}-1 \\
    \hline
    Q_{-+} & Q_{--}+2 & Q_{--}+L\\
    \hline
    Q_{-+}-1 & Q_{--}+U & Q_{--}-1
  \end{array}\right],
   \end{multline*}
on the generating function data, where $\mathfrak{Y}$ consists of $\frac{v}{2}$ inactive intersections and $\mathfrak{Z}$ consists of the $u-v$ inactive intersections. 

In both cases, the active points always come before the inactive ones.
\end{lemmaN}

Note that we are taking advantage of the fact that either the numerator or denominator of the fraction $\frac{u-v}{v}$ representing the tangle is even; applying Proposition \ref{indpartprop} allows us to cut the number of indices from the active or inactive side in half at the expense of adding a Pochhammer symbol $(q^2;q^2)_{K\cdot \textbf{d}}$. The linear form $K$ is the first column of the first matrix on the tangle side of the transformations above. Furthermore, we are not concerned with orderings within the blocks $\mathfrak{Y}$ and $\mathfrak{Z}$ at this time, but we will need to fix an ordering of this basis later.

\begin{proof}[Proof of Lemma \ref{WSclosureformula}]
    In \cite{SW21}, Sto\v si\'c and Wedrich prove
    \begin{equation}
    \label{TUPclosure}
    \text{Cl}(TUP[j,k]) \sim (-q)^kq^{2k^2}{j \brack k}_+ \frac{(a^2q^{2-2j-2k};q^2)_k}{(q^2;q^2)_k}
    \end{equation}
    and 
    \begin{equation}
    \label{TRIclosure}
    \text{Cl}(TRI[j,k]) \sim (-q)^{k-j}a^{2(k-j)}q^{2(k-j)^2}{j \brack k}_+ \frac{(a^2q^{2-2j-2(j-k)};q^2)_{j-k}}{(q^2;q^2)_{j-k}},
    \end{equation}
    as we saw in Lemma \ref{WSclosureformulas}. Recall that we used $\sim$ to denote ``up to framing shift.''

    Given these formulas for $\text{Cl}(T-)$, this lemma follows from some easy algebra and an application of Lemma \ref{algidentity} for splitting indices. See the proof of Theorem 4.1 in \cite{SW21} for similar computations.
\end{proof}

The next natural step in proving Theorem \ref{Knotsthm} is to show that the formulas for $S$, $A$, and the diagonal of $Q$ are correct (up to some overall shift), as they only require computations involving loops in $M$. In order to do so, we need to establish a method for pairing the Lagrangian intersections in $\mathcal{D}(K_{u/v})$ with the ones in $\mathcal{D}(\tau_{(u-v)/v})$. Once we do so, we can show that they are related in the way presented in Lemma \ref{WSclosureformula}. 

First, observe that we can isotope the picture for $K_{u/v}$ in a way that has the effect of undoing the final top twist applied to $\tau_{(u-v)/v}$ so that the blue Lagrangian $\overline{\alpha_{u/v}}$ looks like $\overline{\alpha_{(u-v)/v}}$ and the red Lagrangian $\beta$ is bent into a ``U'' shape. We will refer to this as the \textit{untwisting isotopy}. See Figure \ref{fig8untwist}. In light of Lemma \ref{WSclosureformula}, this is the sort of picture that we will typically be working with to prove Theorem \ref{Knotsthm}.

\begin{figure}
\begin{tikzpicture}
\node at (-4.1,0) {\includegraphics[height=3.7cm, angle=0]{fig8geopic.png}};
\node at (0,0) {$\xrightarrow{\text{isotopy}}$};
\node at (4.1,0) {\includegraphics[height=3.7cm, angle=0]{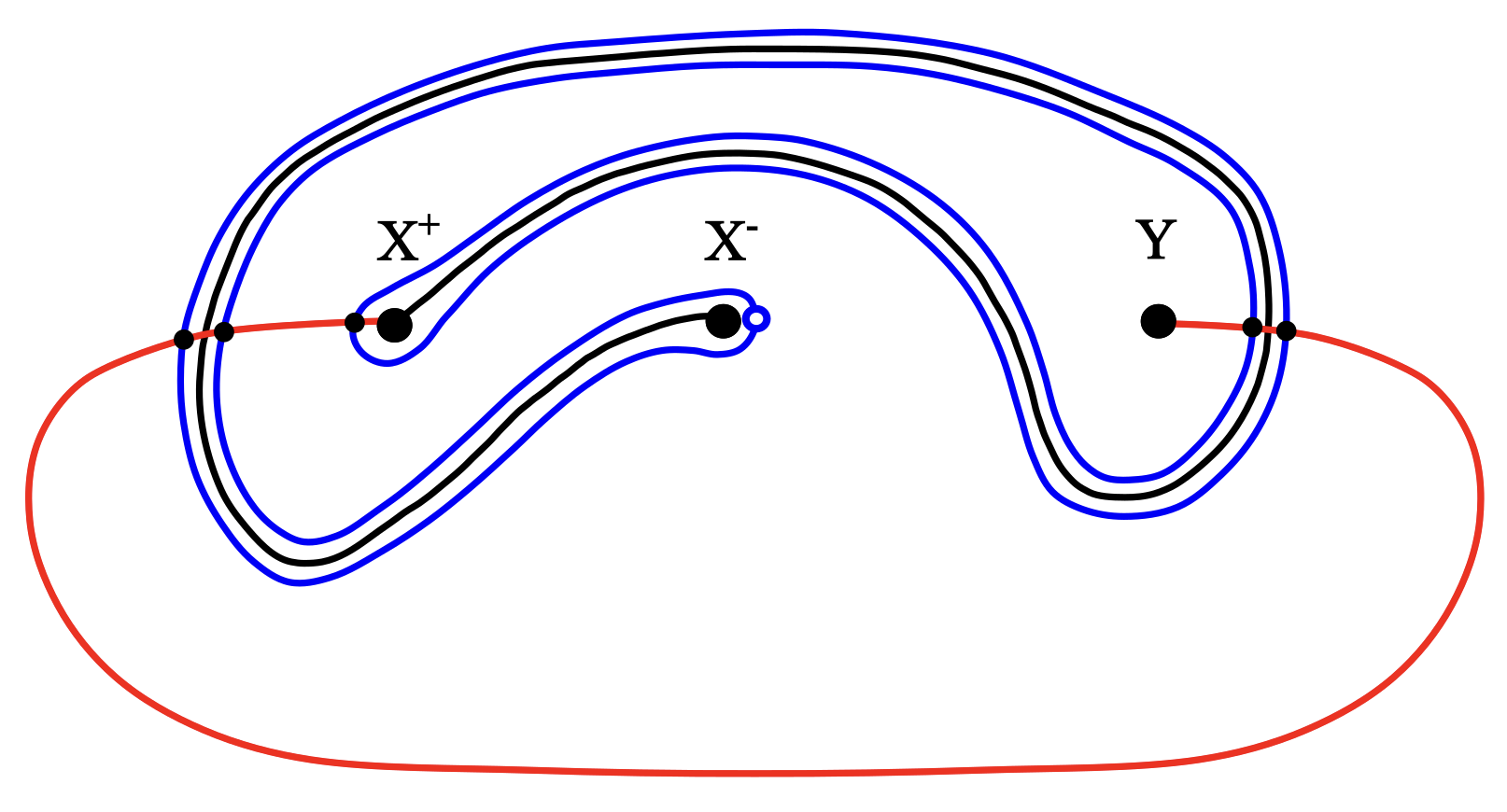}};
\end{tikzpicture}
\caption{The result of the isotopy that takes $\mathcal{D}(K_{5/2})$ to its ``bent'' version. Notice that the blue Lagrangian now loops around $\alpha_{3/2}$.}
\label{fig8untwist}  
\end{figure}

Now that we have bent $\mathcal{D}(K_{u/v})$ to resemble $\mathcal{D}(\tau_{(u-v)/v})$ in the sense that $\alpha_{u/v}$ is transformed to $\alpha_{(u-v)/v}$, we can more easily pair up intersection points of the Lagrangians. Although we do not have ``active'' and ``inactive'' intersection points for the knot, it will help to think of them in this way based on the picture after applying the untwisting isotopy. In particular, it is sometimes convenient to think of the $\kappa_i$ paired with active intersections for $\tau_{(u-v)/v}$ as active intersections for $K_{u/v}$, and similarly for the inactive case. 

Note that everything to the left of the active axis $l_A$ in $\mathcal{D}(\tau_{(u-v)/v})$ looks like concentric semi-circles, potentially with a single arc from $l_A$ to one of the endpoints of $\alpha_{(u-v)/v}$ in the middle. This is the same for everything to the right of $l_I$. For the two cases we need to consider, $X^-$ is always the middle of the three punctures (or $Z=X^-$, in the notation of the previous section), so we only need to concern ourselves with how to pair intersections related by semicircles with the active and inactive axes and by an arc from one of the axes to $X^+$ for the tangle $\tau_{(u-v)/v}$ with the Lagrangian intersections in $\mathcal{D}(K_{u/v})$. The latter case is obvious---there is a single intersection point in $\mathcal{D}(K_{u/v})$ coming from this arc, and this is the point we called $\kappa_\omega$. Some care must be taken for pairing the other types of intersections.

Observe that the pairs of intersections for $\tau_{(u-v)/v}$ coming from a semi-circle made by a portion of $\alpha_{(u-v)/v}$ with $l_A$ or $l_I$ are related to the bent picture for $K_{u/v}$ by the blue Lagrangian running parallel to $\alpha_{(u-v)/v}$ on either side, and the red Lagrangian $\beta$ cuts horizontally through the middle as in Figure \ref{tangleknotpairs}. Thus, we can match the two intersections of the ends of the semi-circle with the active or inactive axis with the two intersections between the blue and red Lagrangians. There are two ways to do so. To determine how, we will impose an orientation on $\alpha_{u/v}$; in particular, we will orient it from $X^-$ to $X^+$, as we did to define $\prec$ in Section \ref{TangleLsection}. Figure \ref{tangleknotpairs} shows the four different situations that can arise and how the points are matched in each.

\begin{figure}
 \raisebox{0pt}{\includegraphics[height=4cm, angle=0]{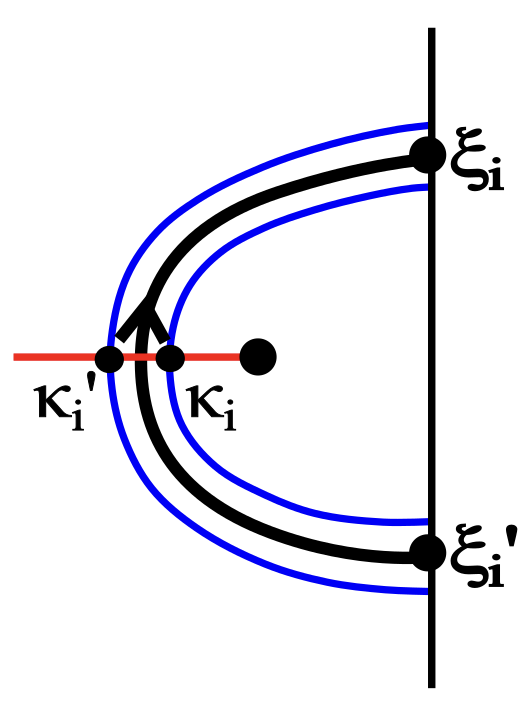}}\qquad \raisebox{0pt}{\includegraphics[height=4cm, angle=0]{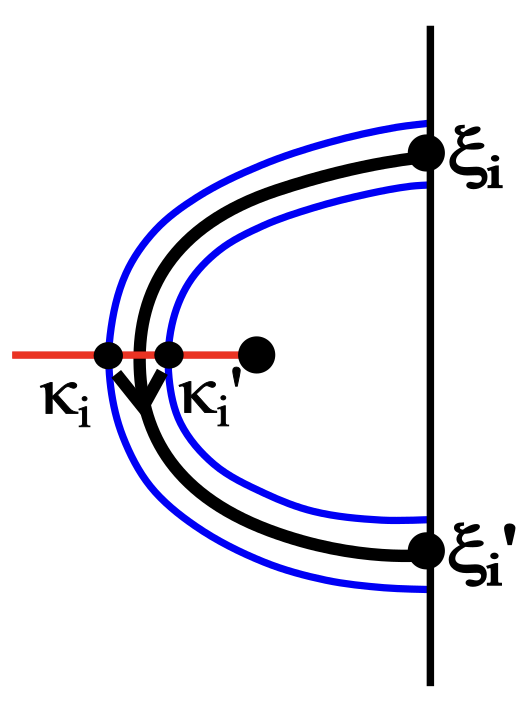}}\qquad \raisebox{0pt}{\includegraphics[height=4cm, angle=0]{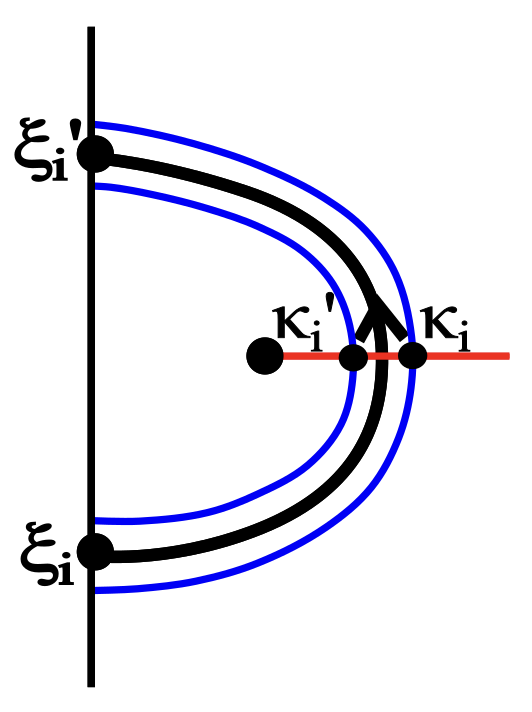}}\qquad \raisebox{0pt}{\includegraphics[height=4cm, angle=0]{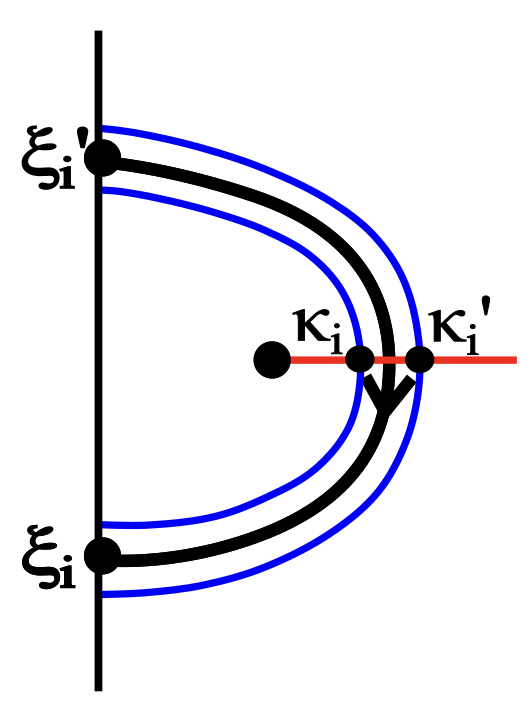}}
\caption{The four cases to consider for pairing intersections in the tangle picture with intersections in the knot picture when coming from a semi-circle. The vertical line is $l_A$ for the two left-most pictures and $l_I$ for the two on the right. Naturally, the intended pairings are $\xi_i \leftrightarrow \kappa_i$ and $\xi_i' \leftrightarrow \kappa_i'$.}
\label{tangleknotpairs}   
\end{figure}

Now that we have matched up intersection points, we have a direct way of relating individual computations for $\tau_{(u-v)/v}$ (as in Proposition \ref{tangleHOMFLYpoly} or \ref{indpartprop}) with corresponding ones for $K_{u/v}$ (using Theorem \ref{Knotsthm}).

For the following lemma, we will use $Q_+$ to denote the diagonal of $Q_{++}$ and $Q_-$ for the diagonal of $Q_{--}$. We can then use this notation to give the linear form defined by the diagonal of $Q$, which is given by the last column of matrices in the statement of Lemma \ref{knotslinforms}.

\begin{lemmaN}
\label{knotslinforms}
    According to the formulas in Theorem \ref{Knotsthm}, if $\tau_{(u-v)/v}$ has orientation $UP$, the $\text{Cl}(T-)$ operation has the effect
\[
    UP, \left[\begin{array}{ c| c | c | c }
    1 & S_+ & A_+ & Q_+ \\
    \hline
    0 & S_- & A_- & Q_-
  \end{array}\right]
    \xrightarrow{\text{Cl}(T-)} 
    \left[\begin{array}{ c | c | c }
    S_++1 & A_+ & Q_++2 \\
    \hline
    S_++2 & A_+ +2 & Q_+-1\\
    \hline
    S_- & A_- & Q_-
  \end{array}\right],
\]
    and if $\tau_{(u-v)/v}$ has orientation $RI$, then $\text{Cl}(T-)$ has the effect
\[
     RI, \left[\begin{array}{ c| c | c | c }
    0 & S_+ & A_+ & Q_+ \\
    \hline
    1 & S_- & A_- & Q_- 
  \end{array}\right]
    \xrightarrow{\text{Cl}(T-)}
    \left[\begin{array}{ c | c | c }
    S_+  & A_+ & Q_+ \\
    \hline
    S_- -1 & A_- -2 & Q_-+2\\
    \hline
    S_-  & A_- & Q_--1
  \end{array}\right].
  \]
\end{lemmaN}

These are the same formulas as in Lemma \ref{WSclosureformula}, but one computation is coming from \cite{JHpI26} and the other from the formulas in Theorem \ref{Knotsthm}. In the statement of the lemma, we are implicitly using the matching of intersection points for $\tau_{(u-v)/v}$ with the ones for $K_{u/v}$. Thus, the lemma says that if we perform the computations in Proposition \ref{indpartprop} for $\tau_{(u-v)/v}$ to get the data on the left-hand side (for the $UP$ or $RI$ case), then Theorem \ref{Knotsthm} computes the data on the right for $K_{u/v}=\text{Cl}(T\tau_{(u-v)/v})$ with the blocks coming from the pairing of intersection points. Observe that the points in $\mathfrak{Y}$ for $\tau_{(u-v)/v}$ double to account for the points in $\mathfrak{X}$.

\begin{proof}[Proof of Lemma \ref{knotslinforms}]
    We will prove the $UP$ case and leave the $RI$ case to the reader. 

    If $\tau_{(u-v)/v}$ has $UP$ orientation with $u-v$ even and $v$ odd, then it has state $(UP, Y | X^-|X^+)$ and $\kappa_\omega$ as an inactive intersection point. The active intersection points for $\tau_{(u-v)/v}$ are partitioned as $\mathfrak{X}\sqcup\mathfrak{Y}$, and in the application of Proposition \ref{indpartprop}, we are only taking $\mathfrak{Y}$ when introducing the $K$ vector. Geometrically, $\mathfrak{Y}$ consists of the top $\frac{u-v}{2}$ intersection points between $\alpha_{(u-v)/v}$ and $l_A$ (the ones ``above'' $Y$, or the $\xi_i$ in the first two images of Figure \ref{tangleknotpairs}). The inactive intersections give the set $\mathfrak{Z}$.

    On the active side, if $\kappa_i$ and $\kappa_i'$ are related as in one of the first two images of Figure \ref{tangleknotpairs}, then we will need to compare them both with $\xi_i$, the corresponding point for the tangle in $\mathfrak{Y}$. Thus, the first block on the knot side in Lemma \ref{knotslinforms} comes from the points $\kappa_i$ paired with $\xi_i\in \mathfrak{Y}$ and the second block comes from the $\kappa_i'$ related to $\kappa_i$  and $\xi_i$ as in Figure \ref{tangleknotpairs}, or the points in $\mathfrak{X}$. There is no ambiguity in handling the inactive intersections, as they are not compressed for the $UP$ case, and the one-to-one (geometric) pairing $\xi_k \leftrightarrow \kappa_k$ for inactive indices translates to a natural one-to-one pairing of the corresponding entries in the linear forms for the tangles and knots. 

    If $\xi_k$ is an inactive index for the tangle $\tau_{(u-v)/v}$, it is easy to verify that $\Psi_W([\gamma^K_k])=\Psi_W([\gamma^T_k])$ for any $W\subseteq \{X^+,X^-,Y\}$, so we proceed to consider what happens on the active side.

     \begin{figure}
        \centering
        \includegraphics[height=5cm, angle=0]{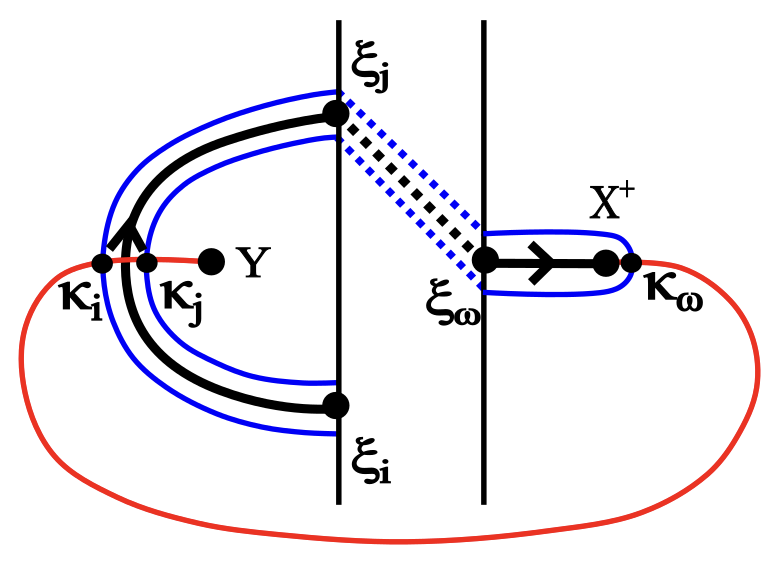}
        \qquad 
        \includegraphics[height=5cm, angle=0]{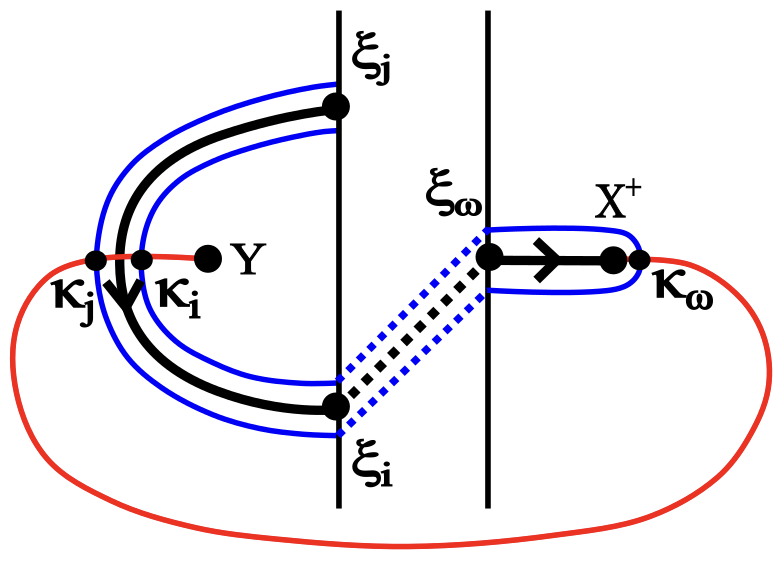}
        \caption{Images for proof of Lemma \ref{knotslinforms}. The dotted portions designate the parts of the Lagrangians $\alpha_{u/v}$ and $\overline{\alpha_{u/v}}$ that result in parallel segments of the loops $\gamma_i$ and $\gamma_j$ for the knot and tangle.}
        \label{linformknotfig}
    \end{figure}

    Observe that the difference between the $\gamma^K_i$ and $\gamma^K_{i'}$ loops for knots and the $\gamma^T_i$ loops for the tangles are the same (they run parallel to each other) everywhere except at the beginning and end, so these are the only parts we need to consider. See Figure \ref{linformknotfig}. We did not add all the distinct loops to the figure to avoid over-complicating it, but they can easily be added by following definitions. After drawing out all the loops, the result now follows from the observation:
\[
     \begin{cases}
        \Psi_Y([\gamma^K_i])-\Psi_Y([\gamma^T_i])=1\\
        \Psi_{X^-}([\gamma^K_i])-\Psi_{X^-}([\gamma^T_i])=1 \\
        \Psi_{X^+}([\gamma^K_i])-\Psi_{X^+}([\gamma^T_i])=0
    \end{cases}
\]
    and
\[
    \begin{cases}
        \Psi_Y([\gamma^K_{i'}])-\Psi_Y([\gamma^T_i])=1\\
        \Psi_{X^-}([\gamma^K_{i'}])-\Psi_{X^-}([\gamma^T_i])=1 \\
        \Psi_{X^+}([\gamma^K_{i'}])-\Psi_{X^+}([\gamma^T_i])=1,
    \end{cases}
\]
    noticing that we lose the $\delta_{i,A}\delta_{\omega,I}$ term in the tangle formula for $S$.
\end{proof}

Thus, we have proved that the formulas for the linear forms involved in Theorem \ref{Knotsthm} (including the diagonal of $Q$) work, up to some overall shift, because we can make them match exactly with what we expect from Lemma \ref{WSclosureformula}. Now, we need to prove that the formula for the off-diagonal entries of $Q$ in Theorem \ref{Knotsthm} also works; this will occupy the next subsection.

\subsection{Step 2: The Quadratic Form}
\label{Knotquadformsec}

Next, we need to show that the formula for the off-diagonal entries of $Q$, or 
\[
Q_{ij}=Q_{ii}+(\Phi-2\Psi_{X^+})([\gamma^K_{j,i}])
\]
allows one to successfully compute the $j$-colored HOMFLY-PT polynomials. Shortly, we will define a precise ordering of the $\kappa_i$ that gives an exact agreement between the upper left $2\times 2$ blocks of the $Q$ matrices for what our formula gives with the one from Lemma \ref{WSclosureformula} in the $UP$ case, and similarly for the lower right $2\times 2$ blocks in the $RI$ case. In particular, this is the ordering that ensures the $+L$ in the (1,2)-block for $UP$ and the (2,3)-block for $RI$. We will find, however, that fixing the ordering in this way is not sufficient for matching up the results of our formulas with Lemma \ref{Knotsthm} completely; more work will be needed for the remaining blocks. 

We now fix our ordering of the $\kappa_i$, which will be defined in terms of three blocks (corresponding to the three blocks seen in Lemma \ref{WSclosureformula}). In the $UP$ case, we have two blocks for intersection points coming from the active side of $\tau_{(u-v)/v}$ and a single block for the inactive side. Once again, orient $\alpha_{(u-v)/v}$ from $X^-$ to $X^+$, which comes with the $\prec$ ordering of the intersection points for the tangle. First, we label $\kappa_1,...,\kappa_{(u-v)/2}$, which constitute the first block for the $UP$ case, as follows. If the $\frac{u-v}{2}$ active intersections for $\tau_{(u-v)/v}$ in $\mathfrak{Y}$ (the ones ``above'' $Y$) are ordered $\xi_1,...,\xi_{(u-v)/2}$ with $\xi_i\prec \xi_j$ if and only if $i<j$, then the first $\frac{u-v}{2}$ $\kappa_i$'s are ordered according to the correspondence $\kappa_i\leftrightarrow \xi_i$. For $\kappa_{(u-v)/2+1},...,\kappa_{u-v}$, we follow a similar procedure but with the $\xi_i$'s in $\mathfrak{X}$ (the ones ``below'' $Y$). 

The third block of intersection points (coming from the inactive side), $\kappa_{u-v+1},...,\kappa_p$, are ordered according to the ordering $\prec$ of the inactive indices for $\tau_{(u-v)/v}$ and the pairing of the $\kappa_i$'s with $\xi_i$'s. An example for the $UP$ case is shown below in Figure \ref{7_3orderfig}.

\begin{figure}[H]
    \centering
    \includegraphics[height=3.5cm]{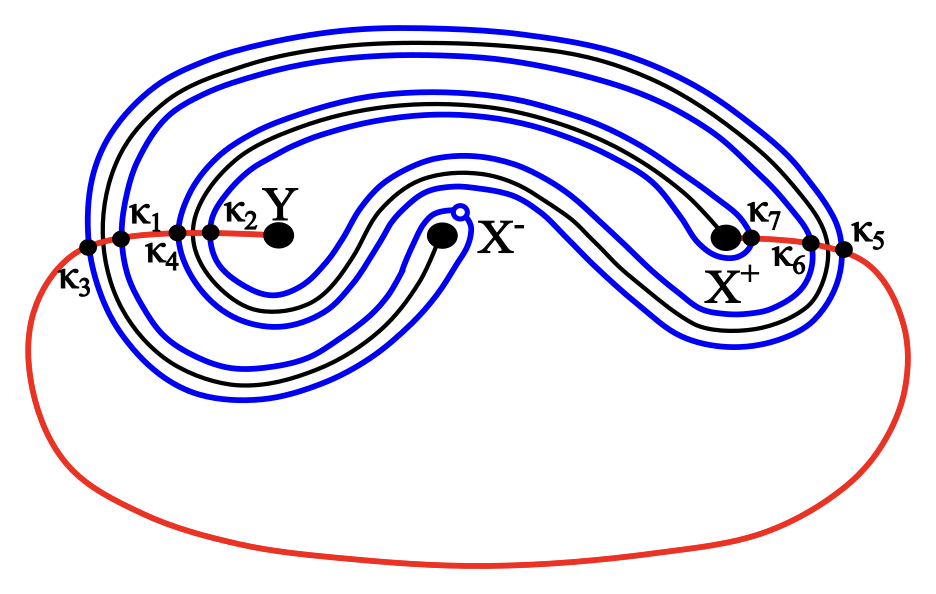}
    \caption{Ordering of the $\kappa_i$ in $\mathcal{D}(K_{7/3})$ (after untwisting isotopy)}
    \label{7_3orderfig}
\end{figure}

The ordering is similar for the $RI$ case. The rules for ordering the active and inactive intersections are switched, and the single block of active intersections still comes first. See Figure \ref{fig8orderfig} for an example.

\begin{note}
    If $\kappa_i$ and $\kappa_i'$ are related as in Figure \ref{tangleknotpairs}, then $i-i'=\frac{u-v}{2}$ in the $UP$ case and $i-i'=\frac{v}{2}$ for the $RI$ case.
\end{note}

\begin{figure}[H]
    \centering
    \includegraphics[height=3.5cm]{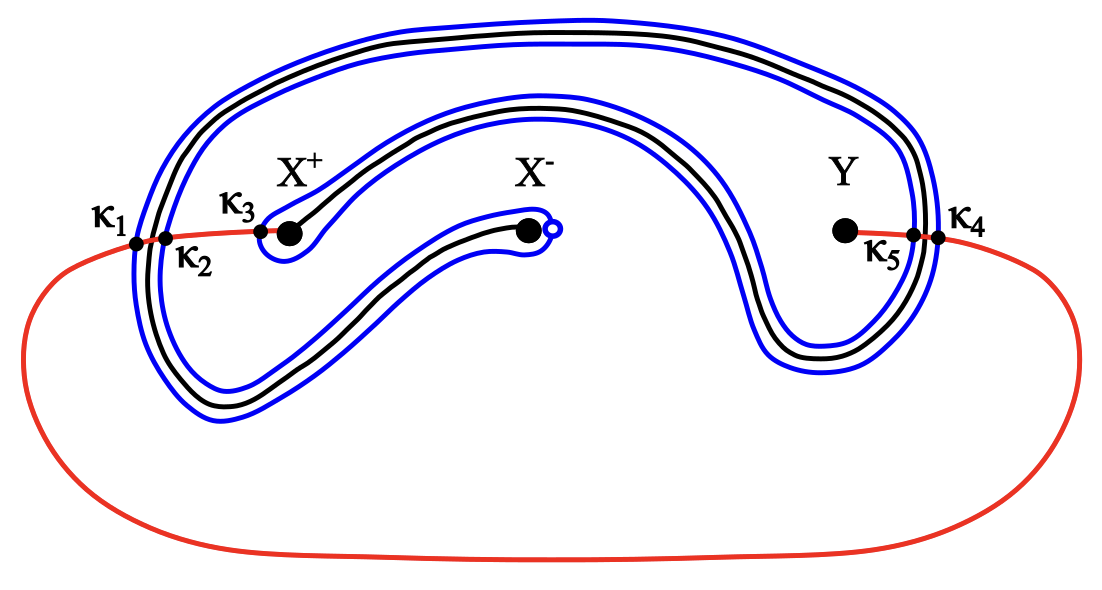}
    \caption{Ordering of the $\kappa_i$ in $\mathcal{D}(K_{5/2})$ (after untwisting isotopy)}
    \label{fig8orderfig}
\end{figure}

\begin{definitionN}
    Let the \textit{standard ordering of the $\kappa_i$} for $K_{u/v}$ be the ordering just described.
\end{definitionN}

For the remainder of the section, we will need to work with particular $Q$'s, which are determined by ordered bases. Thus, we introduce the following notation.

\begin{definitionN}
    Let $Q_{u/v}^K$ denote the particular $u\times u$ matrix for $K_{u/v}$ computed by Theorem \ref{Knotsthm} with ordered basis given by the standard ordering of the $\kappa_i$. 
\end{definitionN}

We can describe $Q_{u/v}^K$ as a $3\times 3$ block matrix (as written on the ``knot side'' in Lemma \ref{WSclosureformula}) determined by a partition $\mathfrak{A}\sqcup\mathfrak{B}\sqcup\mathfrak{C}$ of $\mathcal{G}_{u/v}^K$. More precisely, we let $\mathfrak{A}=\{\kappa_1,...,\kappa_{(u-v)/2}\}$, $\mathfrak{B}=\{\kappa_{(u-v)/2+1},...,\kappa_{u-v}\}$, and $\mathfrak{C}=\{\kappa_{u-v+1},...,\kappa_u\}$ if $\tau_{(u-v)/v}$ has $UP$ orientation. Recall that $\mathcal{G}_{(u-v)/v}^T=\{\xi_1,...\xi_u\}$ is partitioned by the sets $\mathfrak{X},\mathfrak{Y},\mathfrak{Z}$, where we apply Proposition \ref{indpartprop} to express the generating function data for the tangle in terms of restrictions of the linear and quadratic forms to the submodule of $\mathbb{Z}\mathcal{G}_{(u-v)/v}$ given by the basis $\mathfrak{Y}\sqcup\mathfrak{Z}$. In the $UP$ orientation case, $\mathfrak{A}$ was paired with $\mathfrak{Y}$, $\mathfrak{B}$ was paired with $\mathfrak{X}$, and $\mathfrak{C}$ was paired with $\mathfrak{Z}$. If we order $\mathfrak{Y}$ and $\mathfrak{Z}$ independently by $\prec$ with points in $\mathfrak{Y}$ coming before the points in $\mathfrak{Z}$, then we have fixed an ordered basis for this submodule, so Proposition \ref{indpartprop} gives a specific matrix $Q$ for $\tau_{(u-v)/v}$.

\begin{definitionN}
    Let $Q_{(u-v)/v}^T$ denote the $\frac{u+v}{2}\times \frac{u+v}{2}$ matrix $Q$ computed by Proposition \ref{indpartprop}, given the ordering of the basis $\mathfrak{Y}\sqcup\mathfrak{Z}$ just described.
\end{definitionN}

Thus, we see that $Q_{(u-v)/v}^T$ can be represented as a $2\times 2$ block matrix, as shown in Lemma \ref{WSclosureformula}, but now we have the blocks determined by the ordering of the basis. Applying Lemma \ref{WSclosureformula} to $Q_{(u-v)/v}^T$ results in a specific $u\times u$ matrix, which can be expressed as a $3\times 3$ block matrix that can be compared with $Q_{u/v}^K$.

\begin{definitionN}
    Let $Q_{u/v}$ be the $u\times u$ matrix obtained by applying Lemma \ref{WSclosureformula} to $Q_{(u-v)/v}^T$.
\end{definitionN}

We can think of $Q_{u/v}$ as an example matrix that is known to work for the quiver form for $P(K_{u/v})$. 

\begin{lemmaN}
\label{Qplusblockknots}
   If $K_{u/v}=\text{Cl}(T\tau_{(u-v)/v})$, where $\tau_{(u-v)/v}$ has $UP$ orientation, then the top left $2\times 2$ blocks of $Q_{u/v}^K$ agrees with these same blocks for $Q_{u/v}$. If $\tau_{(u-v)/v}$ has $RI$ orientation, then the same holds for the lower right $2\times 2$ blocks.
\end{lemmaN}

\begin{proof}
    Once again, we will prove the $UP$ case and leave the $RI$ case to the reader.

    If we cut $\overline{\alpha_{u/v}}$ by $X^+$ (say, at $\kappa_\omega$) as we have done for $X^-$, we have two arcs parallel to each other and to $\alpha_{(u-v)/v}$, with $\mathfrak{A}$ on one of them and $\mathfrak{B}$ on the other. By sliding the red Lagrangians around to be vertical at $Y$, as in Figure \ref{7_3bendfig}, it is easy to see that the $(1,1)$- and $(2,2)$-blocks of $Q_{u/v}^K$ match $Q_{u/v}$. In particular, this figure illustrates how $\gamma^K_{i,j}$ for $\kappa_i,\kappa_j\in \mathfrak{A}$ can be isotoped to run parallel to $\gamma^T_{i,j}$, so that $\Psi_W([\gamma_{i,j}^K])=\Psi_W([\gamma_{i,j}^T])$ for all $W\subseteq\{X^+,X^-,Y\}$ and $\Phi([\gamma_{i,j}^K])=\Phi([\gamma_{i,j}^T])$. The same procedure can be done for pairs of points in $\mathfrak{B}$ by bending the red Lagrangians the other way. Thus, we see that the $(1,1)$- and $(2,2)$-blocks of $Q_{u/v}^K$ are shifts of the $(1,1)$-block of $Q_{(u-v)/v}^T$, with the shift determined by the diagonal. We know that the same holds true for $Q_{u/v}$, and these shifts from the diagonal are known to be the same by Lemma \ref{knotslinforms}.

    \begin{figure}
\begin{tikzpicture}
\node at (-4.1,0) {\includegraphics[height=4cm, angle=0]{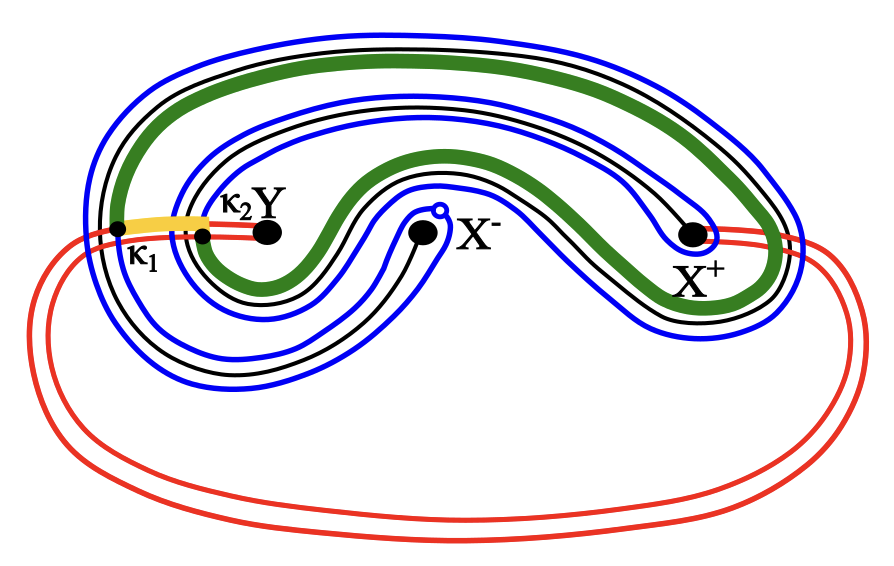}};
\node at (0,0) {$\xrightarrow{\text{isotopy}}$};
\node at (4.1,0) {\includegraphics[height=4cm, angle=0]{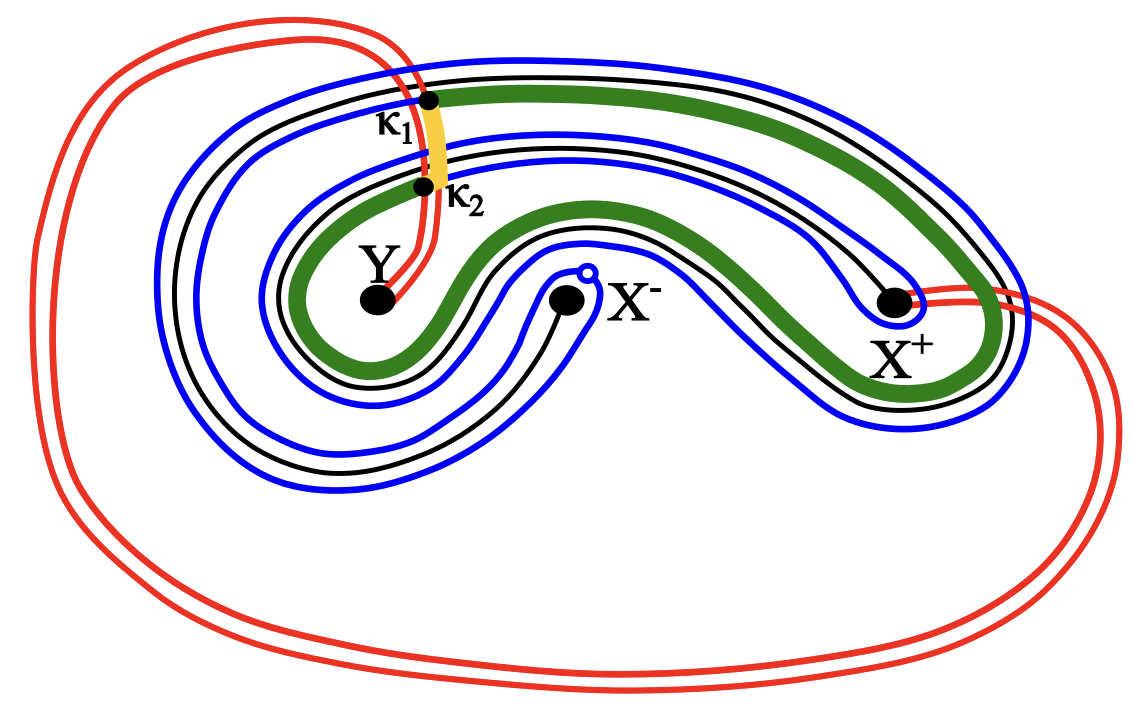}};
\end{tikzpicture}
\caption{An isotopy illustrating how $\gamma_{i,j}$ for $i,j\in X$ can be made parallel to $\lambda_{i,j}$. Shown is $\gamma_{1,2}$ for $K_{7/3}$. If $i,j\in Y$, bend the red Lagrangian the other way near $Y$.}
\label{7_3bendfig}  
\end{figure}

It remains to show that the (1,2)-block of $Q_{u/v}^K$ is $Q_{++}+L$, where $Q_{++}$ represents the $(1,1)$-block of $Q_{(u-v)/v}^T$. Suppose $\kappa_i,\kappa_j\in \mathfrak{A}$ and $\kappa_i',\kappa_j' \in \mathfrak{B}$ such that $i'-i=\frac{u-v}{2}$ and similarly for $j$ and $j'$, so that the pairs $(\kappa_i,\kappa_i')$ and $(\kappa_j,\kappa_j')$ are related as in Figure \ref{tangleknotpairs} or \ref{linformknotfig}. Assume $i<j$, which means $\xi_i\prec\xi_j$. It suffices to show that the $4\times 4$ submatrix of $Q$ with rows and columns indexed by $i,j,i'$, and $j'$ looks like

\[
\begin{blockarray}{ccccc}
    & i & j & i' & j' \\
    \begin{block}{c[cccc]}
        i & X+2 & A+2 & X & A \\
        j & A+2 & Y+2 & A+1 & Y\\
        i' & X & A+1 & X-1 & A-1\\
        j' & A & Y & A-1 & Y-1\\
    \end{block}
\end{blockarray}
\]

for some $X,Y,A\in \mathbb{Z}$. Equivalently, given what we have already have, it suffices to show 

\begin{enumerate}
   \item $Q_{i,i'}=Q_{i,i}-2=Q_{i',i'}+1$ 
   \item $Q_{j,j'}=Q_{j,j}-2=Q_{j',j'}+1$
   \item $Q_{i,j'}=Q_{i,j}-2=Q_{i',j'}+1$
   \item $Q_{j,i'}=Q_{i',j'}+2=Q_{i,j}-1$.
\end{enumerate}

The proof of (1) and (2) is the same, so we only show (1). To prove (1), we will just consider two cases, given by the two possible orientations of $\alpha_{(u-v)/v}$ along the arc between $\xi_i$ and $\xi_{i'}$. The important observation to make is that, in $M$, $\gamma_{i,i'}^K$ and $\gamma_{i',i}^K$ give loops around $X^+$ with opposite orientations, and in $\text{Conf}^2(M)$, we have $\Phi([\gamma_{i',i}^K])=0$ and $\Phi([\gamma_{i,i'}^K])=1$.

\begin{figure}[H]
    \centering
    \includegraphics[height=4.2cm]{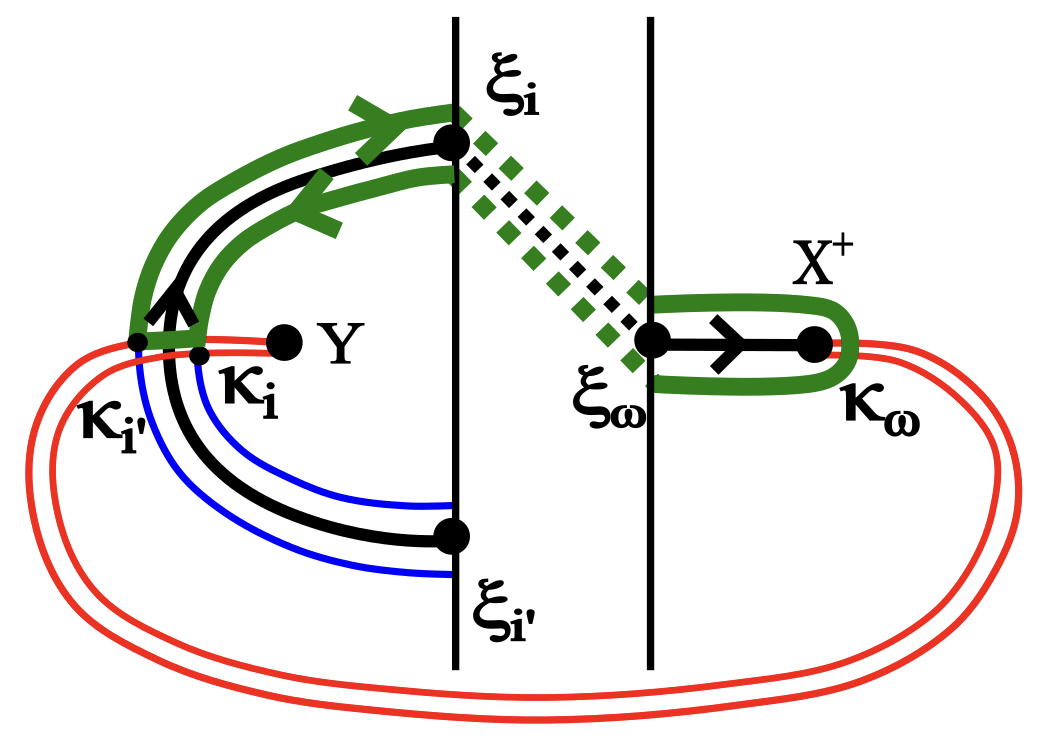} \qquad \qquad
    \includegraphics[height=4.2cm]{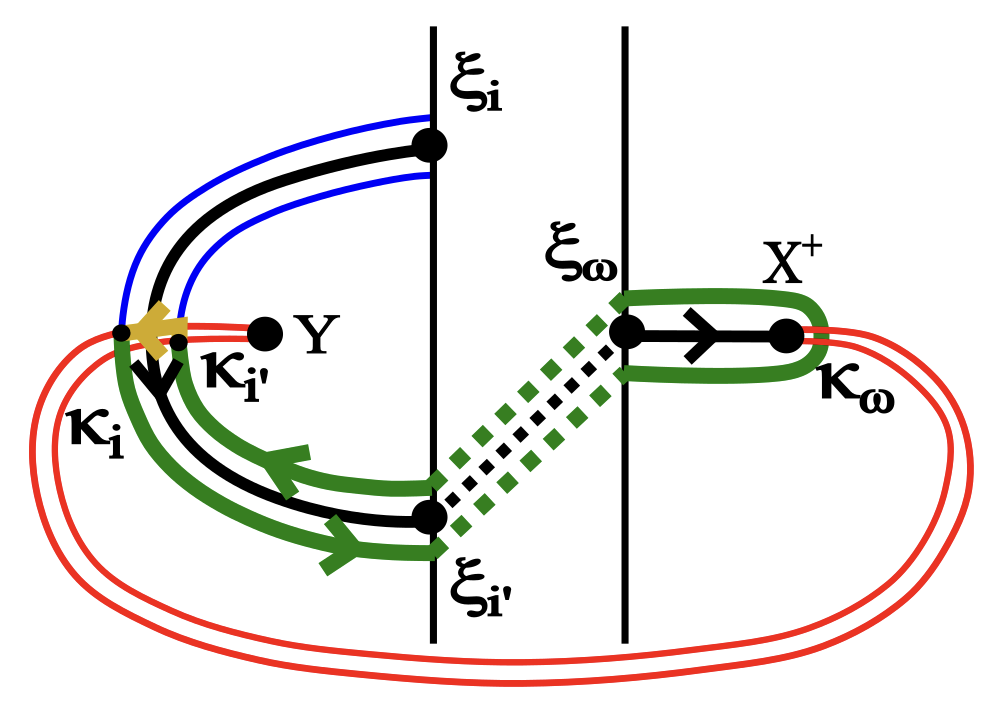}
    \caption{Some loops for proving (1) in Lemma \ref{Qplusblockknots}}
    \label{Qplusblocknotsfig1}
\end{figure}

In Figure \ref{Qplusblocknotsfig1}, we only drew $\gamma^K_{i',i}$ for the first orientation and $\gamma^K_{i,i'}$ for the second, but it is easily seen that the other cases are symmetric. It also follows that we have $\Psi_{X^+}([\gamma_{i',i}^K])=1$, $\Psi_{X^+}([\gamma_{i,i'}^K])=-1$, so (1) follows from
\begin{align*}
Q_{ii'} & =Q_{ii}+(\Phi-2\Psi_{X^+})([\gamma_{i',i}^K]) \\
& = Q_{i'i'}+(\Phi-2\Psi_{X^+})([\gamma_{i,i'}^K]).
\end{align*}

To prove (3) and (4), we just need to show $Q_{ij'}=Q_{ij}-2$ and $Q_{ji'}=Q_{i'j'}+2$ because we already know $Q_{ij}-Q_{i'j'}=3$. These follow easily if we consider the loops going against the orientation of $\alpha_{(u-v)/v}$. In particular, for $Q_{ij'}=Q_{ij}-2$, we use
\[
Q_{ij'}=Q_{ii}+(\Phi-2\Psi_{X^+})([\gamma_{j',i}^K]),
\]
and we observe that $\gamma_{j',i}^K\sim \sigma_{j',i}\cdot \gamma_{j,i}^K$, where $\sim$ denotes that the loops are homologous and $\sigma_{j',i}:[0,1]\to \text{Conf}^2(M)$ is a loop at the same basepoint as $\gamma_{j',i}$, which we now define. The loop may be written as $(\sigma_1,\sigma_2)$ with $\sigma_1$ being the loop that travels along the blue Lagrangian $\overline{\alpha_{(u-v)/v}}$ from $\kappa_{j'}$ to $\kappa_j$ on the same parallel copy of $\beta$ (the red Lagrangian) and then back to $\kappa_{j'}$ along $\beta$, so it wraps around $X^+$ once, and $\sigma_2$ is the constant loop at $\kappa_i$ on the other copy of $\beta$. We clearly have $\Phi([\sigma_{j',i}])=0$. Recall that $\Psi_{X^+}$ is a homomorphism defined on $\pi_1(M)$ rather than $\pi_1(\text{Conf}^2(M))$; consequently, we need a slightly modified $\sigma_{j',i}$ when working with this homomorphism. The natural choice is that we collapse $\beta_1$ and $\beta_2$ to $\beta$ (as in $\mathcal{D}(K_{u/v})$) and drop the constant loop $\sigma_2$ at $\kappa_i$. By how we have defined the intersection points with respect to the given orientation of $\alpha_{(u-v)/v}$, it is not difficult to see that $\Psi_{X^+}([\sigma_{j',i}])=1$. Thus, we get
\[
Q_{ij'}=Q_{ii}+(\Phi-2\Psi_{X^+})([\sigma_{j',i}\cdot \gamma_{j,i}^K])= Q_{ii}+(\Phi-2\Psi_{X^+})([ \gamma_{j,i}^K])-2=Q_{ij}-2.
\]
Similarly, we have 
\[
Q_{ji'}=Q_{i'i'}+(\Phi-2\Psi_{X^+})([\gamma_{j,i'}^K]),
\]
where $\gamma_{j,i'}^K \sim \sigma_{j,i'}\cdot \gamma_{j',i'}^K$, with $\sigma_{j,i'}$ being defined similarly to the previous case, but now we have $\Phi([\sigma_{j,i'}])=0$ and $\Psi_{X^+}([\sigma_{j,i'}])=-1$, so
\[
Q_{j,i'}=Q_{i',i'}+(\Phi-2\Psi_{X^+})([\sigma_{j,i'}\cdot \gamma_{j',i'}])= Q_{i',i'}+(\Phi-2\Psi_{X^+})([ \gamma_{j',i'}])+2=Q_{i',j'}+2,
\]
which concludes the proof.
\end{proof}

The following lemma extends the main step in the proof of Lemma \ref{Qplusblockknots} regarding the structure of the $4\times 4$ submatrix of $Q_{u/v}^K$ with rows and columns indexed by $\kappa_i, \kappa_{i'}, \kappa_j,$ and $\kappa_{j'}$.

\begin{lemmaN}
\label{4x4blocklem}
    For any $\kappa_i,\kappa_i',\kappa_j,$ and $\kappa_j'$ with the pairs $(\kappa_i,\kappa_i')$ and $(\kappa_j,\kappa_j')$ related as shown in Figure \ref{tangleknotpairs}, if $\xi_i\prec \xi_j$, then the $4\times 4$ submatrix of $Q_{u/v}^K$ with rows and columns indexed by these points is of the form
    \begin{equation}
    \label{4x4blockstructure}
\begin{blockarray}{ccccc}
    & i & i' & j & j' \\
    \begin{block}{c[cccc]}
        i & X+3 & X+1 & A+3 & A+1 \\
        j & X+1 & X & A+2 & A\\
        i' & A+3 & A+2 & Y+3 & Y+1\\
        j' & A+1 & A & Y+1 & Y\\
    \end{block}
\end{blockarray}
\end{equation}
for some $X,Y,A\in\mathbb{Z}$, up to conjugation by some permutation matrix.
\end{lemmaN}

Note that this statement allows the $i$- and $j$-labeled points to both come from the active/inactive side for $\tau_{(u-v)/v}$ or one can come from the active side and the other from the inactive side. We must include the ``up to conjugation'' statement because $Q_{u/v}^K$ is defined by the standard ordering of the $\kappa$ intersection points, which does not guarantee that the row and column indices be ordered in this particular way. In particular, the primed indices are do not necessarily come after the unprimed ones. 

\begin{proof}[Proof of Lemma \ref{4x4blocklem}]
    The lemma follows from the same argument used in the proof of Theorem \ref{Qplusblockknots} (there is no true distinction between active and inactive intersections for knots).
\end{proof}

Lemma \ref{Qplusblockknots} shows that the standard ordering of the $\kappa_i$ gives an exact agreement for the top left $2\times 2$ blocks of $Q_{u/v}^K$ and $Q_{u/v}$ for the $UP$ case (and the lower right $2\times 2$ block for the $RI$ case). Unfortunately, we do not have the same result for the other blocks. Next, we consider the $(3,3)$-block of $Q_{u/v}$ in the $UP$ case and the $(1,1)$-block for the $RI$ case. These blocks are equal to the $(2,2)$- and $(1,1)$-blocks of $Q_{(u-v)/v}^T$, respectively, without shifts. We will see that certain entries in these blocks for $Q_{(u-v)/v}^T$ can be permuted (in a way allowable by Theorem \ref{permthm}) to agree with the appropriate blocks in $Q_{u/v}^K$. 


\begin{lemmaN}
\label{tQmswap}
    Given $\tau_{(u-v)/v}$ with $UP$ orientation, then $Q_{(u-v)/v}^T\sim \widetilde{Q}_{(u-v)/v}^T$, where $\widetilde{Q}_{(u-v)/v}^T$ is obtained from $Q_{(u-v)/v}^T$ by applying all permutations $Q_{i,j'}\leftrightarrow Q_{i',j}$ and $Q_{j',i}\leftrightarrow Q_{j,i'}$ within $Q_{--}$, the $(2,2)$-block, where $\xi_i,\xi_{i'}, \xi_j,$ and $\xi_{j'}$ are related as in Figure \ref{tQmswapfig} with $\xi_i\prec \xi_j$. The result is analogous for the $RI$ orientation, but with the entries being swapped in $Q_{++}$, the $(1,1)$-block.
\end{lemmaN}

\begin{figure}
    \centering
    \includegraphics[height=5cm]{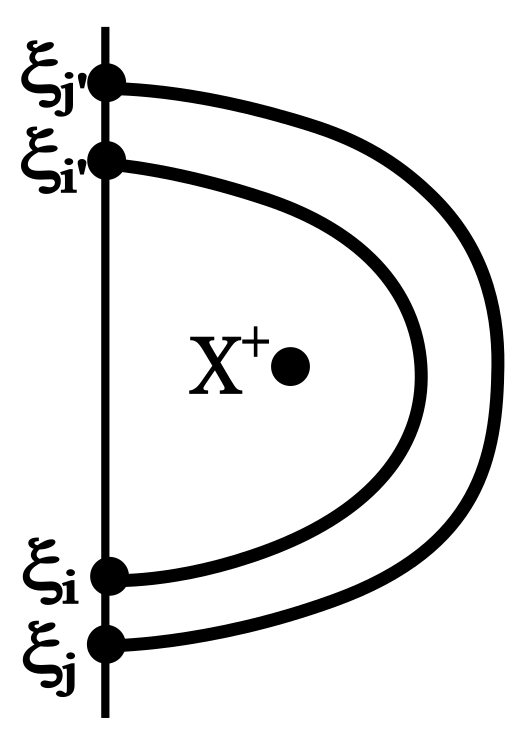} \qquad \qquad
    \includegraphics[height=5cm]{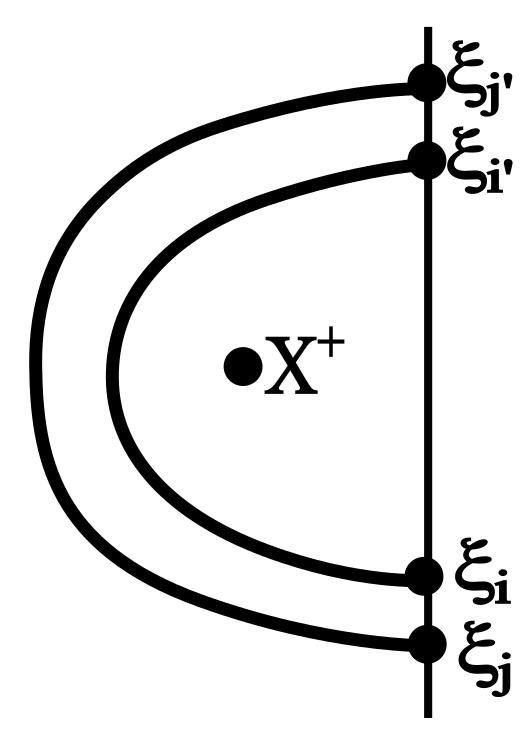}
    \caption{Labeling of intersection points for quiver permutation in Lemma \ref{tQmswap}. Left: $UP$ case; Right: $RI$ case}
    \label{tQmswapfig}
\end{figure}

Recall from Section \ref{bgrd} that $Q\sim Q'$ means  
\[
P_Q(x_1,...,x_u)\bigg|_{x_i\mapsto (-1)^{Q_{i,i}-S_i}q^{S_i-1}a^{A_i}x} = P_{Q'}(x_1,...,x_u)\bigg|_{x_i\mapsto (-1)^{Q'_{i,i}-S_i}q^{S_i-1}a^{A_i}x},
\]
i.e. that $Q$ and $Q'$ give the same generating functions for the colored HOMFLY-PT invariants after the appropriate change of variables. In order to prove Lemma \ref{tQmswap}, we will need to use Theorem \ref{permthm} from Section \ref{LQCsec}, which was proved by Jankowski, Kucharski, Larragu\'ivel, Noshchenko, and Su\l{}kowski in \cite{JKLN21}. Recall that, if $\Lambda_i=(-1)^{Q_{i,i}-S_i}q^{S_i-1}a^{A_i}$, Theorem \ref{permthm} says that $Q\sim Q'$ if $Q'$ is related to $Q$ by a sequence of transpositions of off-diagonal entries $Q_{ab}\leftrightarrow Q_{cd}$ and $Q_{ba}\leftrightarrow Q_{dc}$ such that $\Lambda_a\Lambda_b=\Lambda_c\Lambda_d$ and $|Q_{ab}-Q_{cd}|=1$, with 
\[
Q_{ai}+Q_{bi}=Q_{ci}+Q_{di}-\delta_{ci}-\delta_{di}
\]
for all $i$, if $Q_{ab}=Q_{cd}-1$.

Thus, Theorem \ref{permthm} states some sufficient conditions for permuting pairs of off-diagonal entries in $Q$ while still yielding the same generating function, so we just need to show that there is a sequence of these legal permutations.

\begin{proof}[Proof of Lemma \ref{tQmswap}]
To simplify notation, we will write $Q$ instead of $Q_{(u-v)/v}^T$ and $Q'$ instead of $\widetilde{Q}_{(u-v)/v}^T$ throughout this proof.

Because we are only permuting off-diagonal entries of $Q$ and nothing is done to $S$ and $A$, it is clear that the condition $\Lambda_i=\Lambda_i'$ is satisfied for all $i\in Q_0$. Furthermore, it is easy to see that $\Lambda_i\Lambda_{j'}=\Lambda_{i'}\Lambda_j$ for all pairs of entries we want to swap given how the four points are related geometrically.

Now, observe the general fact that, if $\xi_a,\xi_b,$ and $\xi_c$ are related as in the figure below on the left (so $\xi_c$ lies between $\xi_a$ and $\xi_b$ on $l_I$), then
\[
Q_{ac}-Q_{bc}=1
\]
and if they are related as in the figure on the right (which also holds if $\xi_c$ lies below $\xi_a$ and $\xi_b$ on $l_I$), then
\[
Q_{ac}-Q_{bc}=2.
\]
\[
\vcenter{\hbox{\includegraphics[height=4cm,angle=0]{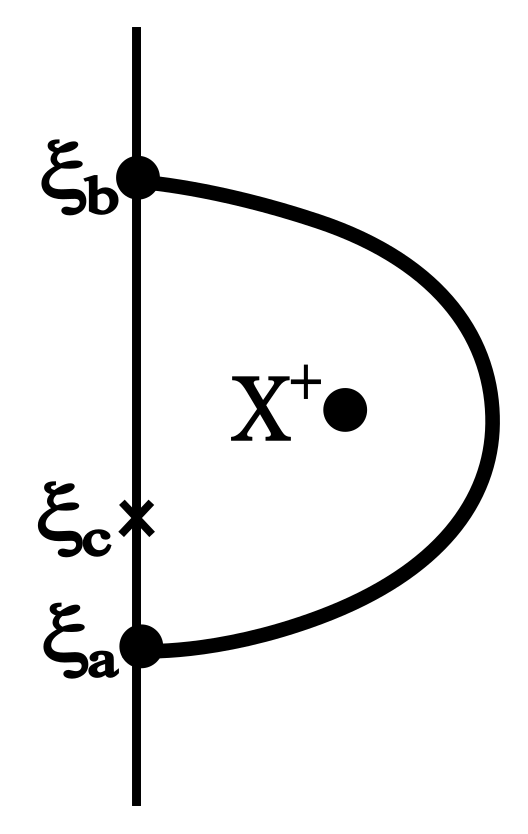}}} \qquad \qquad 
\vcenter{\hbox{\includegraphics[height=4cm,angle=0]{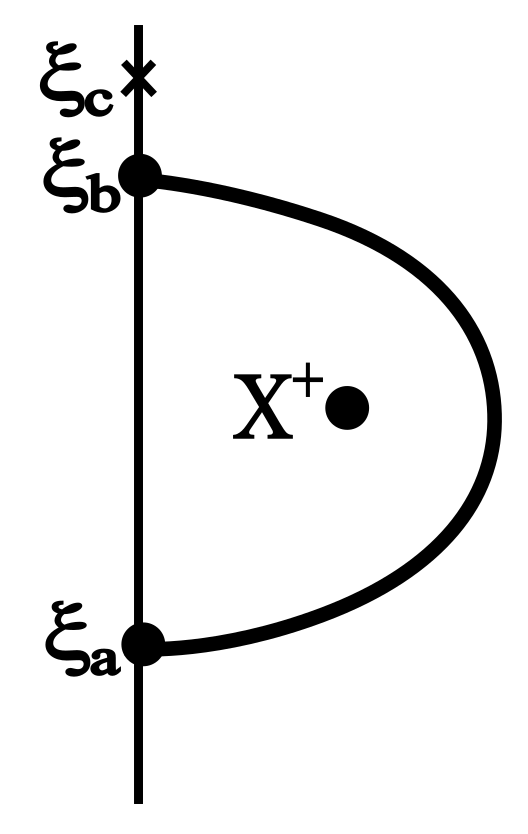}}}
\]
This follows from $\Psi_{X^+}([\gamma_{b,c}^T])-\Psi_{X^+}([\gamma_{a,c}^T])=1$ and $\Phi([\gamma_{b,c}^T])-\Phi([\gamma_{a,c}^T])$ is 1 in the first case and 0 in the second. 

From this observation, it follows that $Q_{i'j}=Q_{ij'}-1$ for all $\xi_i,\xi_{i'},\xi_j$, and $\xi_{j'}$ related as in Figure \ref{tQmswapfig}. Thus, to apply the swap $Q_{ij'}\leftrightarrow Q_{i'j}$ and $Q_{j'i}\leftrightarrow Q_{ji'}$, what we need to show is 
\begin{equation}
\label{tangpermeq}
Q_{i'k}+Q_{jk}=Q_{ik}+Q_{j'k}-\delta_{ik}-\delta_{j'k}
\end{equation}
for all $\xi_k$. This is trivial if $\xi_k$ is an active intersection point. Thus, we need to check that there is some sequence of the desired permutations such that Equation \ref{tangpermeq} holds at each step for all inactive $\xi_k$.

Following a quick computation involving our equations above for $Q_{ac}-Q_{bc}$, one finds that the $4\times 4$ submatrix of $Q_{--}$ with rows and columns indexed by $i,i',j,$ and $j'$ is of the form
\begin{equation}
\label{4x4submattangle}
\begin{blockarray}{ccccc}
    & i & i' & j & j' \\
    \begin{block}{c[cccc]}
        i & X+3 & X+1 & A+3 & A+2 \\
        i' & X+1 & X & A+1 & A\\
        j & A+3 & A+1 & Y+3 & Y+1\\
        j' & A+2 & A & Y+1 & Y\\
    \end{block}
\end{blockarray}
\end{equation}
for some $X,Y,A\in\mathbb{Z}$. From here, Equation \ref{tangpermeq} follows if $k\in\{i,i',j,j'\}$.

The figure below shows how $\tau_{(u-v)/v}$ looks to the right of $l_I$, where $n=\frac{v-1}{2}$ and we have only recorded the subscripts $i$ for the $\xi_i$, ordered from the inside-out.

\[
\vcenter{\hbox{\includegraphics[height=5cm,angle=0]{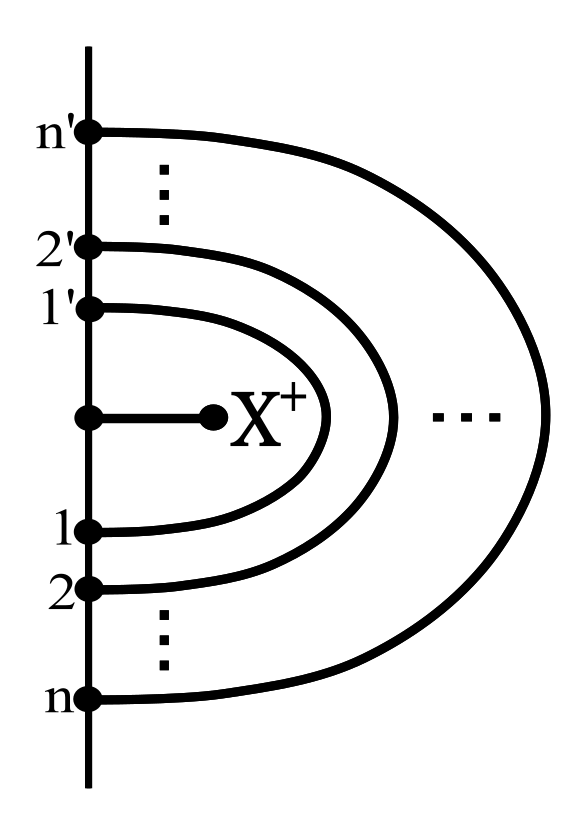}}}
\]

Now, it follows from the equations involving $Q_{ac}-Q_{bc}$ from the beginning of the lemma that
\begin{equation}
\label{QikmQi'k}
Q_{ik}-Q_{i'k}=2-\delta_{|k|<i}-\delta_{ki'},
\end{equation}
where $|k|=k$ if $k$ is unprimed and, otherwise, it just removes the prime (i.e. $|1|=|1'|=1$). Thus, assuming $j>i$, if we take the corresponding equation for $Q_{jk}-Q_{j'k}$ and subtract Equation \ref{QikmQi'k} from it, we get
\begin{equation}
\label{qpermeqndelta}
Q_{jk}+Q_{i'k}=Q_{j'k}+Q_{ik}-\delta_{i<|k|<j}-\delta_{kj'}-\delta_{ki}.
\end{equation}

To simplify notation further, denote the pair of permutations $Q_{j',i}\leftrightarrow Q_{j,i'}$ and $Q_{i,j'}\leftrightarrow Q_{i',j}$ by $(i,j)$, where $i<j$. Recall that we want to apply permutations $(i,j)$ to $Q_{--}$ such that $\xi_i\prec \xi_j$, so we will attach a superscript to each $\xi_i$ designating its ordering among the other (unprimed) inactive intersections with respect to $\prec$, and $\xi_{i'}$ will get the same superscript. In particular, $i$ and the superscript for $\xi_i$ are related by a bijective function $\upsilon$ on the set $[n]=\{1,...,n\}$ determined by $\prec$; for example, if $\xi_i\prec \xi_j$ for all $j\leq n$, $j\neq i$, then we write $\xi_i$ and $\xi_{i'}$ as $\xi_i^{\upsilon(i)}=\xi_i^1$ and $\xi_{i'}^{\upsilon(i)}=\xi_{i'}^1$. If we set $\mu=\upsilon^{-1}$, we can also write $\xi_i^{\upsilon(i)}=\xi_i^1=\xi_{\mu(1)}^1$. Thus, given $\xi_i$ and $\xi_j$, $\upsilon(i)<\upsilon(j)$ if and only if $\xi_i\prec\xi_j$; conversely, $\xi_{\mu(i)}\prec \xi_{\mu(j)}$ if and only if $i<j$.

Now, we describe the desired sequence of permutations $(i,j)$. Observe that the permutations we want to perform according to the lemma can be written in the form $(i,j)=(\mu(k),\mu(k)+l)$ where $\mu(k)<\mu(k)+l\leq n$ (this is the condition $i<j$) and $k<\upsilon(\mu(k)+l)$ (this is the condition $\xi_i\prec \xi_j$). The sequence of these $(\mu(k),\mu(k)+l)$ permutations is determined by a type of ``lexicographical order'' on the pairs $(k,l)$. We will describe this in further detail now.

First, start with $\xi_i^1=\xi_{\mu(1)}^1$. Since $i=\mu(1)$, we want all permutations $(i,j)$ for $j>i$, and the sequence we want is
\[
(\mu(1),\mu(1)+1)\rightarrow(\mu(1),\mu(1)+2)\rightarrow...\rightarrow(\mu(1),n).
\]
The first permutation $(\mu(1),\mu(1)+1)$, assuming $\mu(1)\neq n$, is allowed by Theorem \ref{permthm} and Equation \ref{qpermeqndelta} because $\delta_{i<|k|<j}=0$ for all $k$, and then this term is canceled at each step by the previous permutations. After all such $(i,j)$ permutations with $i=\mu(1)$ have been performed, one does a similar procedure with $\mu(2)$, but with permutations $(\mu(2),\mu(2)+l)$ omitted, where $\upsilon(\mu(2)+l)<2$. Note that this can only happen if $\mu(2)+k=\mu(1)$. Then, one proceeds in a similar manner for permutations of the form $(i,j)$
with $i=\mu(3)$, then $i=\mu(4)$, up through $i=\mu(n)$. At each step, the previous permutations perfectly cancel the $\delta_{i<|k|<j}$ term, so they are allowed in this order by Theorem \ref{permthm}. To see this, we need to check that this cancellation happens when the permutations are performed according to our rule, regardless of how $\xi_i,\xi_j,$ and $\xi_k$ are ordered with respect to $\prec$. To do so, it suffices to consider the case where $n=3$.

To illustrate, if $n=3$, we have
\[
\begin{blockarray}{ccccccc}
    & 1 & 1' & 2 & 2' & 3 & 3' \\
    \begin{block}{c[cccccc]}
        1 & X+3 & X+1 & A+3 & A+2 & B+3 & B+2\\
        1' & X+1 & X & A+1 & A & B+1 & B\\
        2 & A+3 & A+1 & Y+3 & Y+1 & C+3 & C+2\\
        2' & A+2 & A & Y+1 & Y & C+1 & C\\
        3 & B+3 & B+1 & C+3 & C+1 & Z+3 & Z+1\\
        3' & B+2 & B & C+2 & C & Z+1 & Z\\
    \end{block}
\end{blockarray}.
\]
before applying any permutations. If we assume $\xi_2\prec\xi_1\prec\xi_3$ (or $\xi_2^1\prec\xi_1^2\prec\xi_3^3$), then our desired sequence of permutations is $(2,3)\rightarrow (1,3)$; if $\xi_1^1\prec\xi_2^2\prec\xi_3^3$, then we would have $(1,2)\rightarrow(1,3)\rightarrow(2,3)$. Given any other possible ordering $\xi_{\mu(1)}^1\prec\xi_{\mu(2)}^2\prec\xi_{\mu(3)}^3$ of $\xi_1,\xi_2,$ and $\xi_3$ specified by a bijection $\mu:[3]\rightarrow[3]$, we can easily write out the sequence of permutations, and then it is not difficult to check that the $\delta_{i<|k|<j}$ terms vanish at each step. 

\end{proof}

Now we have two matrices $Q_{(u-v)/v}^T$ and $\widetilde{Q}_{(u-v)/v}^T$ representing the quadratic form in the quiver form for $P(\tau_{(u-v)/v})$, i.e. $Q_{(u-v)/v}^T\sim \widetilde{Q}_{(u-v)/v}^T$. Applying Lemma \ref{WSclosureformula} to $Q_{(u-v)/v}^T$ gives $Q_{u/v}$ and applying the lemma to $\widetilde{Q}_{(u-v)/v}^T$ gives another $u\times u$ matrix $\widetilde{Q}_{u/v}$. It is easy to see that $Q_{u/v}\sim \widetilde{Q}_{u/v}$ and that these matrices only differ in the $(3,3)$-block for the $UP$ case and the $(1,1)$-block for the $RI$ case. Consequently, Lemma \ref{Qplusblockknots} still holds if we replace $Q_{u/v}$ with $\widetilde{Q}_{u/v}$. 


This leads us to the next step in the proof of Theorem \ref{Knotsthm}, where we show that the $(3,3)$- or $(1,1)$-block of $\widetilde{Q}_{u/v}$ (depending on the orientation of $\tau_{(u-v)/v}$) agrees with the same block in $Q_{u/v}^K$.

\begin{lemmaN}
\label{Qmmagree}
   If $K_{u/v}=\text{Cl}(T\tau_{(u-v)/v})$ where $\tau_{(u-v)/v}$ has $UP$ orientation, then the $(3,3)$-block of $Q_{u/v}^K$ is equal to the $(3,3)$-block of $\widetilde{Q}_{u/v}$ after some re-ording of the points in $\mathfrak{C}$. If $\tau_{(u-v)/v}$ has $RI$ orientation, we have the analogous result for the $(1,1)$-block.
\end{lemmaN}

\begin{proof}
    As usual, we consider the $UP$ case, as the argument for the $RI$ case is the same.

    Consider a $4\times 4$ submatrix of $Q_{u/v}^K$ with rows and columns indexed by $\kappa_i,\kappa_i',\kappa_j,\kappa_j'\in \mathfrak{C}$. There is a corresponding $4\times 4$ submatrix of $Q_{u/v}$ coming from the $Q_{--}$ block of $\widetilde{Q}_{(u-v)/v}^T$ indexed by the points $\xi_i,\xi_i',\xi_j,\xi_j'$ paired with the $\kappa$'s; suppose that these intersection points are related as shown in Figure \ref{tQmswapfig}. By Lemma \ref{knotslinforms}, we know that the diagonal of these two $4\times 4$ matrices are equal, so we can easily relate them to each other by the calculations we have already made. By Lemma \ref{4x4blocklem}, this submatrix of $Q_{u/v}^K$ looks like
    \begin{equation}
    \label{4b1}
    \begin{blockarray}{ccccc}
    & i & i' & j & j' \\
    \begin{block}{c[cccc]}
        i & X+3 & X+1 & A+3 & \underline{A+1} \\
        i' & X+1 & X & \underline{A+2} & A\\
        j & A+3 & \underline{A+2} & Y+3 & Y+1\\
        j' & \underline{A+1} & A & Y+1 & Y\\
    \end{block}
\end{blockarray}
    \end{equation}
    if $\xi_i\prec\xi_j$, up to some reordering of the indices. Also, we know from (\ref{4x4submattangle}) in the proof of Lemma \ref{tQmswap} that the corresponding submatrix of $Q_{(u-v)/v}^T$ should look like 
    \begin{equation}
    \label{4b2}
    \begin{blockarray}{ccccc}
    & i & i' & j & j' \\
    \begin{block}{c[cccc]}
        i & X+3 & X+1 & A+3 & \underline{A+2} \\
        i' & X+1 & X & \underline{A+1} & A\\
        j & A+3 & \underline{A+1} & Y+3 & Y+1\\
        j' & \underline{A+2} & A & Y+1 & Y\\
    \end{block}
\end{blockarray}
    \end{equation}
    with respect to the same (reordering) of the indices, using the matching between $\xi$'s and $\kappa$'s. However, $\widetilde{Q}_{(u-v)/v}^T$ was defined to be the matrix obtained from $Q_{(u-v)/v}^T$ by applying all permutations $(i,j)$ in this block when $\xi_i\prec\xi_j$, which takes (\ref{4b2}) to (\ref{4b1}). The $(3,3)$-block of $\widetilde{Q}_{u/v}$ is equal to this permuted $Q_{--}$ block in $\widetilde{Q}_{(u-v)/v}^T$, so it follows that the $(3,3)$-blocks of $Q_{u/v}^K$ and $\widetilde{Q}_{u/v}$ are equal after some permutation of the row and column indices.

\end{proof}

Now, we have shown that all entries $Q_{ij}$ of $Q_{u/v}^K$ match the corresponding entries of a quiver already known to work for $K_{u/v}$ via the links-quivers correspondence, as long as $\kappa_i$ and $\kappa_j$ are both paired with active or inactive intersections for $\tau_{(u-v)/v}$. Thus, the final step needed to prove Theorem \ref{Knotsthm} is to show that the remaining entries of $Q_{u/v}^K$, computed by our theorem, are related to the $Q_{+-}$ blocks from Lemma \ref{WSclosureformula} by some ``legal'' (in the sense of Theorem \ref{permthm}) sequence of permutations. 

\begin{lemmaN}
\label{Qpmperm}
The two matrices $Q_{u/v}^K$ and $Q_{u/v}$ satisfy $Q_{u/v}^K\sim Q_{u/v}.$
\end{lemmaN}

\begin{proof}
We have already shown that $Q_{u/v}\sim \widetilde{Q}_{u/v}$ where $\widetilde{Q}_{u/v}$ has the same upper-left $2\times 2$ blocks and $(3,3)$-block as $Q_{u/v}^K$, up to conjugation by some permutation matrix. Now, we need to show
that we can legally permute entries in the remaining blocks of $\widetilde{Q}_{u/v}$ to obtain $Q_{u/v}^K$. Assume now that $\widetilde{Q}_{u/v}$ has been conjugated so that the $(3,3)$-block is equal to the same block for $Q_{u/v}^K$. This means that the elements of $\mathfrak{Z}$ are, once again, ordered by $\prec$.

The $3\times 3$ block matrix $\widetilde{Q}_{u/v}$ can be thought of as a quadratic form on $\mathbb{Z}\mathcal{G}_{u/v}^K$ or on $\mathbb{Z}\mathcal{G}_{(u-v)/v}^T$ with basis $\mathcal{G}_{(u-v)/v}^T=\mathfrak{X}\sqcup\mathfrak{Y}\sqcup\mathfrak{Z}$ (because of the matching $\mathfrak{A}\sqcup\mathfrak{B}\sqcup\mathfrak{C}\leftrightarrow \mathfrak{X}\sqcup\mathfrak{Y}\sqcup\mathfrak{Z}$). The matrix $\widetilde{Q}_{u/v}$ was obtained by applying Lemma \ref{WSclosureformula} to $\widetilde{Q}_{(u-v)/v}^T$, which is defined on the submodule with basis $\mathfrak{Y}\sqcup\mathfrak{Z}$, and the block for $\mathfrak{Y}$ doubled to account for $\mathfrak{X}$. The remaining blocks we need to consider in this proof are the $(1,3)$- and $(2,3)$-blocks (the $(3,1)$ and $(3,2)$ ones will follow by symmetry), which are the blocks relating active and inactive intersections for $\tau_{(u-v)/v}$. 

The remaining permutations of $\widetilde{Q}_{u/v}$ we need to consider are of the form $(i,j)$ for $\xi_j$ inactive and $\xi_i$ active, with $\xi_j\prec \xi_i$. Below, we establish our notational convention for the relevant $\xi$'s from $\tau_{(u-v)/v}$, where we have only written the subscripts. Note that there are four different ways of pairing the intersection points for the knot with the ones from the tangle due to the two possible orientations on each of the black arcs; thus, we do not label the $\kappa$'s to avoid unnecessary clutter, but they can easily be determined using Figure \ref{tangleknotpairs} for each orientation.

\[
\vcenter{\hbox{\includegraphics[height=6cm,angle=0]{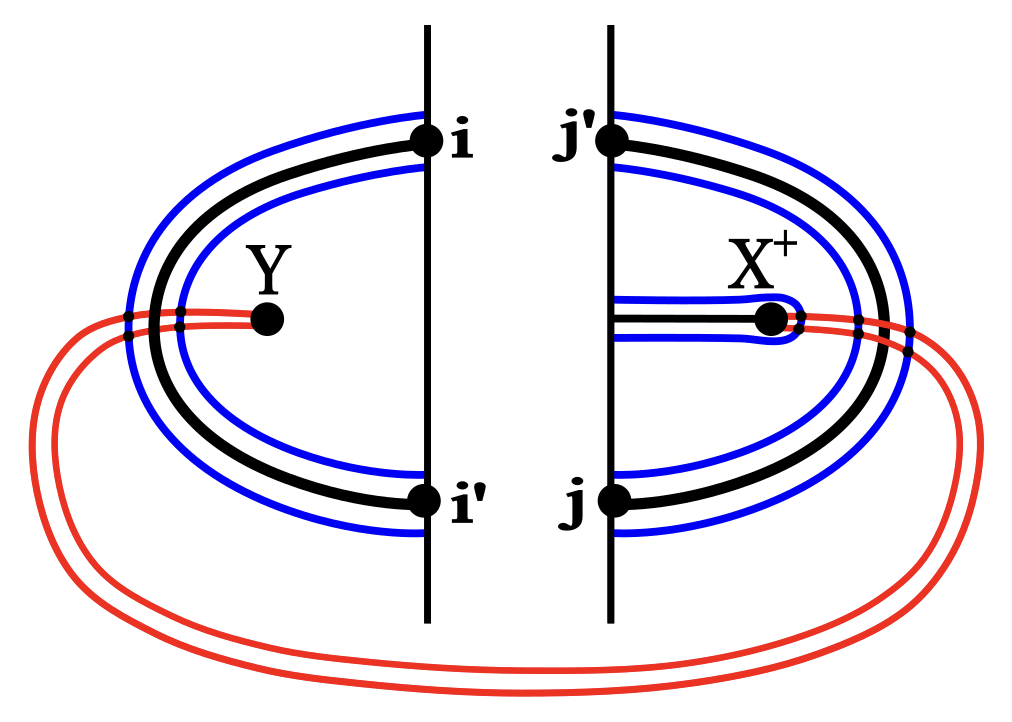}}}
\]

For $\tau_{(u-v)/v}$, we only have the intersections $\xi_i\in \mathfrak{Y}$ and  $\xi_j,\xi_{j'}\in \mathfrak{Z}$ contributing to $\widetilde{Q}_{(u-v)/v}^T$. The corresponding $3\times 3$ submatrix of $\widetilde{Q}_{(u-v)/v}^T$ and the $4\times 4$ submatrix of $\widetilde{Q}_{u/v}$ are related by applying $\text{Cl}(T-)$, which gives
\begin{equation}
\begin{blockarray}{cccc}
    & i & j & j' \\
    \begin{block}{c[ccc]}
        i & X & A+2 & A\\
        j & A+2 & Y+2 & Y\\
        j' & A & Y & Y-1\\
    \end{block}
\end{blockarray}
\xrightarrow{Cl(T-)}
\begin{blockarray}{ccccc}
    & i & i' & j & j' \\
    \begin{block}{c[cccc]}
        i & X+2 & X & A+2 & A \\
        i' & X & X-1 & A+1 & A-1\\
        j & A+2 & A+1 & Y+2 & Y\\
        j' & A & A-1 & Y & Y-1\\
    \end{block}
\end{blockarray}
\end{equation}
by Lemma \ref{WSclosureformula}. Once again, we can invoke Lemma \ref{4x4blocklem} to compare the matrix on the right with the corresponding $4\times 4$ submatrix of $Q_{u/v}^K$ with rows and columns indexed by $\kappa_i,\kappa_i',\kappa_j,$ and $\kappa_j'$. In particular, it should look the way it is shown if $\xi_i\prec \xi_j$, and the $A$ and $A+1$ terms should be swapped if $\xi_i\succ \xi_j$. This confirms that the permutations we need are indeed of the form $(i,j)$ for $\xi_i$ active, $\xi_j$ inactive and $\xi_j\prec \xi_i$.

Now that we know which permutations to apply, we need to show that there is a sequence of them permitted by Theorem \ref{permthm}. In other words, at each step of the sequence, in order to apply $(i,j)$, we need
\begin{equation}
\label{Qswapequation}
Q_{jk}+Q_{i'k}=Q_{j'k}+Q_{ik}-\delta_{kj'}-\delta_{ki}
\end{equation}
for each $k$, since $Q_{i'j}=Q_{ij'}-1$ when computed by our Theorem \ref{Knotsthm} and we are applying the permutation to make the $4\times 4$ block agree with the one on the right above. 

By a similar argument to the one in the proof of Lemma \ref{tQmswap}, one sees that the order of permutations can be determined by a rule similar to the lexicographical order: the permutation $(i,k)$ should be performed before $(j,l)$ if $i>j$, and if $i=j$, then $k<l$. This can be proved by considering the $8\times 8$ submatrix of $\widetilde{Q}_{u/v}$ with rows and columns indexed by $\xi_i,\xi_i',\xi_j,\xi_j',\xi_k,\xi_k',\xi_l,$ and $\xi_l'$ where the first four points are active intersections and the last four are inactive intersections (so $i,j,k,$ and $l$ are all distinct). We assume that we need to apply the permutations $(i,k)$ and $(j,l)$, with $\xi_i\prec\xi_j$ and $\xi_k\prec \xi_l$. In this case, we must have the permutation $(j,k)$ as well, and we may or may not need the $(i,l)$ permutation. By Lemma \ref{4x4blocklem}, this $8\times 8$ matrix must have the form 
\begin{equation}
    \begin{blockarray}{ccccccccc}
    & i & i' & j & j' & k & k' & l & l' \\
    \begin{block}{c[cccccccc]}
        i & W+2 & W & A+2 & \underline{A} & B+2 & \underline{B} & C+2 & \underline{C}\\
        i' & W & W-1 & \underline{A+1} & A-1 & \underline{B+1} & B-1 & \underline{C+1} & C-1\\
        j & A+2 & \underline{A+1} & X+2 & X & D+2 & \underline{D} & E+2 & \underline{E}\\
        j' & \underline{A} & A-1 & X & X-1 & \underline{D+1} & D-1 & \underline{E+1} & E-1\\
        k & B+2 & \underline{B+1} & D+2 & \underline{D+1} & Y+2 & Y & F+2 & \underline{F}\\
        k' & \underline{B} & B-1 & \underline{D} & D-1 & Y & Y-1 & \underline{F+1} & F-1\\
        l & C+2 & \underline{C+1} & E+2 & \underline{E+1} & F+2 & \underline{F+1} & Z+2 & Z\\
        l' & \underline{C} & C-1 & \underline{E} & E-1 & \underline{F} & F-1 & Z & Z-1\\
    \end{block}
\end{blockarray}
\end{equation}
with all values in $\mathbb{Z}$. To consider the most general case, we assume that all four permutations must be applied, and the entries being permuted are underlined, as usual. Using the same notation as in the proof of Lemma \ref{tQmswap} for indicating the order of permutations, one can easily compute that 
\[
\begin{tikzcd}[sep=small]
(j,k) \arrow[r]\arrow[d] & (i,k) \arrow[d]\\
(j,l) \arrow[r] & (i,l)\\
\end{tikzcd}
\]
is required in order to guarantee that (\ref{Qswapequation}), with the appropriate indices, holds at each step. We have the freedom to choose whether we do $(j,l)$ or $(i,k)$ first, and we can determine the ordering specified above by choosing to do $(j,l)$ first.

The lemma follows since we have shown that the entries of $\widetilde{Q}_{u/v}$ can be permuted to agree with $Q_{u/v}^K$, so $Q_{u/v}\sim\widetilde{Q}_{u/v}\sim Q_{u/v}^K$. This also concludes the proof of Theorem \ref{Knotsthm}.
\end{proof}

To finish the proof of Theorem \ref{Knotsthm}, we must now consider the correction terms that we need to guarantee
\[
S_i-Q_{ii}-2A_i=\sigma(K_{u/v})
\]
for all $1\leq i \leq u$.

\subsection{Step 3: The Correction Terms}
\label{cortermsec}


The correction terms in Theorem \ref{Knotsthm} are computed by considering the inductive procedure for constructing $K_{u/v}$, i.e. the twists in $K_{u/v}=\text{Cl}(T^{a_{2k+1}}R^{a_{2k}}...R^{a_2}T^{a_1}\tau_{0/1})$. In particular, we will need to count the number of twists of type $TX$ and $RX$ for $X\in\{UP,OP,RI\}$. 

In \cite{SW21}, Sto\v si\'c and Wedrich showed how to put $P(K_{u/v})$ in quiver form by first studying how the twist operations affect $S,A,$ and $Q$ for $P(\tau_{u'/v'})$ at each step of the inductive procedure, and then applying $\text{Cl}(T-)$ after reaching $P(\tau_{(u-v)/v})$. Importantly, they used the appropriately scaled twist rules to guarantee that one ends with 
\[
S_i-Q_{ii}-2A_i=\sigma(K_{u/v}).
\]
The first thing we need to do is compute the twist rules at the matrix level, using Theorem \ref{Knotsthm}. More precisely, if $P(\tau_{u'/v'})$ is represented by the data $X, [K|S|A], Q$, then we need to determine how this data changes when applying $T$ or $R$ to $\tau_{u'/v'}$. 

Before doing this, we must consider one additional detail. The formulas in Proposition \ref{tangleHOMFLYpoly} and \ref{indpartprop} rely upon whether $\xi_\omega$ is active or inactive. If a tangle has $UP$ orientation, we know $\xi_\omega$ is inactive, and if it has $RI$ orientation, we know $\xi_\omega$ must be active. However, if a tangle has $OP$ orientation, then $\xi_\omega$ could be active or inactive depending on whether $u> v$ or $u<v$. 

\begin{definitionN}
    We say $\tau_{u/v}$ has $OP^+$ orientation if it has $OP$ orientation and $\xi_\omega$ is active (i.e. $u>v$). Analogously, $\tau_{u/v}$ has $OP^-$ orientation if it has $OP$ orientation and $\xi_\omega$ is inactive (i.e. $u<v$).
\end{definitionN}

Thus, in the lemma below, we need consider the (matrix-level) twist rules for $TUP$, $TRI$, $TOP^+$, $TOP^-$, $RUP$, $RRI$, $ROP^+,$ and $ROP^-$. We already know how to compute everything correctly, up to some shift, and the shift in $Q$ is determined by the diagonal, so we only need to work with the linear forms.

\begin{lemmaN}
\label{linformtwists}
    If $P(\tau_{u'/v'})$ is put in (almost) quiver form and is given by the orientation and linear form data 
    \[
X, 
    \left[\begin{array}{ c| c | c | c  }
    K_+ & S_+ & A_+ & Q_+\\
    \hline
    K_- & S_- & A_- & Q_- 
  \end{array}\right],
\]
then the effects of applying $T$ and $R$, according to Propsition \ref{tangleHOMFLYpoly} or \ref{indpartprop}, depend on $X$ by the following rules.

\begin{align*}
&TUP, 
    \left[\begin{array}{ c| c | c | c}
    K_+ & S_++1 & A_+ & Q_++1\\
    \hline
    K_- & S_-+1 & A_- & Q_-\\
    \hline
    K_- & S_- & A_- & Q_-
  \end{array}\right]
  \qquad
  &TOP^+, 
    \left[\begin{array}{ c| c | c | c}
    K_+ & S_+ & A_+ & Q_+\\
    \hline
    K_- & S_- & A_--1 & Q_-+1\\
    \hline
    K_- & S_--1 & A_--1 & Q_-+1
  \end{array}\right] 
\\
 &TRI, 
    \left[\begin{array}{ c| c | c | c}
    K_+ & S_+ & A_+ & Q_+\\
    \hline
    K_- & S_- & A_- & Q_--1\\
    \hline
    K_- & S_--1 & A_--1 & Q_-+1
  \end{array}\right] 
  \qquad
 &TOP^-, 
    \left[\begin{array}{ c| c | c | c}
    K_+ & S_+ & A_++1 & Q_+-1\\
    \hline
    K_- & S_- & A_- & Q_-\\
    \hline
    K_- & S_--1 & A_- & Q_-
  \end{array}\right] 
  \end{align*}
  \begin{align*}
  &RUP, 
    \left[\begin{array}{ c| c | c | c}
    K_+ & S_++1 & A_++1 & Q_+-1\\
    \hline
    K_+ & S_+ & A_+ & Q_++1\\
    \hline
    K_- & S_- & A_- & Q_-
  \end{array}\right]
  \qquad
  &ROP^+, 
    \left[\begin{array}{ c| c | c | c}
    K_+ & S_++1 & A_+ & Q_+\\
    \hline
    K_+ & S_+ & A_+ & Q_+\\
    \hline
    K_- & S_- & A_--1 & Q_-+1
  \end{array}\right] 
\\
 &RRI, 
    \left[\begin{array}{ c| c | c | c}
    K_+ & S_++1 & A_+ & Q_++1\\
    \hline
    K_+ & S_+ & A_+ & Q_++1\\
    \hline
    K_- & S_- & A_- & Q_-
  \end{array}\right] 
  \qquad
 &ROP^-, 
    \left[\begin{array}{ c| c | c | c}
    K_+ & S_++1& A_++1 & Q_+-1\\
    \hline
    K_+ & S_+ & A_++1 & Q_+\\
    \hline
    K_- & S_- & A_- & Q_-
  \end{array}\right]  
\end{align*}
\end{lemmaN}

\begin{proof}
    We prove the lemma for the $TUP$ case. The proofs for the remaining cases are similar.
    
    Suppose $\xi_i$ is an intersection point for $\tau_{u'/v'}$ such that $S,A,$ and the diagonal of $Q$ are computed in terms of the loop $\gamma_i:[0,1]\to M$. If $\xi_i$ is active, then applying the top twist $T$ takes $\xi_i$ to a new active intersection point $\xi_i'$ for $\tau_{(u'+v')/v'}$, and we easily see that $\Psi_V([\gamma_{i'}])-\Psi_V([\gamma_i])=1$ if $V=Z$, the middle of the three points, which is either $X^-$ or $Y$, and it is $0$ otherwise. If $\xi_i$ was inactive, then applying $T$ results in $\xi_i$ doubling to $\xi_i'$ and $\xi_i''$ for $\tau_{(u'+v')/v'}$, where $\xi_i'$ is now active and $\xi_i''$ is still inactive (we can think of $\xi_i''$ as being the same point as $\xi_i$). In this case, we have $\Psi_V([\gamma_{i'}])-\Psi_V([\gamma_i])$ and $\Psi_V([\gamma_{i''}])-\Psi_V([\gamma_i])$ are $0$ for all $V\in\{X^+,X^-,Y\}$. The picture below helps illustrate what we have just discussed.
    \[
    \begin{tikzpicture}
        \node at (0,0) {\includegraphics[height=2cm]{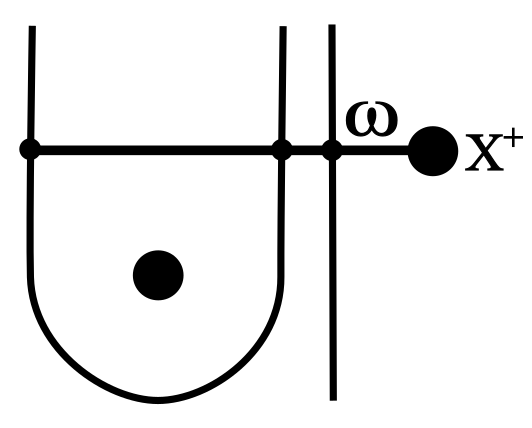}};
        \node at (2,0) {$\longrightarrow$};
        \node at (4,0) {\includegraphics[height=2cm]{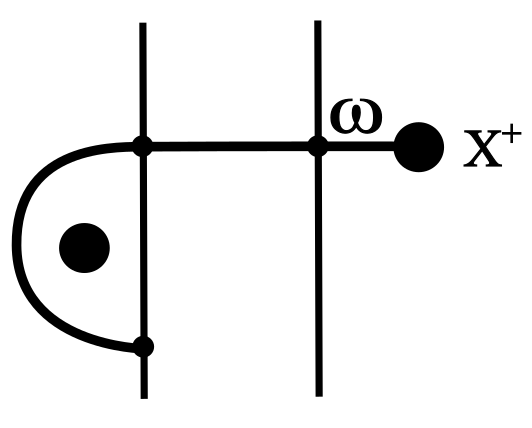}};
    \end{tikzpicture}
    \]
    Finally, due to the $\delta_{i,A}\delta_{\omega,I}$ term in $S_i$, we get the $+1$ in the $S_-+1$ block for the new active intersections obtained from the inactive points of $\tau_{u'/v'}$ after applying $T$. 
   
\end{proof}

These twist rule formulas can be compared directly with Proposition 3.5 from \cite{SW21}. This leads to the following lemma.

\begin{lemmaN}
\label{linformshifts}
    Suppose $P(\tau_{(u-v)/v})$ has been put in quiver form in two ways: the first, by following Proposition 3.5 of \cite{SW21} with linear form data $[K',S',A',Q']$, and the second following Proposition \ref{indpartprop} with linear form data $[K,S,A,Q]$. Then, up to some reordering of $\mathcal{G}_{u/v}^T$,
    \begin{align*}
        K'&=K\\
        S'&= S+\#\{TOP^\pm,TRI\}-\#\{TUP,RUP,ROP^\pm\}\\
        A' &= A+\#\{TRI,TOP^+\}-\#\{ROP^-,RUP\}\\
        Q' &= Q+\#\{RUP,ROP^-\}-\#\{TOP^+,TRI,RRI\}.
    \end{align*}
\end{lemmaN}

Naturally, the notation in the lemma indicates the number of twists of certain types used to build $\tau_{(u-v)/v}$. Also, we are using $Q$ and $Q'$ above to denote the diagonal of the matrices representing the quadratic forms, but we know that shifting the diagonal shifts the whole matrix.

\begin{proof}[Proof of Lemma \ref{linformshifts}]
    This can easily be shown by comparing Lemma \ref{linformtwists} with Proposition 3.5 in \cite{SW21}.
\end{proof}

Propositions \ref{tangleHOMFLYpoly} and \ref{indpartprop} guarantee that $S_\omega=A_\omega=Q_{\omega\omega}=0$, so we get
\begin{align*}
S_\omega'-Q'_{\omega\omega}-2A_\omega'=\#\{TOP^-,RRI\}-\#\{TUP,ROP^+\}
\end{align*}
for $\tau_{(u-v)/v}.$ If we call this value $\delta$, we have that $S_i'-Q_{ii}'-2A_i'=\delta$ for all $\xi_i$ on the same (in)active axis as $\xi_\omega$, and the intersection points on the opposite axis give $S_i'-Q_{ii}'-2A_i'=\delta\pm1$, depending on whether $\xi_\omega$ is active or inactive. However, it can easily be checked that applying $\text{Cl}(T-)$ by Lemma \ref{knotslinforms} results in $S_i'-Q_{ii}'-2A_i'=\delta$ for all $\kappa_i$ if we shift $S,A,$ and $Q$ for $K_{u/v}$ by the same amounts as $S,A,$ and $Q$ for $\tau_{(u-v)/v}$ in Lemma \ref{linformshifts}. This inspires the following.

\begin{definitionN}
\label{correctiontermsdef}
    We define the three correction terms in Theorem \ref{Knotsthm} as follows:
    \begin{align}
    &\mu_1(K_{u/v})=\#\{TOP^\pm,TRI\}-\#\{TUP,RUP,ROP^\pm\}\\
    &\mu_2(K_{u/v})=\#\{TRI,TOP^+\}-\#\{ROP^-,RUP\}\\
    &\mu_3(K_{u/v})=\#\{RUP,ROP^-\}-\#\{TOP^+,TRI,RRI\},
\end{align}
where these formulas are counting the twists used to build $\tau_{(u-v)/v}$, not $\tau_{u/v}$.
\end{definitionN}

Now, we drop the primed notation and assume we are working with the linear forms of Theorem \ref{Knotsthm} shifted by these amounts so that 
\[
S_i-Q_{ii}-2A_i=\#\{TOP^-,RRI\}-\#\{TUP,ROP^+\}.
\]
Thus, it remains to show the following lemma.

\begin{lemmaN}
    For a rational knot $K_{u/v}=\text{Cl}(T\tau_{(u-v)/v})$, 
    \[
    \sigma(K_{u/v})=\#\{TOP^-,RRI\}-\#\{TUP,ROP^+\},
    \]
    where we are counting the twists for $\tau_{(u-v)/v}.$
\end{lemmaN}

\begin{proof}
As stated in \cite{SW21}, it can easily be shown using the methods of \cite{GL78} and \cite{QAQ14} that
\begin{equation}
\label{signatureeq}
\sigma(K_{u/v})=1+\#RRI-\#TUP,
\end{equation}
where the formula is counting the twists for the tangle $\tau_{u/v}$, in contrast to $\tau_{(u-v)/v}$.

It is not difficult to see that $TOP^-$ is equivalent to $(TR)UP$ and $ROP^+$ is equivalent to $(RT)RI$. Recall the hexagonal diagram of states from Lemma \ref{twistdiagram}. We can add the vertical edges for $TOP^-$ and $ROP^+$, and we will mark the edges according to whether they contribute a $+1$ or $-1$ according to the expression $\#\{TOP^-,RRI\}-\#\{TUP,ROP^+\}$, as shown below.

\[\begin{tikzcd}[sep=small]
   & \underline{(UP,\,\,\,\, Y\vert X^-\vert X^+)} \arrow[rr,red,"T"]\arrow[ld]\arrow[dd,blue,bend left, "+1"] && (UP,\,\,\,\, X^-\vert Y\vert X^+) \arrow[ll,red,"-1"]\arrow[rd,"R"]\arrow[dd,blue,bend left, "+1"] &\\
   (OP,\,\,\,\, Y\vert X^+\vert X^-) \arrow[ru,"R"]\arrow[rd]&&&& (OP,\,\,\,\,X^-\vert X^+\vert Y) \arrow[lu]\arrow[ld,"T"]\\
    & (RI,\,\,\,\, X^+\vert Y\vert X^-) \arrow[lu,"T"]\arrow[rr,blue,"+1"]\arrow[uu,red,bend left,"-1"] && \underline{(RI,\,\,\,\, X^+\vert X^-\vert Y)} \arrow[ll,blue,"R"]\arrow[ru]\arrow[uu,red,bend left,"-1"] &
\end{tikzcd}\]
We have also underlined the two possible states for $\tau_{(u-v)/v}$, given the conditions of Lemma \ref{ratlinkslem}, such that $\text{Cl}(T\tau_{(u-v)/v})$ gives us $K_{u/v}$. 

First, suppose $\tau_{(u-v)/v}$ has $(UP, Y|X^-|X^+)$ state, which is the same as $\tau_{0/1}$. We can think of the hexagonal graph above as having an ``upstairs'' portion given by the $UP$ states and  ``downstairs'' portion given by the $RI$ states. The sequence of top and right twists used to build $\tau_{(u-v)/v}$ determines a path in the graph, where the only edges contributing to $\#\{TOP^-,RRI\}-\#\{TUP,ROP^+\}$ are the ones within the upstairs and downstairs portions, or the ones passing between the two levels. Because the path associated with $\tau_{(u-v)/v}$ starts and ends at the same state in the upstairs portion, all vertical arrows must cancel. Thus, we have $\delta=\#RRI-\#TUP$, where we are counting the twists used to build $\tau_{(u-v)/v}$. Since $\tau_{u/v}$ is obtained by applying one more twist of type $TUP$ to $\tau_{(u-v)/v}$ and Equation (\ref{signatureeq}) is counting the twists used to build $\tau_{u/v}$, the lemma follows for this case.

The case where $\tau_{(u-v)/v}$ has $(RI, X^+|X^-|Y)$ state is very similar. The main difference is that the $+1$ in Equation (\ref{signatureeq}) now comes from the fact that the path corresponding to $\tau_{(u-v)/v}$ must go downstairs one more time than it goes upstairs.
\end{proof}

We have shown that, once we take the shifted $S,A,$ and $Q$ from Theorem \ref{Knotsthm}, we get that all terms are $\delta$-homogeneous with $\delta=\sigma(K_{u/v})$, as desired. Following Sto\v si\'c and Wedrich in \cite{SW21}, it follows that
\[
\mathcal{P}^{\bigwedge^1}_{K_{u/v}}(q,a,t)=\sum_{i=1}^uq^{S_i+Q_{ii}}a^{A_i}t^{-Q_{ii}}
\]
up to some framing shift, where $\mathcal{P}^{\bigwedge^1}_{K_{u/v}}$ is the Poincar\'e polynomial for the (uncolored) HOMFLY-PT homology of $K_{u/v}$ and $t$ denotes the homological grading. Now, we can rephrase Conjecture \ref{Poincarepolyconj} from the introduction for rational knots as follows.

\begin{conjectureN}
    For $K_{u/v}$ a rational knot with $S,A,$ and $Q$  computed by Theorem \ref{Knotsthm}, we have
    \begin{equation}
        \sum_{j\geq 0} \mathcal{P}^{\bigwedge^j}_{K_{u/v}}(q,a,t)x^j=\sum_{\textbf{d}=(d_1,...,d_u)\in\mathbb{N}^u}q^{S\cdot\textbf{d}+\textbf{d}\cdot Q\cdot\textbf{d}^T}a^{A\cdot\textbf{d}}t^{T\cdot \textbf{d}}{d_1+...+d_u\brack d_1,...,d_u}x^{d_1+...+d_u}
    \end{equation}
    where $\mathcal{P}^{\bigwedge^j}_{K_{u/v}}(q,a,t)$ is the Poincar\'e polynomial for the $\bigwedge^j$-colored HOMFLY-PT homology, and $T=(-Q_{11},...,-Q_{uu})\in \mathbb{Z}^u$.
\end{conjectureN}

\subsection{Example: Rational Torus Knots}
\label{rattorusknots}

The rational torus knots, $K_{(2n+1)/1}=T(2,2n+1)$, are a convenient family of knots to study, and we can provide clean formulas for $Q_{2n+1/1}^K$ with minimal work. In fact, the formula follows immediately from Lemma \ref{WSclosureformula} and Example 6.3 from \cite{JHpI26} because none of the complications with inactive indices occur for this particular family of knots.

In \cite{JHpI26}, it was shown that the generating function data for $P(\tau_{n/1})$ with $n$ even is computed by
\[
    \begin{cases}
        K_i = \delta_{i\leq n/2}\\
        S_i=n-2i+1\\
        A_i=0\\
        \begin{cases}
            Q_{ij}=n-2k, \qquad & \text{if} \,\, (i,j)\in\mathcal{I}_k, \,1 \leq k \leq n/2\\
            Q_{ij}=0,\qquad & \text{if} \,\, (i,j)\in\mathcal{I}_k,\, k=n/2+1,
        \end{cases}
    \end{cases}
    \]
where
\[
\mathcal{I}_k=\{(i,j)\in\mathbb{N}^2 : i=k \, \text{and} \, j\leq k, \text{or} \, j=k \, \text{and} \, i\leq k\}.
\]

\begin{figure}
    \centering
    \includegraphics[height=5cm]{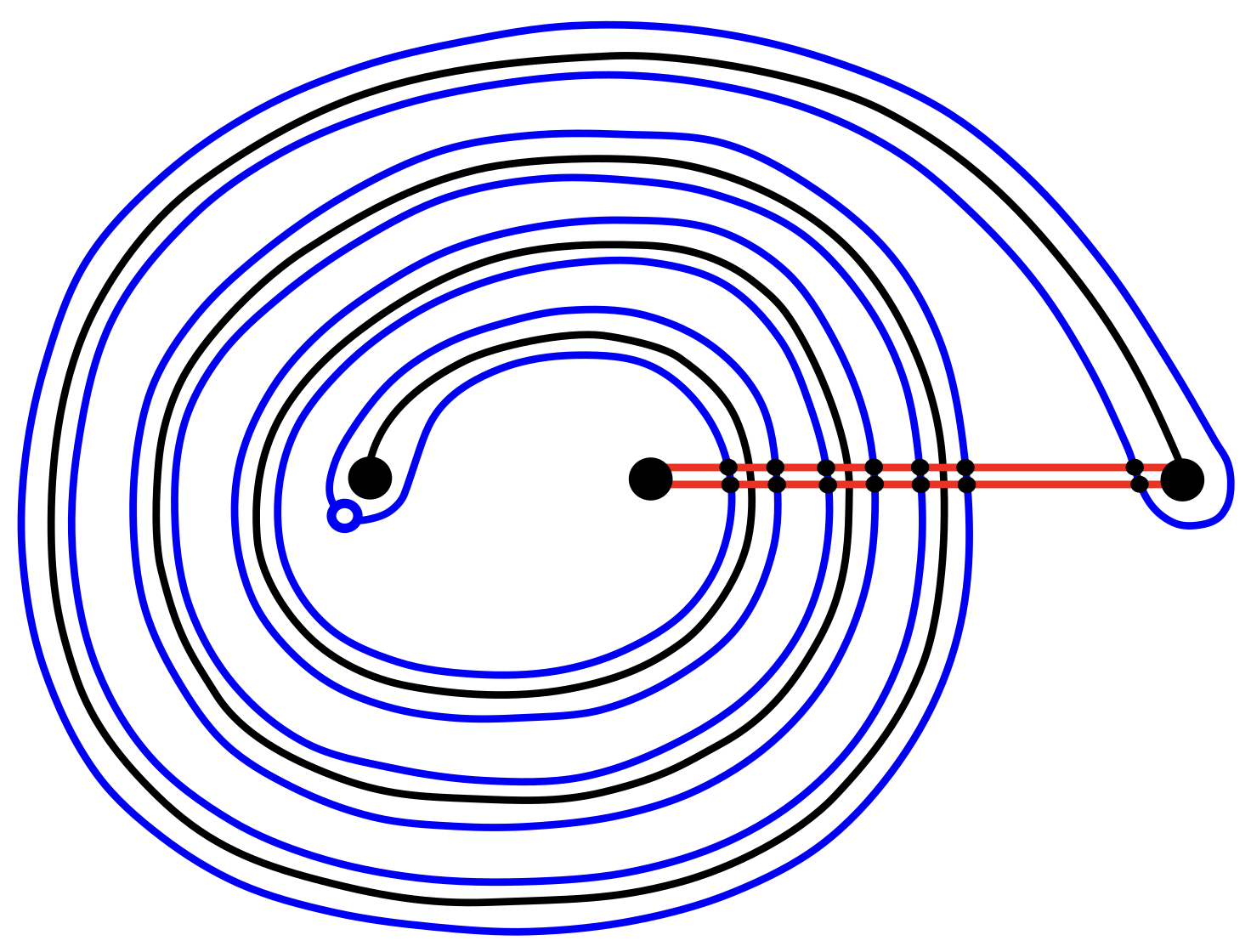}
    \caption{Geometric picture (in $\text{Conf}^2(M)$) for $K_{7/1}=T(2,7)$. Note that the three distinguished points, ordered from left to right, are $X^-,Y,$ and $X^+$.}
    \label{K71fig}
\end{figure}

For example, the data for $\tau_{6/1}$ is
 \begin{figure}[H]
    \centering
\begin{minipage}[b]{.45\textwidth}

$\left[\begin{array}{ c| c | c }
   K&  S & A \end{array}\right] =
  \left[\begin{array}{ c | c | c }
    1 & 5 & 0 \\
    1 & 3 & 0 \\
    1 & 1 & 0 \\
    0 & 0 & 0 \\
   
    \end{array}\right]$
    \end{minipage}
    \begin{minipage}[b]{.45\textwidth}
    $ Q_{6/1}^T=\left[\begin{array}{ c c c | c }
   
    4 & 2 & 0 & 0  \\
    2 & 2 & 0 & 0 \\
    0 & 0 & 0 & 0 \\
    \hline
    0 & 0 & 0 & 0
    \end{array}\right]$.
    \end{minipage}
\end{figure} 
Now, $K_{7/1}=\text{Cl}(T\tau_{6/1})$, and $\mathcal{D}(K_{7/1})$ is shown in Figure \ref{K71fig}. If we take the standard ordering of the $\kappa_i$, we can compute $S,A,$ and $Q_{7/1}^K$ by Theorem \ref{Knotsthm}, which gives the following.

 \begin{figure}[H]
    \centering
\begin{minipage}[b]{.45\textwidth}

$\left[\begin{array}{ c| c }
   S & A \end{array}\right] =
  \left[\begin{array}{ c | c }
    0 & 0 \\
    -2 & 0\\
    -4 & 0\\
    1 & 2\\
    -1 & 2\\
    -3 & 2\\
    \hline
    -6 & 0\\ 
   
    \end{array}\right]$
    \end{minipage}
    \begin{minipage}[b]{.45\textwidth}
    $ Q_{7/1}^K=\left[\begin{array}{ c c c | c c c | c}
   
    6 & 4 & 2 & 4 & 2 & 0 & 0 \\
    4 & 4 & 2 & 3 & 2 & 0 & 0 \\
    2 & 2 & 2 & 1 & 1 & 0 & 0 \\
    \hline
    4 & 3 & 1 & 3 & 1 & -1 & -1\\
    2 & 2 & 1 & 1 & 1 & -1 & -1\\
    0 & 0 & 0 & -1 & -1 & -1 & -1\\
    \hline
    0 & 0 & 0 & -1 & -1 & -1 & 0\\ 
    \end{array}\right]$.
    \end{minipage}
\end{figure} 

In this case, we have $\mu_1(K_{7/1})=-6$ and $\mu_2(K_{7/1})=\mu_3(K_{7/1})=0$. In general, we get $\mu_1(K_{(2n+1)/1})=-2n$ and $\mu_2(K_{(2n+1)/1})=\mu_3(K_{(2n+1)/1})=0$, which gives $\delta=-2n$, the signature of the knot. It should also be observed that $Q_{7/1}^K$ agrees exactly with $Q_{7/1}$, the matrix we get by applying Lemma \ref{WSclosureformula} to $Q_{6/1}^T$. This generalizes for arbitrary rational torus knots.


\begin{propositionN}
\label{ratknotforms}
    For a rational torus knot $K_{(2n+1)/1}$, we have $Q_{(2n+1)/1}^K=Q_{(2n+1)/1}$. 
\end{propositionN}

Naturally, the linear forms also agree when computed by the two different methods, after shifting $S$ appropriately. Proposition \ref{ratknotforms} follows easily by considering the proof of Theorem \ref{Knotsthm} and observing that we do not need to apply any of the permutations (because $\tau_{2n/1}$ has a single inactive intersection). Furthermore, we will compare our matrix computations with the ones from \cite{K19} in Section \ref{sympolysec}, where we will discuss the symmetric-colored HOMFLY-PT polynomials.



\subsection{Example: The Figure-8 knot, $K_{5/2}$}
\label{K52exsec}

Not only is the figure-8 knot important and interesting to study, but we will consider it here since it provides an example of an $RI$ tangle closure, which mirrors, but is distinct from, the $UP$ tangle closure example discussed explicitly in the proof of Theorem \ref{Knotsthm}.

\begin{figure}[H]
    \centering
    \includegraphics[height=5cm]{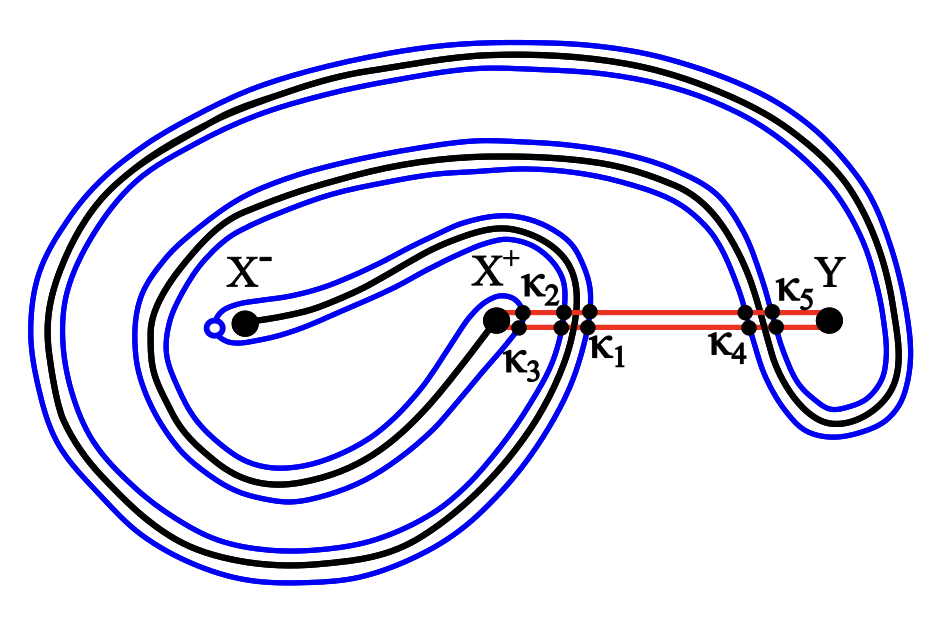}
    \caption{$\mathcal{D}^2(K_{5/2})$}
    \label{K52labelfig}
\end{figure}

Recall that $K_{5/2}=\text{Cl}(T\tau_{3/2})$, so we can compute $S, A,$ and $Q$ for the knot by finding them for $\tau_{3/2}$ by Theorem \ref{tangleHOMFLYpoly} and then applying Lemma \ref{WSclosureformula}. Doing so, one computes

\begin{figure}[h]
    \centering
\begin{minipage}[b]{.45\textwidth}

$\left[\begin{array}{ c| c }
   S & A \end{array}\right] =
  \left[\begin{array}{ c | c }
    2 & 2 \\
    1 & 0\\
    0 & 0\\
    -2 & -2\\
    -1 & 0\\

    \end{array}\right]$
    \end{minipage}
    \begin{minipage}[b]{.45\textwidth}
    $ Q_{5/2}=\left[\begin{array}{ c c c c c}
   
    -2 & -1 & -1 & -2 & -1 \\
    -1 & 1 & 0 & 0 & 1 \\
    -1 & 0 & 0 & -1 & 0\\
    -2 & 0 & -1 & -1 & 0\\
    -1 & 1 & 0 & 0 & 2\\
    
    \end{array}\right]$.
    \end{minipage}
\end{figure}
\noindent
before applying the correction terms. For $K_{5/2}$, we compute $\mu_1(K_{5/2})=\mu_2(K_{5/2})=-1$ and $\mu_3(K_{5/2})=1$, which corresponds to a framing shift. Since we are only concerned with the HOMFLY-PT polynomials up to some framing shift, we will ignore the correction terms for this example.

By Lemma \ref{knotslinforms}, Theorem \ref{Knotsthm} computes the same $S, A,$ and diagonal of $Q_{5/2}^K$, but some of the off-diagonal entries of $Q_{5/2}^K$ are permuted. In particular, Theorem \ref{Knotsthm} gives

 \begin{figure}[h]
    \centering

    \begin{minipage}[b]{.45\textwidth}
    $ Q_{5/2}^K=\left[\begin{array}{ c c c c c}
   
    -2 & -1 & -1 & -2 & \underline{0} \\
    -1 & 1 & 0 & \underline{-1} & 1 \\
    -1 & 0 & 0 & -1 & 0\\
    -2 & \underline{-1} & -1 & -1 & 0\\
    \underline{0} & 0 & 1 & 0 & 2\\
    
    \end{array}\right]$,
    \end{minipage}
\end{figure}
\noindent
where the different permuted entries are underlined. Note that these permutations match with the ones predicted in the proof of Theorem \ref{Knotsthm}. 

Given $Q_{55}=2$, we will show that $Q_{51}=0$ by our formulas. We are using
\begin{equation}
\label{Q15calc}
Q_{51}=Q_{55}+(\Phi-2\Psi_{X^+})([\gamma_{1,5}]);
\end{equation}
the loop $\gamma_{1,5}$ is shown below, along with the corresponding braid diagram with the relevant strands for computing $(\Phi-2\Psi_{X^+})([\gamma_{1,5}])$.

\[
\vcenter{\hbox{\includegraphics[height=4cm,angle=0]{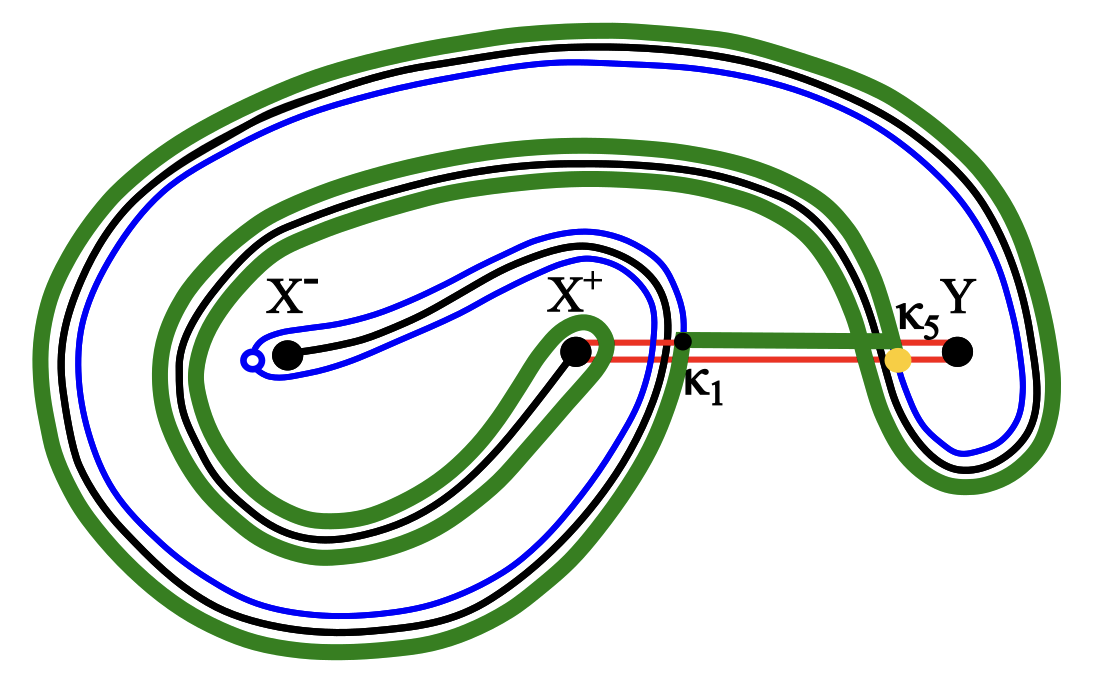}}} \qquad \qquad 
\vcenter{\hbox{\includegraphics[height=4cm,angle=0]{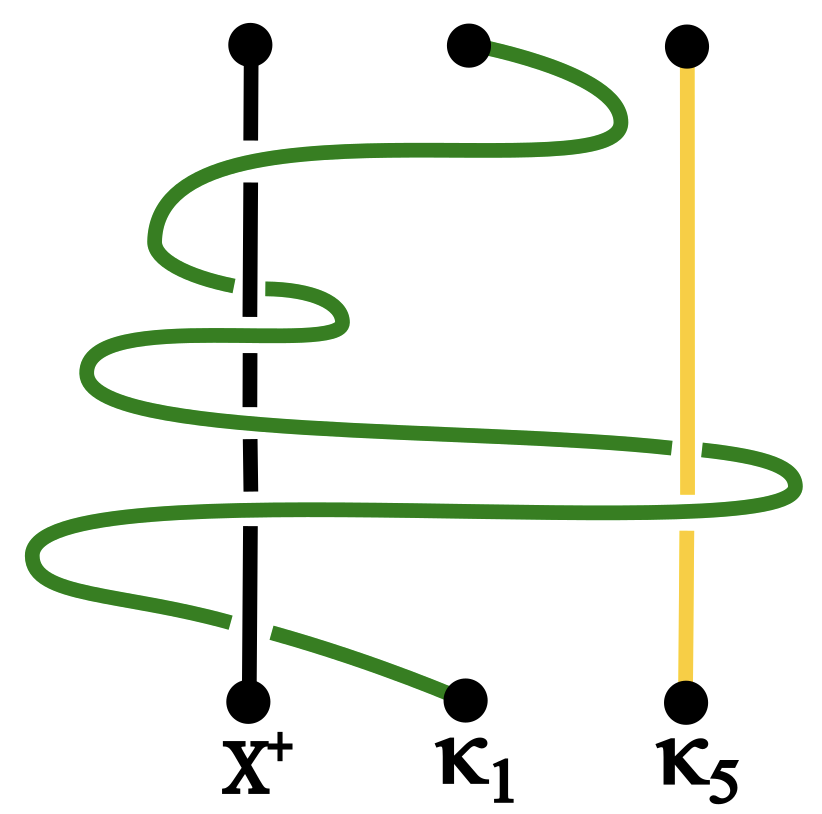}}}
\]

One can easily see now that $\Phi([\gamma_{1,5}])=2$ by looking at the strands corresponding to $\kappa_1$ and $\kappa_5$, and similarly $\Psi_{X^+}([\gamma_{1,5}])=2$ by looking at the strands for $X^+$ and $\kappa_1$. Thus, we get $Q_{51}=0$ from Equation \ref{Q15calc}.

\section{2-component Rational Links}
\label{linkssec}

In this section, we will consider the case where $\text{Cl}(\tau_{u/v})$ is a 2-component link, which occurs when $u$ is even. We will write this as $L_{u/v}=\text{Cl}(\tau_{u/v})$. Although the ``geometric picture'' for the 2-component links is different from the knots, the general structure of this section will be similar to that of the previous one, section \ref{knotssec}, so we will leave out details to avoid too much redundancy, while also emphasizing where things differ.

For the remainder of the paper, we will use ``links'' to designate the 2-component rational links and ``knots'' for the single component rational links.

\subsection{Setting Up the Theorem for Links}

First, we need to carefully describe the geometric picture for rational links. The blue Lagrangian will be defined the same as for knots, and we will continue to denote it by $\overline{\alpha_{u/v}}$. The red Lagrangian, however, differs from the knot case: just as the blue Lagrangian may be thought of as a band tied tightly around $\alpha_{u/v}$, the red Lagrangian for links can be thought of as a band wrapped between $X^-$ and $X^+$. We denote this new red Lagrangian by $\beta'$. As we should expect, there are $2u$ intersection points between $\overline{\alpha_{u/v}}$ and $\beta'$ for $L_{u/v}.$ See Figure \ref{L83}. We will label these intersection points as $\lambda_i$ for $1\leq i \leq 2u$. We will use $\mathcal{G}_{u/v}^L$ to denote this set of intersection points.

\begin{definitionN}
    Given a 2-component rational link $L_{u/v}$, let $\mathcal{D}(L_{u/v})$ be the picture in the 3-punctured plane $M$ consisting of $\alpha_{u/v}, \overline{\alpha_{u/v}},$ and $\beta'$.
\end{definitionN}

One way to see that we get $2u$ intersections is that we can view our link $L_{u/v}$ as $\text{Cl}(T\tau_{(u-v)/v})$ with $\tau_{(u-v)/v}$ having $UP$ or $RI$ orientation, just as we saw for knots; geometrically, we can apply an untwisting isotopy to compare $\mathcal{D}(L_{u/v})$ with $\mathcal{D}(\tau_{(u-v)/v})$, where we can make sense of active and inactive intersection points. Then, we see that we have $2(u-v)$ intersection points coming from the active side and $2v$ from the inactive, giving a total of $2u$. Figure \ref{L83untwistfig} helps illustrate this.

\begin{figure}
    \centering
    \includegraphics[height=7cm]{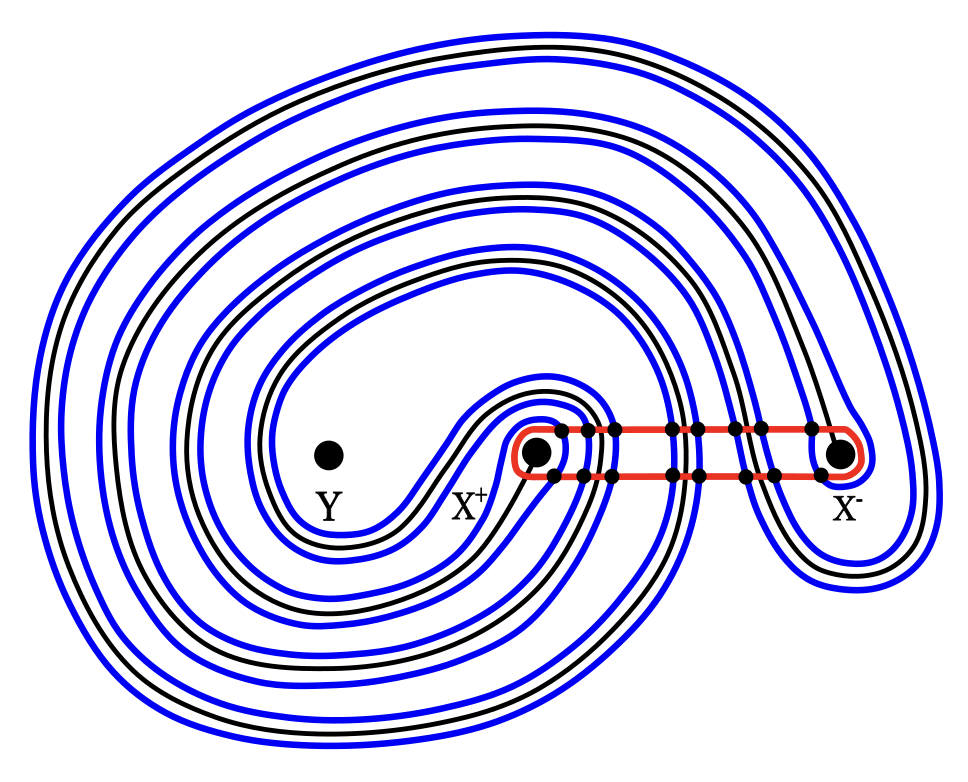}
    \caption{$\mathcal{D}(L_{8/3})$}
    \label{L83}
\end{figure}

The quiver form, according to \cite{SW21}, for the colored HOMFLY-PT generating function for $L_{u/v}$ is
\begin{equation}
   P(L_{u/v})=\sum_{\textbf{d}=(d_1,...,d_{2u})\in \mathbb{N}^{2u}} \frac{(-q)^{S\cdot \textbf{d}}a^{A\cdot \textbf{d}}q^{\textbf{d}\cdot Q \cdot\textbf{d}^T}}{\prod_{i=1}^{2u}(q^2;q^2)_{d_i}}x^{d_1+...+d_{2u}},
\end{equation}
which differs from the case for knots in that it is no longer a Laurent polynomial. However, the computations of $S,A,$ and $Q$ will prove to be very similar to how they were done in Section \ref{knotssec}. Just as we saw for knots, in order to compute the linear and quadratic forms of the quiver form, we will only need to consider loops in $M$ and $\text{Conf}^2(M)$. We were able to define loops from $\kappa_i$ (for any $i$) to $\kappa_\omega$, the intersection point immediately preceding $X^+$, but there is not an obvious choice of this point for links--- there are two of them. Thus, we need to specify which one is our $\lambda_\omega \in \mathbb{L}^1_{u/v}\cap \mathbb{L}^1$ before proceeding.

This will depend on whether $\tau_{(u-v)/v}$ has $UP$ or $RI$ orientation. If $\mathcal{D}(L_{u/v})$ is drawn without the untwisting isotopy, as in Figure \ref{L83}, then of the two intersection points next to $X^+$, take the bottom one to be $\lambda_\omega$ in the $UP$ case and the top one for the $RI$ case. These are flipped in the picture once the untwisting isotopy is applied. For $L_{8/3}$, since $\tau_{5/3}$ has $RI$ orientation, $\lambda_\omega$ is the top-left intersection point on the left picture in Figure \ref{L83untwistfig}; after undoing the final top twist $T$, $\lambda_\omega$ is the lower of the two intersections next to $X^+$. 

\begin{figure}
   \begin{tikzpicture}
\node at (-4.1,0) {\includegraphics[height=5cm, angle=0]{L83.png}};
\node at (0,0) {$\xrightarrow{\text{isotopy}}$};
\node at (4.1,0) {\includegraphics[height=5.5cm, angle=0]{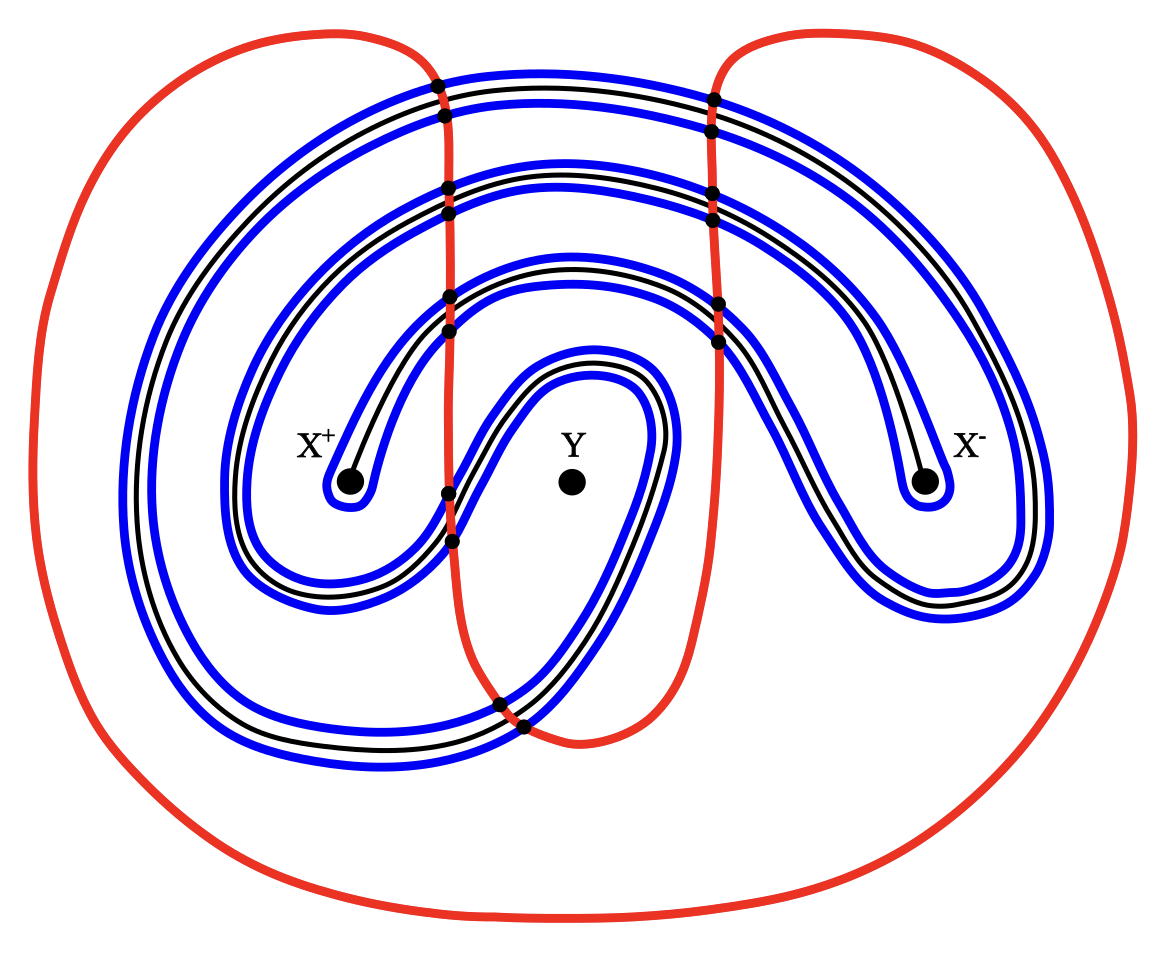}};
\end{tikzpicture}
\caption{The untwisting isotopy for $L_{8/3}$}
\label{L83untwistfig}
\end{figure}

The loops in $M$ used for computing $S,A,$ and the diagonal of $Q$ are given by concatenating paths between the specified intersection points along the blue Lagrangian $\overline{\alpha_{u/v}}$ with a path back to the basepoint along the red Lagrangian $\beta'$. As both of the Lagrangians are now simple closed curves, there are always two possible directions to go along each of them. Just as we did for knots by removing a point from the blue Lagrangian to uniquely specify the loops up to homotopy, we can now remove a point from both the blue and red Lagrangians to uniquely determine the desired loops up to homotopy. In particular, we will remove the same point from the blue Lagrangians as for knots (the one next to $X^-$); for the red Lagrangian, remove the point where it intersects the underlying arc $\alpha_{u/v}$ next to $X^+$. Thus, we can define $\eta_{i,j}$ to be the (unique up to homotopy) path from $\lambda_i$ to $\lambda_j$ along $\overline{\alpha_{u/v}}$ that misses hole, and $\overline\eta_{j,i}$ to be the (unique up to homotopy) path from $\lambda_j$ back to $\lambda_i$ along $\beta'$ that misses the hole in $\beta'$. Now, we can define our loops in $M$.

\begin{definitionN}
    Given $\lambda_i$ and $\lambda_j$ define $\gamma^L_{i,j}:[0,1]\to M$ to be $\gamma^L_{i,j}:=\eta_{i,j}\cdot \overline\eta_{j,i}$. Furthermore, we define $\gamma^L_i:=\gamma^L_{i,\omega}$.
\end{definitionN}

See Figure \ref{L83lambdaifig} for an example loop $\gamma^L_i$ for $L_{8/3}$, where open circles denote the ``holes'' in the Lagrangians (the no-pass zones) and $i$ and $\omega$ denote the intersection points $\lambda_i$ and $\lambda_\omega$.

\begin{figure}
    \centering
    \includegraphics[height=7cm]{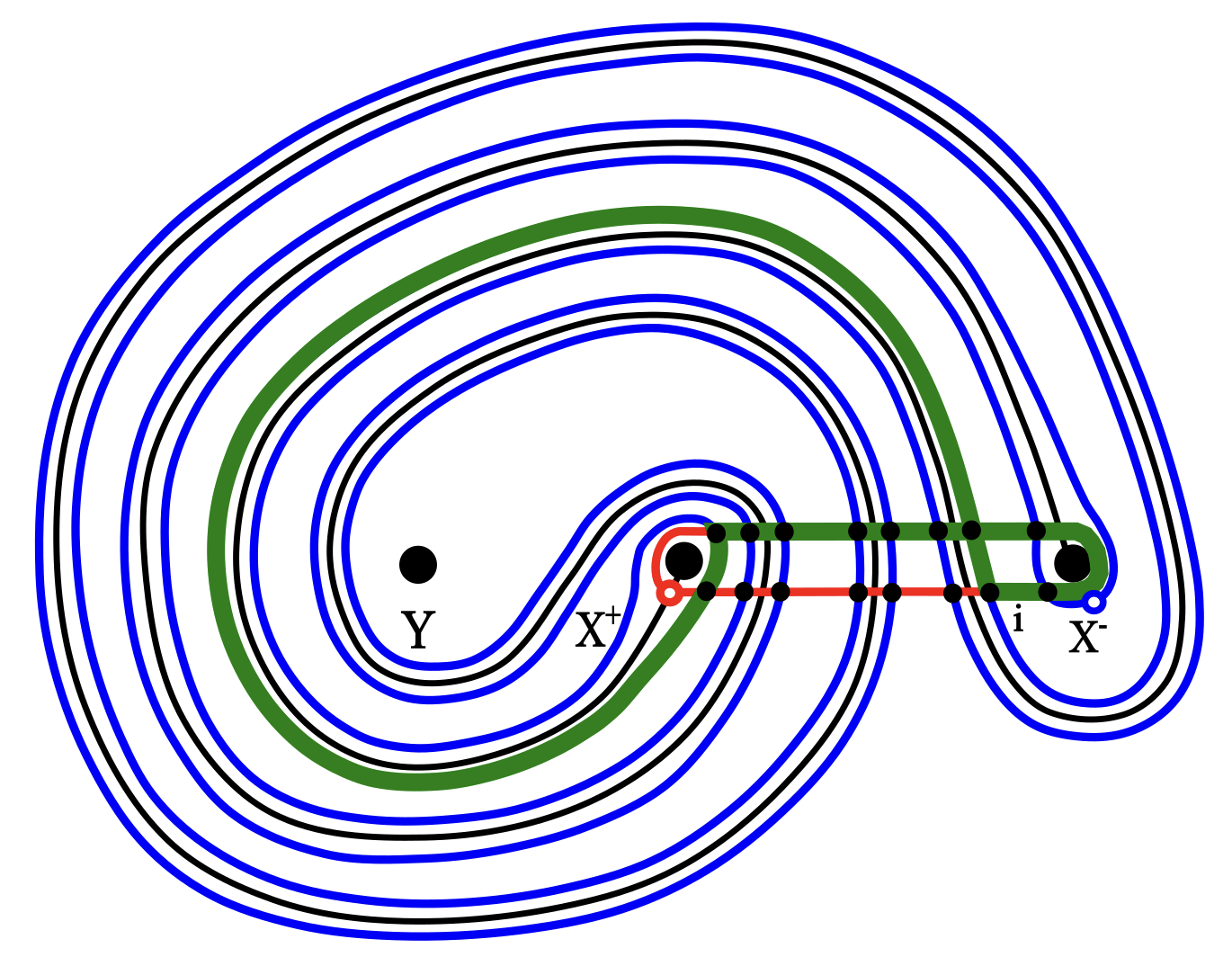}
    \caption{An example loop $\gamma_i=\gamma_{i,\omega}$ in $M$, drawn in green}
    \label{L83lambdaifig}
\end{figure}

For loops in $\text{Conf}^2(M)$, we will draw everything in the plane by taking two parallel copies of the red Lagrangian $\beta'$ and thinking of the loop as a pair of paths as we did for knots. Just as we had $\mathcal{D}^2(K_{u/v})$ as our diagrammatic model for defining our loops in $\text{Conf}^2(M)$, we will also define $\mathcal{D}^2(L_{u/v})$ for 2-component links. Let $\beta_1'$ and $\beta_2'$ be two parallel copies of $\beta'$, with $\beta_1'$ being the inner one. We can once again define
\[
\mathbb{L}_{u/v}^j:=\text{Sym}^j(\overline{\alpha_{u/v}})\setminus \Delta \subset \text{Conf}^j(M)
\]
and 
\[
\mathbb{L}^j:=\pi(\beta'_1\times ...\times \beta'_j) \subset \text{Conf}^j(M),
\]
as we did for knots.

\begin{definitionN}
  Let $\mathcal{D}^2(L_{u/v})$ be the configuration of Lagrangians in $\text{Conf}^2(M)$ consisting of  $\mathbb{L}_{u/v}^2$ and $\mathbb{L}^2$. This is modeled in $M$ by taking $\overline{\alpha_{u/v}}, \beta_1'$, and $\beta_2'$, where points in $\text{Conf}^2(M)$ are represented by pairs of distinct points on the Lagrangians and the loops are given by pairs of disjoint paths with basepoints $(x_1,x_2)$ such that $x_1\in \overline{\alpha_{u/v}}\cap \beta'_1$ and $x_2\in \overline{\alpha_{u/v}}\cap \beta'_2$.
\end{definitionN}

Now, given  $\lambda_i$ and $\lambda_j$, to define the loop $\tilde\gamma^L_{i,j}\subset\text{Conf}^2(M)$, the same holes in the blue and red Lagrangians discussed in the previous paragraph determine homotopy classes of loops once we have determined which copy of $\beta'$ to place $\lambda_i$ and $\lambda_j$ on. This is determined by the following rule: in a neighborhood of $\beta'$ in $M$, $\mathcal{D}(L_{u/v})$ (without untwisting isotopy) looks like the image below on the left for the $UP$ case and the image on the right for the $RI$ case, up to isotopy; $\lambda_\omega$ is labeled, and the black arrow on the outside (not part of $\mathcal{D}(L_{u/v})$) gives an ordering on the $\lambda_i$ such that if $\lambda_i$ comes before $\lambda_j$ with respect to this ordering, then $\lambda_i$ is placed on $\beta_2'$, the outside copy of $\beta'$, for defining loops $\tilde\gamma^L_{i,j}$ and $\tilde\gamma^L_{j,i}$ in $\text{Conf}^2(M)$. Thus, we see that $\lambda_\omega$ is always placed on $\beta_2'$.

\[
\vcenter{\hbox{\includegraphics[height=2.5cm,angle=0]{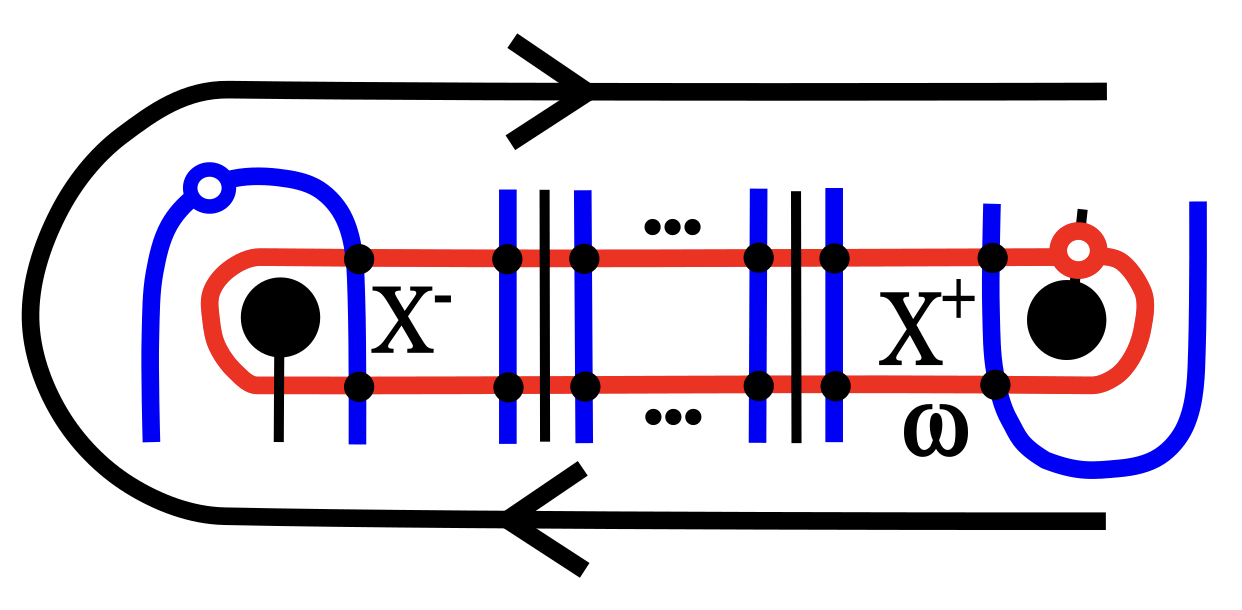}}} \qquad \qquad 
\vcenter{\hbox{\includegraphics[height=2.5cm,angle=0]{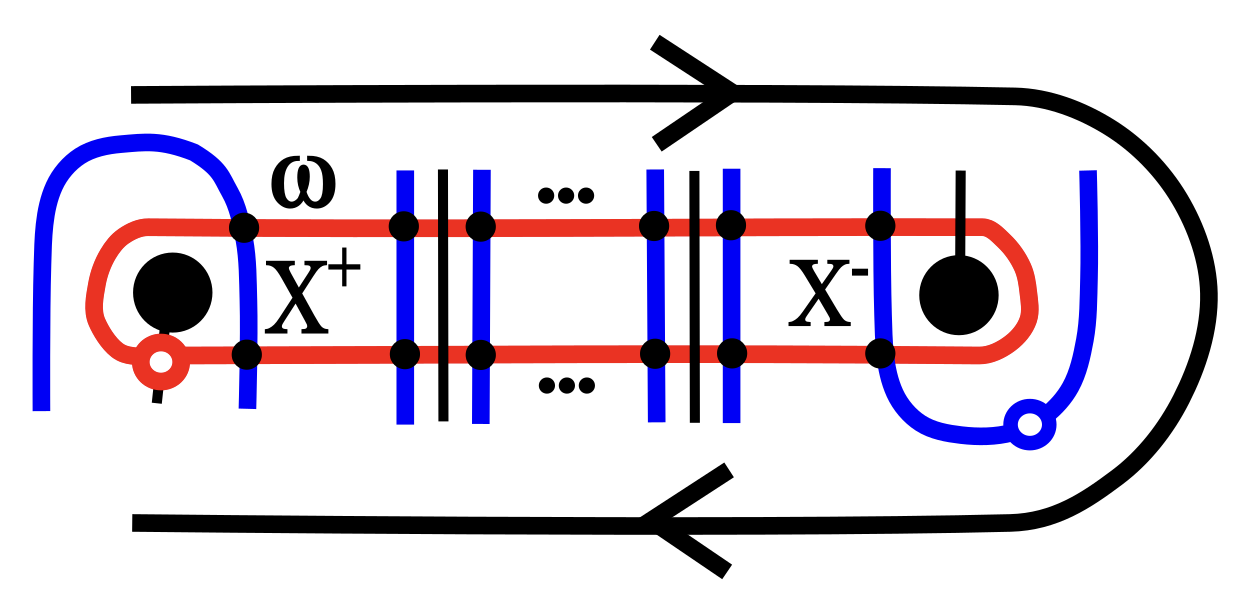}}}
\]

Once we have determined which copy of $\beta'$ to place $\lambda_i$ and $\lambda_j$ on, then we can define $\tilde\gamma^L_{i,j}$ by the same rule that we did for knots, but using our new Lagrangians. An example $\tilde\gamma^L_{i,j}\subset\text{Conf}^2(M)$ is shown in Figure \ref{L83lambdaijfig} (with the green and yellow denoting its two paths components, as we did in the last section). Notice that the basepoint has $\lambda_i$ on the outer copy of the red Lagrangian and $\lambda_j$ on the inner one, in agreement with what was just established.

\begin{figure}
    \centering
    \includegraphics[height=9cm]{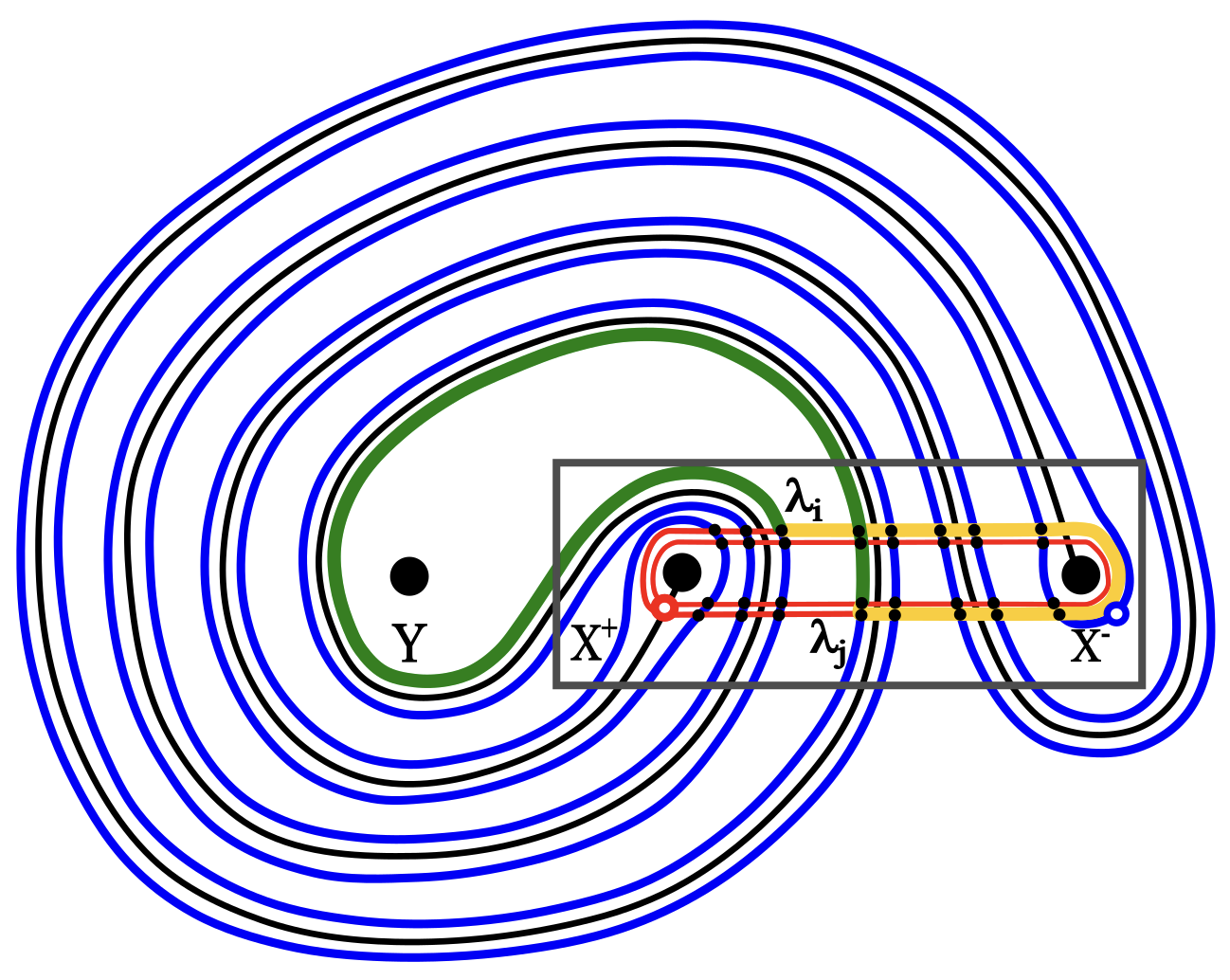}
    \qquad \qquad
    \includegraphics[height=5cm]{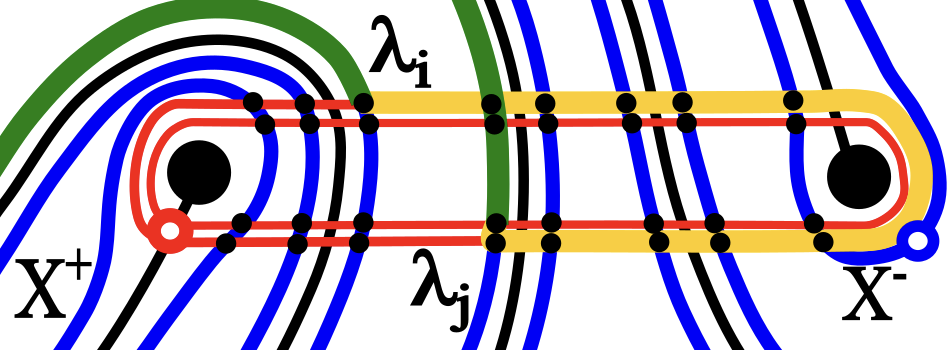}
    \caption{Above: An example loop $\gamma_{i,j}$ in $\text{Conf}^2(M)$, Below: A zoomed-in look at the loop in the boxed region}
    \label{L83lambdaijfig}
\end{figure}

Now, we can state the main theorem for links. Notice that formulas for $S,A,$ and $Q$ are the same as in the case for knots. Of course, the $2u$ indices are coming, geometrically, from the $2u$ intersection points of the Lagrangians in $M$; at this time, we do not need to specify how these intersection points are ordered because the generating function is invariant under the $S_{2u}$ action that permutes index labels. However, it will be helpful to fix the ordering when proving the theorem, as we did in Section \ref{knotssec}.

\begin{theoremN}
\label{Linkthm}
    The generating function for the colored HOMFLY-PT polynomials for the rational link $L_{u/v}=\text{Cl}(\tau_{u/v})$ can be written as
    \begin{equation}
        P(L_{u/v})=\sum_{\textbf{d}=(d_1,...,d_{2u})\in \mathbb{N}^{2u}} \frac{(-q)^{S\cdot \textbf{d}}a^{A\cdot \textbf{d}}q^{\textbf{d}\cdot Q \cdot\textbf{d}^T}}{\prod_{i=1}^{2u}(q^2;q^2)_{d_i}}x^{d_1+...+d_{2u}}
    \end{equation}
    where $S$, $A$, and $Q$ are defined on the free $\mathbb{Z}$-module $\mathbb{Z}\mathcal{G}_{u/v}^L$ (with basis $\mathcal{G}_{u/v}^L$) and are computed by the following formulas:
    \[
    \begin{cases}
        S_i=\Psi_{\{X^\pm,Y\}}([\gamma_i])+\mu_1(L_{u/v})\\
        A_i=2\Psi_{X^+}([\gamma_i])+\mu_2(L_{u/v})\\
        \begin{cases}
            Q_{ii}=(\Psi_{\{X^-,Y\}}-3\Psi_{X^+})([\gamma_i])+\mu_3(L_{u/v})\\
            Q_{ij}=Q_{ii}+(\Phi-2\Psi_{X^+})([\gamma_{j,i}]).
        \end{cases}
    \end{cases}
    \]
    The correction terms are computed by the same procedure used for knots.
\end{theoremN}

We include the correction terms to agree with the computations of Sto\v si\'c and Wedrich in \cite{SW21} up to a $\pm 1$ framing shift and permutations of matrix entries, but we will make no further comments on them. The proof of Theorem \ref{Linkthm} will use the formulas ignoring the $\mu_i(L_{u/v})$ terms, as the primary concern is that everything works up to some shift.

Before proving the theorem, it should be noted again, as we did for Theorem \ref{Knotsthm}, that the formulas here are written in the way that makes them easiest to compute, but we can still define a single formula for entries of $Q$ by
\[
Q_{ij}=(2\Phi^2-\Psi_{X^+}^2)([\gamma_{j,i}^{2,L}])+\mu_3(L_{u/v}).
\]
The $\gamma_{j,i}^{2,L}$ loops are defined in a similar way for links as they were done for tangles in \cite{JHpI26} and knots in Section \ref{knotssec}, but adapted to the new Lagrangians. We can also get a ``links analogue'' of Proposition \ref{KnotQsymprop}, which demonstrates that the formula for $Q$ in Theorem \ref{Linkthm} gives $Q_{ij}=Q_{ji}$.

\subsection{Proof of Theorem \ref{Linkthm}}

Just as we did for knots, we will prove the main theorem for links, Theorem \ref{Linkthm}, by proving a sequence of lemmas. These lemmas will closely mirror the ones for knots, but adapted to our new geometric picture for links. As before, the first step is to see how the closure formulas from \cite{SW21} predict that $S,A,$ and $Q$ for $L_{u/v}=\text{Cl}(T\tau_{(u-v)/v})$ are related to the ones for $\tau_{(u-v)/v}$. 

\begin{lemmaN}
   \label{WSclosureformulaL}
    If $L_{u/v}=\text{Cl}(T\tau_{(u-v)/v})$ for $\tau_{(u-v)/v}$ having $UP$ orientation, then the $\text{Cl}(T-)$ operation has the effect
   \begin{multline*}
      UP, \left[\begin{array}{ c | c  }
    S_+ & A_+  \\
    \hline
    S_- & A_- 
  \end{array}\right],  
  \left[\begin{array}{ c |c }
    Q_{++} & Q_{+-} \\
    \hline
    Q_{-+} & Q_{--}
  \end{array}\right]
  \\
    \xrightarrow{\text{Cl}(T-)}
    \left[\begin{array}{ c | c  }
    S_++2 & A_++2 \\
    \hline
    S_++1 & A_+  \\
    \hline
    S_-+1 & A_- \\
    \hline
    S_- & A_-
  \end{array}\right], 
  \left[\begin{array}{ c | c | c | c}
    Q_{++}-1 & Q_{++}+U & Q_{+-} & Q_{+-}-1 \\
    \hline
    Q_{++}+L & Q_{++}+2 & Q_{+-}+1 & Q_{+-}\\
    \hline
    Q_{-+} & Q_{-+}+1 & Q_{--}+1 & Q_{--}+L \\
    \hline
    Q_{-+}-1 & Q_{-+} & Q_{--}+U & Q_{--}
  \end{array}\right];
   \end{multline*}

  In the case that $\tau_{(u-v)/v}$ has $RI$ orientation, then $\text{Cl}(T-)$ has the effect
\begin{multline*}
      RI, \left[\begin{array}{ c | c  }
    S_+ & A_+  \\
    \hline
    S_- & A_- 
  \end{array}\right],  
  \left[\begin{array}{ c |c }
    Q_{++} & Q_{+-} \\
    \hline
    Q_{-+} & Q_{--}
  \end{array}\right]
  \\
    \xrightarrow{\text{Cl}(T-)}
    \left[\begin{array}{ c | c  }
    S_+ +1 & A_+  \\
    \hline
    S_+ & A_-  \\
    \hline
    S_-  & A_- \\
    \hline
    S_--1 & A_--2
  \end{array}\right], 
   \left[\begin{array}{ c | c | c | c}
    Q_{++}+1 & Q_{++}+L & Q_{+-} & Q_{+-}+1 \\
    \hline
    Q_{++}-1+U & Q_{++} & Q_{+-}-1 & Q_{+-}\\
    \hline
    Q_{-+} & Q_{-+}-1 & Q_{--}-1 & Q_{--}+U \\
    \hline
    Q_{-+}+1 & Q_{-+} & Q_{--}+L & Q_{--}+2
  \end{array}\right].
   \end{multline*}
 
\end{lemmaN}

Note that, in contrast to Lemma \ref{WSclosureformula}, we are no longer using the linear form $K$ for $\tau_{(u-v)/v}$. In other words, we are using Proposition \ref{tangleHOMFLYpoly} instead of Proposition \ref{indpartprop} to compute the generating function data for $\tau_{(u-v)/v}$. On the tangle side in the lemma below, we are thinking of $\mathcal{G}_{(u-v)/v}^T$ as $\mathfrak{X}\sqcup \mathfrak{Y}$, where $\mathfrak{X}$ consists of all active intersections and $\mathfrak{Y}$ consists of all inactive intersections for $\tau_{(u-v)/v}$. In the absence of $K$, both $\mathfrak{X}$ and $\mathfrak{Y}$ double to account for the four blocks we get after applying $\text{Cl}(T-)$, which is consistent with the number of Lagrangian intersections in $\mathcal{D}(L_{u/v})$ being twice the number of intersections in $\mathcal{D}(\tau_{(u-v)/v})$. In particular, as we already discussed, if we apply the untwisting isotopy to $\mathcal{D}(L_{u/v})$, we expect that the Lagrangian intersections should relate to the intersection points in $\mathcal{D}(\tau_{(u-v)/v})$ such that each point in $\mathcal{D}(\tau_{(u-v)/v})$ corresponds to two points in $\mathcal{D}(L_{u/v})$.

\begin{proof}[Proof of Lemma \ref{WSclosureformulaL}]
    The proof is similar to the proof for Lemma \ref{WSclosureformula}. In particular, it follows from Equations (\ref{TUPclosure}) and (\ref{TRIclosure}) again, but Lemma \ref{algidentity} must be applied twice since there is no cancellation of $(q^2;q^2)_{K\cdot \textbf{d}}$ terms.

    Before proceeding, it is worth noting how we get the Pochhammer symbols in the denominator of the quiver form for links, whereas this did not happen for knots. Consider the $UP$ case. Ignoring the parts of the terms involving $S,A,$ and $Q$, one has expressions like
    \[
    {d_1+...+d_{u-v}\brack d_1,...,d_{u-v}}{d_{u-v+1}+...+d_u\brack d_{u-v+1},...,d_u}UP[j,k]
    \]
    in the terms of the quiver form for $\tau_{(u-v)/v}$. Then one applies Equation (\ref{TUPclosure}) which gives (ignoring again the parts that contribute to $S$ and $Q$)
    \[
    \frac{(q^2;q^2)_j}{\prod_{i=1}^{u-v}(q^2;q^2)_{d_i} \prod_{i=1}^v(q^2;q^2)_{d_{u-v+i}}} \frac{(a^2q^{2-2j-2k};q^2)_k}{(q^2;q^2)_k}= \frac{(a^2q^{2-2j-2k};q^2)_k}{\prod_{i=1}^{u-v}(q^2;q^2)_{d_i}} \frac{(q^{2+2k};q^2)_{j-k}}{\prod_{i=1}^v(q^2;q^2)_{d_{u-v+i}}},
    \]
    where the two fractions on the right are the ones to which we apply Lemma \ref{algidentity}, and this leaves us with the Pochhammer symbols in the denominator.
\end{proof}

The next lemma is the links analogue for Lemma \ref{knotslinforms}.

\begin{lemmaN}
\label{linkslinform}
    When computed by Theorem \ref{Linkthm}, $S,A,$ and the diagonal of $Q$ for $L_{u/v}=\text{Cl}(T\tau_{(u-v)/v})$ agree with Lemma \ref{WSclosureformulaL}.
\end{lemmaN}

In order to make sense of Lemma \ref{linkslinform}, we need to see how to partition the Lagrangian intersections of $\mathcal{D}(L_{u/v})$ into four disjoint sets. We also need to determine how these four partitioning sets relate to the two used for $\tau_{(u-v)/v}$ ($\mathfrak{X}$ and $\mathfrak{Y}$, consisting of the active and inactive intersections, respectively). In addition to determining these partitions, we will also provide our intersection points for $L_{u/v}$ with a new ``standard ordering,'' thus providing an ordered basis for the free $\mathbb{Z}$-module that $S,A,$ and $Q$ are defined on in Theorem \ref{Linkthm}.

It is clear that both the active and inactive intersection points for $\tau_{(u-v)/v}$ are doubling, resulting in two new blocks, so if $\xi_i$ is an active intersection point ($1\leq i \leq u-v$), then it is paired with $\lambda_i$ and $\lambda_{u-v+i}$ for the link; if $\xi_i$ is inactive, then it is paired with $\lambda_{2(u-v)+i}$ and $\lambda_{2u-v+i}$. Assume the active and inactive $\xi_i$ are ordered separately by $\prec$ (so $i<j$ if $\xi_i$ is active and $\xi_j$ is inactive; otherwise, if both are active or inactive, then $i<j$ if $\xi_i\prec \xi_j$). 

For the $UP$ case, a neighborhood of $\beta'$ looks like:

\[
\vcenter{\hbox{\includegraphics[height=3cm,angle=0]{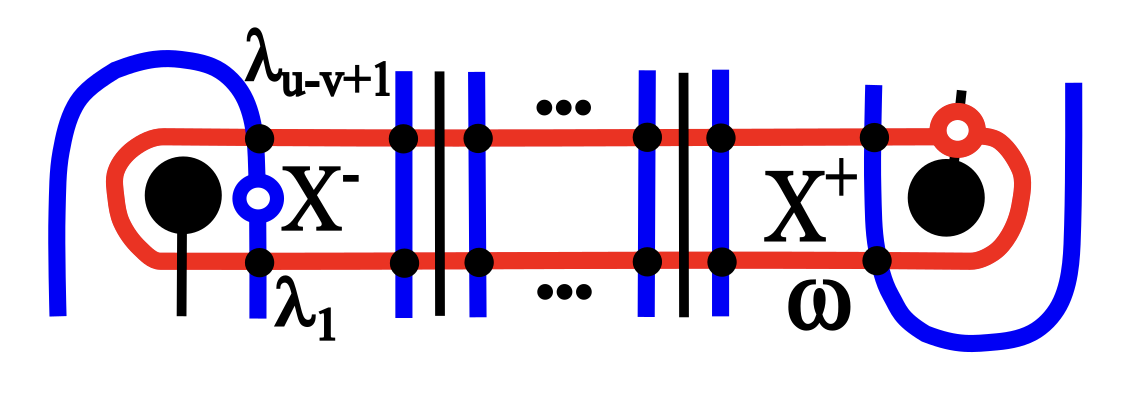}}} 
\]
where $\lambda_1$ and $\lambda_{u-v+1}$ are labeled as the two points closest to $X^-$, so they are naturally paired with $\xi_1$, with $\lambda_1$ being the bottom of the two. The hole in the blue Lagrangian determines distinct paths along it from $\lambda_1$ and $\lambda_{u-v+1}$ to $\lambda_\omega$; the blocks $\mathfrak{A}=\{\lambda_1,...,\lambda_{u-v}\}$ and $\mathfrak{B}=\{\lambda_{u-v+1},...,\lambda_{2(u-v)}\}$ are the intersection points paired with active points for $\tau_{(u-v)/v}$, ordered according to this path. 

For the blocks $\mathfrak{C}=\{\lambda_{2(u-v)+1},...,\lambda_{2u-v}\}$ and $\mathfrak{D}=\{\lambda_{2u-v+1},...,\lambda_{2u}\}$ coming from the inactive side of the tangle, we already have $\lambda_{2u}=\lambda_\omega$ determined, and $\lambda_{2u-v}$ is the point directly above it. The remaining intersection points in $\mathfrak{C}$ and $\mathfrak{D}$ come in groups of four given by rectangles of blue and red sides. For $UP$, the labels for the points on the left are one less than the ones on the right, the top two points are in $\mathfrak{C}$, and the bottom two are in $\mathfrak{D}$. Thus, the four intersections coming from one of these rectangles are related by 

\[
\vcenter{\hbox{\includegraphics[height=3cm,angle=0]{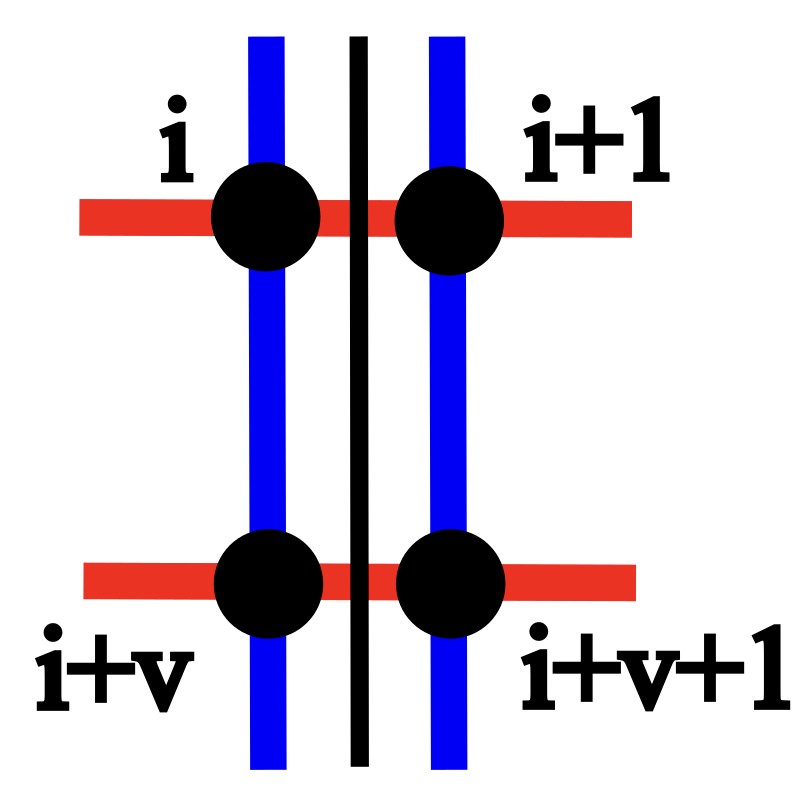}}}, 
\]
where only the subscripts of the $\lambda$'s are written. These squares can then be ordered with respect to each other by the same parametrization of the underling arc $\alpha_{u/v}$ from $X^-$ to $X^+$; in particular, if one of these squares comes before another with respect to this parametrization, then the labels for the ``earlier'' square are lower than the ones for the ``later'' one.

For $RI$, the procedure is similar, but with the roles of the active and inactive indices swapped. Furthermore, the labeling scheme for rectangles on the active side (notice this was on the inactive side for the $UP$ case) is

\[
\vcenter{\hbox{\includegraphics[height=3cm,angle=0]{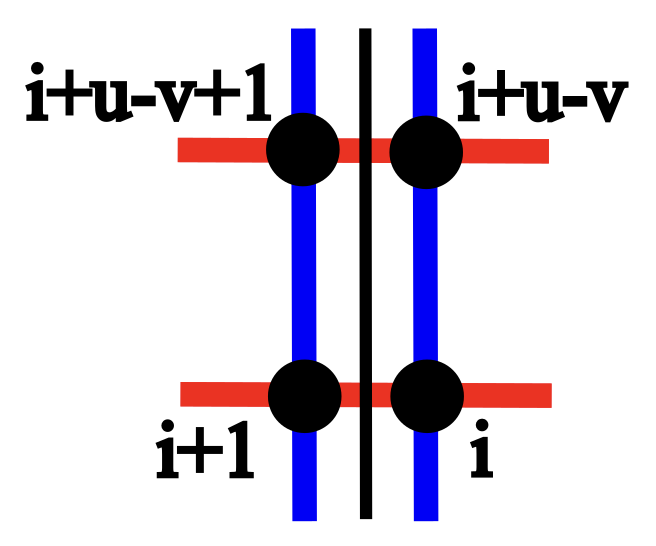}}}.
\]

As an example, if we look in a neighborhood of the intersection points for $L_{8/3}=\text{Cl}(T\tau_{5/3})$ (notice $\tau_{5/3}$ has $RI$ orientation), then our procedure for ordering the $\lambda_i$ gives

\[
\vcenter{\hbox{\includegraphics[height=4cm,angle=0]{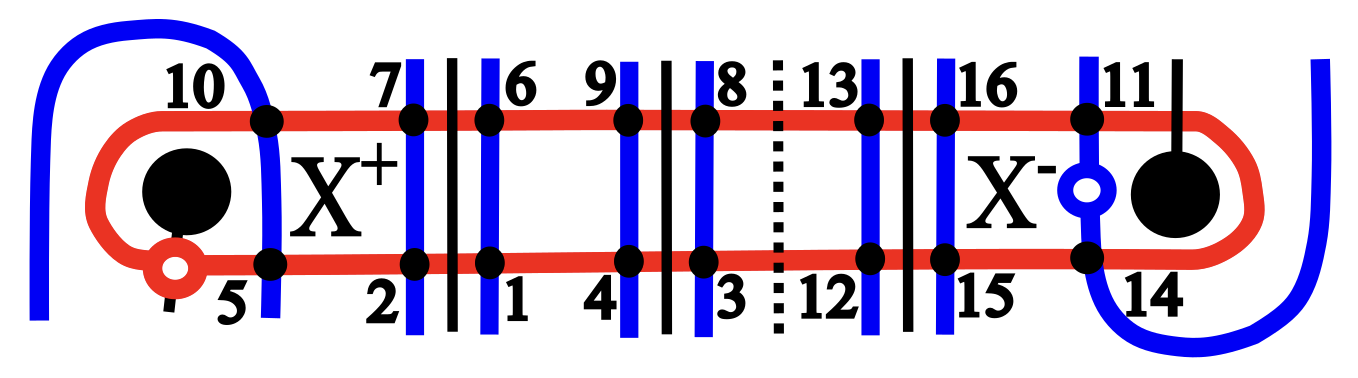}}},
\]
where the dotted line separates the intersections coming from the active side (blocks $\mathfrak{A}$ and $\mathfrak{B}$) from the ones coming from the inactive side (blocks $\mathfrak{C}$ and $\mathfrak{D}$). Compare with the whole picture for $L_{8/3}$ in Figure \ref{L83} to see how the rules described above result in these labels.

\begin{definitionN}
    Let the ordering of the $\lambda_i$ just described be called the \textit{standard ordering of the $\lambda_i$} for $L_{u/v}$. Furthermore, let $Q_{u/v}^L$ be the $2u\times 2u $ matrix computed by Theorem \ref{Linkthm} with respect to this ordered basis. It can be thought of as a $4\times 4$ block matrix indexed by $\mathfrak{A},\mathfrak{B},\mathfrak{C},$ and $\mathfrak{D}$.
\end{definitionN}

Now, we can prove the lemma.

\begin{proof}[Proof of Lemma \ref{linkslinform}]
As in most of the proofs in Section \ref{knotssec}, we will provide details for the $UP$ case.

The statement of the lemma is that we should have
 \[
      UP, \left[\begin{array}{ c | c | c  }
    S_+ & A_+ & Q_+ \\
    \hline
    S_- & A_- & Q_-
  \end{array}\right]
    \xrightarrow{\text{Cl}(T-)}
    \left[\begin{array}{ c | c | c }
    S_++2 & A_++2 & Q_+-1\\
    \hline
    S_++1 & A_+ & Q_++2 \\
    \hline
    S_-+1 & A_- & Q_-+1\\
    \hline
    S_- & A_- & Q_-
  \end{array}\right],
   \]
   where this should be read the same way as in Lemma \ref{knotslinforms}. First we show that this is true for the $\mathfrak{A}$ and $\mathfrak{B}$ blocks (the ones paired with active intersections), and then we will consider the $\mathfrak{C}$ and $\mathfrak{D}$ blocks (the ones paired with inactive intersections).

It will be advantageous to consider how $\mathcal{D}(L_{u/v})$ looks after applying the untwisting isotopy, as this will make it easier to see how the labeled intersections for $L_{u/v}$ relate to the ones for $\tau_{(u-v)/v}$. If $\alpha_{u/v}$ is given the orientation from $X^-$ to $X^+$, then the figure below shows what happens after applying the untwisting isotopy to one of the rectangles of intersection points from $\mathfrak{A}$ and $\mathfrak{B}$ formed by $\alpha_{u/v}$ passing down through $\beta'$. We have also included the part of the blue Lagrangian around $X^+$ and the intersection point $\lambda_\omega$.

\[
\vcenter{\hbox{\includegraphics[height=4cm,angle=0]{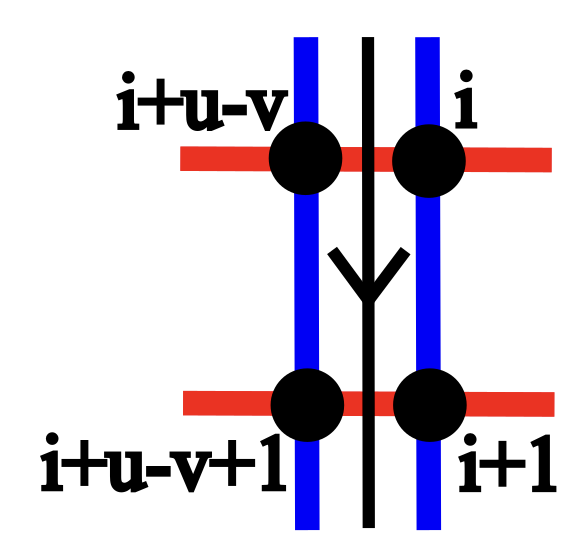}}}
\xrightarrow{\text{isotopy}}
\vcenter{\hbox{\includegraphics[height=7cm,angle=0]{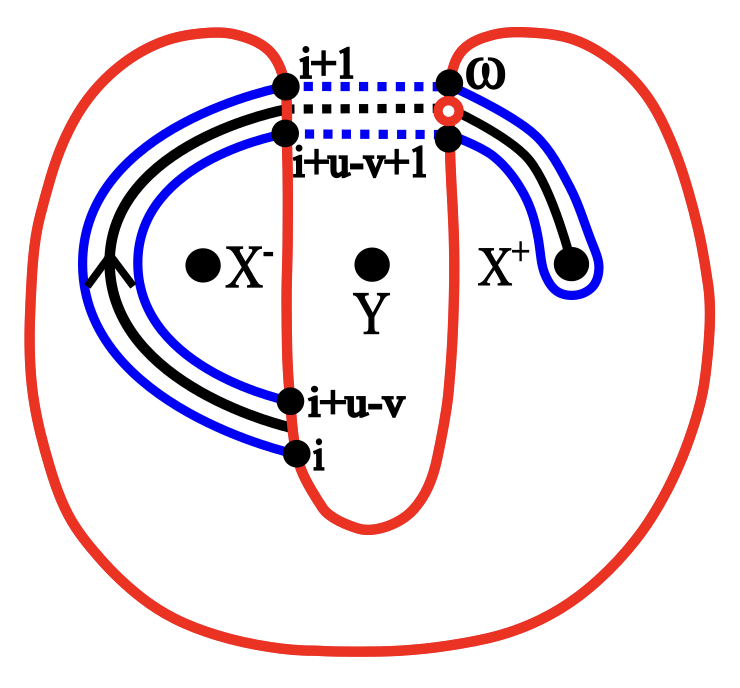}}}
\]

It is not too difficult to see that the rule for labeling intersection points given previously results in the labeling as shown, where the left-middle vertical part of $\beta'$ imitates $l_A$ in that its intersection with $\alpha_{(u-v)/v}$ on the bottom matches with $\xi_i$ and the intersection on the top matches with $\xi_{i+1}$. Thus, for $k\in\{i,i+1\}$, it is easy to see that
\[
\begin{cases}
    \Psi_{X^+}([\gamma^L_k])-\Psi_{X^+}([\gamma^T_k])=1 \\
    \Psi_{X^-}([\gamma^L_k])-\Psi_{X^-}([\gamma^T_k])=1 \\
    \Psi_{Y}([\gamma^L_k])-\Psi_{Y}([\gamma^T_k])=1
\end{cases}
\qquad
\text{and} 
\qquad
\begin{cases}
    \Psi_{X^+}([\gamma^L_{k+u-v}])-\Psi_{X^+}([\gamma^T_k])=0 \\
    \Psi_{X^-}([\gamma^L_{k+u-v}])-\Psi_{X^-}([\gamma^T_k])=1 \\
    \Psi_{Y}([\gamma^L_{k+u-v}])-\Psi_{Y}([\gamma^T_k])=1.
\end{cases}
\]
Following the same procedure in the case where $\alpha_{u/v}$ passes through one of the rectangles going the opposite direction yields the same result. It is also not difficult to see that we get the same equations when $k=1$. This proves the lemma for the $\mathfrak{A}$ and $\mathfrak{B}$ blocks (recall the $\delta_{i,A}\delta_{\omega,I}$ term in the formula of $S$ for tangles).

Now, we consider the blocks $\mathfrak{C}$ and $\mathfrak{D}$---the ones coming from the inactive side for $\tau_{(u-v)/v}$. We will follow the same procedure. In this case, if $\alpha_{u/v}$ passes up through one of the rectangles coming from the inactive side, then the picture we get from the untwisting isotopy looks like
\[
\vcenter{\hbox{\includegraphics[height=4cm,angle=0]{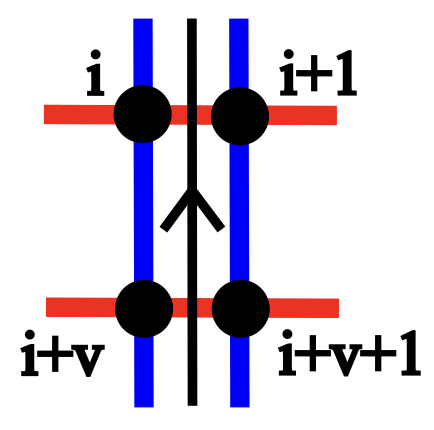}}}
\xrightarrow{\text{isotopy}}
\vcenter{\hbox{\includegraphics[height=7cm,angle=0]{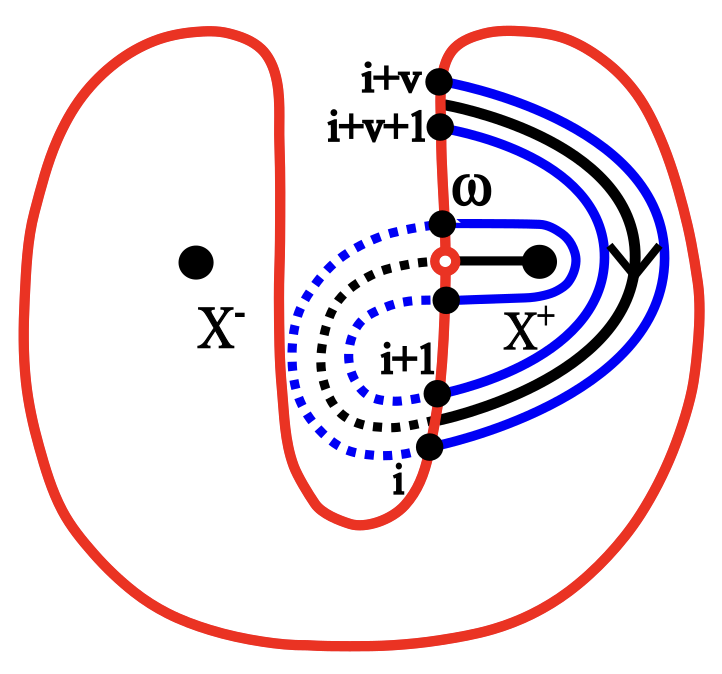}}}.
\]
Here, notice that the right-middle part of the red Lagrangian imitates $l_I$ so that its intersection on the top with $\alpha_{(u-v)/v}$ corresponds with some $\xi_j$ and the bottom intersection with $\xi_{j+1}$ (using the orientation induced by $\prec$). Abusing notation slightly, we will call $j$ ``i'' so that we can write the same sorts of equations we did for the active side. Then, for $k\in\{i,i+1\}$, we see
\[
\begin{cases}
    \Psi_{X^+}([\gamma^L_k])-\Psi_{X^+}([\gamma^T_k])=0 \\
    \Psi_{X^-}([\gamma^L_k])-\Psi_{X^-}([\gamma^T_k])=1 \\
    \Psi_{Y}([\gamma^L_k])-\Psi_{Y}([\gamma^T_k])=0
\end{cases}
\qquad
\text{and} 
\qquad
\begin{cases}
    \Psi_{X^+}([\gamma^L_{k+v}])-\Psi_{X^+}([\gamma^T_k])=0 \\
    \Psi_{X^-}([\gamma^L_{k+v}])-\Psi_{X^-}([\gamma^T_k])=0 \\
    \Psi_{Y}([\gamma^L_{k+v}])-\Psi_{Y}([\gamma^T_k])=0
\end{cases}.
\]
Once again, this can also be easily determined if $\alpha_{(u-v)/v}$ is pointing down through the rectangle and when $k=2u-v$, which proves the lemma for the $\mathfrak{C}$ and $\mathfrak{D}$ blocks.
\end{proof} 

All that remains for us to show now is that the off-diagonal entries of $Q$ agree with the one from Lemma \ref{WSclosureformulaL} or can be permuted to agree by Theorem \ref{permthm}. The sequence of steps will be similar to what we did for knots in Section \ref{pfknotssec}.

Similar to what we did in Section \ref{knotssec}, we will need to compare matrices determined by fixed orderings of the $\lambda_i$ and $\xi_i$. We will reuse some notation from the last section to emphasize the parallels between the knots and links, but it should be noted how our new matrices differ from the ones in Section \ref{knotssec}. 

\begin{definitionN}
    For a rational link $L_{u/v}=\text{Cl}(T\tau_{(u-v)/v})$, let $Q_{(u-v)/v}^T$ denote the $u\times u$ matrix for $\tau_{(u-v)/v}$ computed by Proposition \ref{tangleHOMFLYpoly} with $\mathfrak{X}$ and $\mathfrak{Y}$ independently ordered by $\prec$ with elements of $\mathfrak{X}$ coming before elements of $\mathfrak{Y}$.
\end{definitionN}

Thus, $Q_{(u-v)/v}^T$ is the matrix in Lemma \ref{WSclosureformulaL} coming from the tangle side, where we have now determined the blocks for $\mathfrak{X}$ and $\mathfrak{Y}$. We can then apply Lemma \ref{WSclosureformula} to get the following $2u\times 2u$ matrix.

\begin{definitionN}
    Let $Q_{u/v}$ be the $2u \times 2u$ matrix obtained by applying Lemma \ref{WSclosureformulaL} to $Q_{(u-v)/v}^T$.
\end{definitionN}

Observe that $Q_{u/v}$ has the same $4\times 4$ block structure as $Q_{u/v}^L$, so we can compare blocks the same way we did with $Q_{u/v}$ and $Q_{u/v}^K$ in Section \ref{knotssec}.

\begin{lemmaN}
  If $L_{u/v}=\text{Cl}(T\tau_{(u-v)/v})$, where $\tau_{(u-v)/v}$ has $UP$ orientation, then the top left $2\times 2$ blocks of $Q_{u/v}^L$ and $Q_{u/v}$ agree. If $\tau_{(u-v)/v}$ has $RI$ orientation, then the same holds for the lower right $2\times 2$ blocks.  
\end{lemmaN}

\begin{proof}
    Consider the $UP$ case---the $RI$ case is essentially the same, as usual. This is trivial to see for the $(1,1)$ and the $(2,2)$ blocks. In particular, once the untwisting isotopy is applied, the loops in $\text{Conf}^2(M)$ used to compute the off-diagonal entries are perfectly parallel to the corresponding ones used to compute the same entries in the $(1,1)$-block, or $Q_{++}$, for $\tau_{(u-v)/v}$. Then, we are done for these blocks since we already know that the diagonal entries have the appropriate shift.

    The argument that the $(1,2)$- and $(2,1)$- blocks of $Q_{u/v}^L$ and $Q_{u/v}$ also agree is similar to the one used in the proof of Lemma \ref{Qplusblockknots}. 
\end{proof}

Now, it remains to show that the remaining block matrices can be permuted according to Theorem \ref{permthm} to agree with Theorem \ref{WSclosureformulaL}. Recall Lemma \ref{tQmswap}, where we applied permutations to the $Q_{--}$ block in $Q_{(u-v)/v}^T$ if $\tau_{(u-v)/v}$ has $UP$ orientation. We will need to use the same result in this section. The $Q_{(u-v)/v}^T$ used in the last section is different from the one used here, but the blocks being permuted in $Q_{(u-v)/v}^T$ are the same, and the same argument used in the proof of Lemma \ref{tQmswap} applies to our new matrices.

\begin{definitionN}
    Let $\widetilde{Q}_{(u-v)/v}^T$ be the result of applying the permutations of Lemma \ref{tQmswap} to $Q_{(u-v)/v}^T$. Furthermore, let $\widetilde{Q}_{u/v}$ be the $2u \times 2u$ matrix obtained by applying the $\text{Cl}(T-)$ operation of Lemma \ref{WSclosureformulaL} to $\widetilde{Q}_{(u-v)/v}$.
\end{definitionN}

It is important to note that, in this new setting, the $Q_{--}$ block for $Q_{(u-v)/v}^T$ splits into the lower right $2\times 2$ blocks of $Q_{u/v}$, so $\widetilde{Q}_{u/v}$ has these same permutations within each of those blocks.

\begin{lemmaN}
\label{Lbottomblockperm}
    If $L_{u/v}=\text{Cl}(T\tau_{(u-v)/v})$ where $\tau_{(u-v)/v}$ has $UP$ orientation, then the lower right $2\times 2$ block of matrices in $Q_{u/v}^L$ and $\widetilde{Q}_{u/v}$ agree, up to an overall reordering of the indices. The result is similar if $\tau_{(u-v)/v}$ has $RI$ orientation, but with the top left $2\times 2$ block of matrices.
\end{lemmaN}

\begin{proof}

\begin{figure}
    \centering
    \includegraphics[height=10cm]{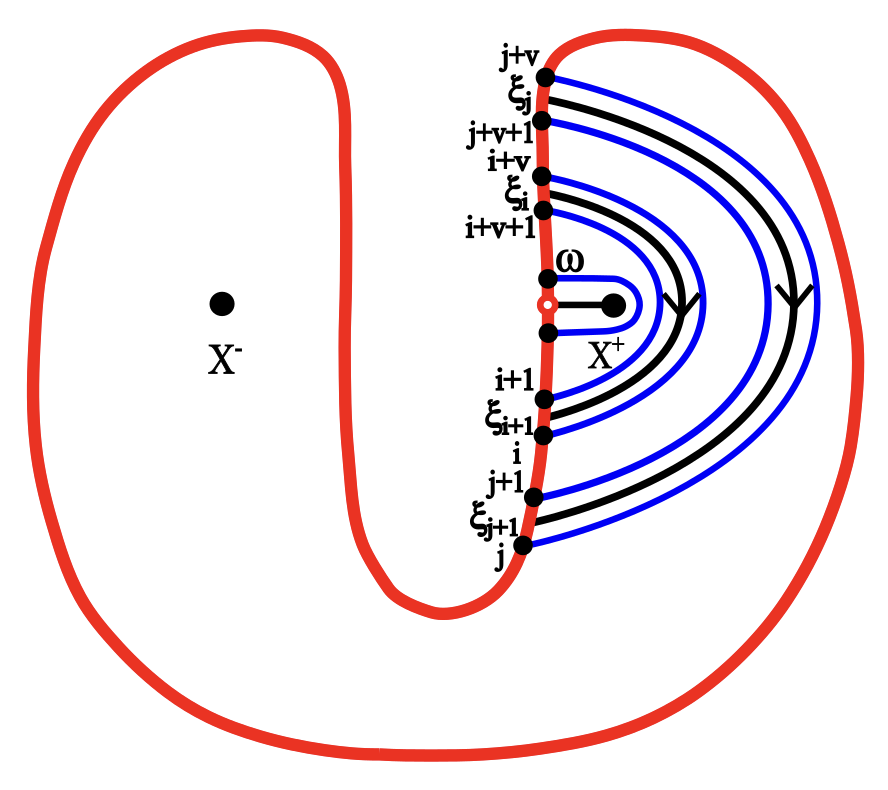}
    \caption{The geometric picture for considering off-diagonal entries of $Q$ in the proof of Lemma \ref{Lbottomblockperm} (with particular choice of orientation)}
    \label{LQmmperm}
\end{figure}

    Considering the $UP$ case, if $\lambda_i,\lambda_{i+1},\lambda_{i+v},$ and $\lambda_{i+v+1}$ are related as in the second figure from the proof of Lemma \ref{linkslinform}, a quick computation shows that the corresponding $4\times 4$ submatrix of $Q_{u/v}^L$ agrees with the same submatrix of $\widetilde{Q}_{u/v}$, regardless of the orientation on $\alpha_{(u-v)/v}$. In particular, we have

    \begin{equation}
\begin{blockarray}{ccc}
    & i & i+1  \\
    \begin{block}{c[cc]}
        i & A+2 & A\\
        i+1 & A & A-1\\
    \end{block}
\end{blockarray}
\xrightarrow{Cl(T-)}
\begin{blockarray}{ccccc}
    & i & i+1 & i+v & i+v+1 \\
    \begin{block}{c[cccc]}
        i & A+3 & A+1 & A+2 & A \\
        i+1 & A+1 & A & A+1 & A-1\\
        i+v & A+2 & A+1 & A+2 & A\\
        i+v+1 & A & A-1 & A & A-1\\
    \end{block}
\end{blockarray}
\end{equation}
or
    \begin{equation}
\begin{blockarray}{ccc}
    & i & i+1  \\
    \begin{block}{c[cc]}
        i & A-1 & A\\
        i+1 & A & A+2\\
    \end{block}
\end{blockarray}
\xrightarrow{Cl(T-)}
\begin{blockarray}{ccccc}
    & i & i+1 & i+v & i+v+1 \\
    \begin{block}{c[cccc]}
        i & A & A+1 & A-1 & A \\
        i+1 & A+1 & A+3 & A+1 & A+2\\
        i+v & A-1 & A+1 & A-1 & A\\
        i+v+1 & A & A+2 & A & A+2\\
    \end{block}
\end{blockarray}
\end{equation}
depending on the orientation of $\alpha_{(u-v)/v}$ along the arc between the tangle intersection points $\xi_i$ and $\xi_{i+1}$, where the computations from Theorem \ref{Linkthm} agrees with the matrices on the right.

We must still consider entries in this submatrix of $Q_{u/v}^L$ that aren't related geometrically in this simple way. It is easy to check that any entry of $Q_{u/v}^L$ with row or column indexed by $\omega$ or its paired point next to $X^+$ agrees with the corresponding entries of $\widetilde{Q}_{u/v}$. Any other sort of entry fits into an $8\times 8$ submatrix given geometrically by a picture as shown in Figure \ref{LQmmperm}. There are four possible ways the two parallel arcs can be oriented, but only one will be considered explicitly.

In order to guarantee the ``$+L$'' in the $(3,4)$-block of $Q$ (or the ``$+U$'' in the $(4,3)$-block), we will need to reorder the $\lambda_i$ from the standard ordering. In fact, the reordering of the $\lambda_i$ can be defined in terms of a reordering of the 
inactive $\xi_i$ because we want the same reordering within $\mathfrak{C}$ and $\mathfrak{D}$. In particular, if the pairs $(\xi_i,\xi_{i+1})$ and $(\xi_j,\xi_{j+1})$ are such that the arc between the second pair wraps around the arc connecting the first pair, then we reorder the indices so that $j$ and $j+1$ come before $i$ and $i+1$. By the doubling of $\mathfrak{Y}$ to $\mathfrak{C}$ and $\mathfrak{D}$, this reordering extends to a reordering within $\mathfrak{C}$ and $\mathfrak{D}$.

Assuming we have this ordering fixed, we can compute what we expect for the $8\times 8$ submatrix of $Q$ with rows and columns indexed by $\lambda_j,\lambda_{j+1},\lambda_i,\lambda_{i+1},$ and the shifts of these by $v$. The precise way that this matrix looks depends on the orientations of the two relevant arcs, but in the case that both are pointing down, as in Figure \ref{LQmmperm}, Lemma \ref{WSclosureformulaL} gives

\begin{multline}
\begin{blockarray}{ccccc}
    & j & j+1 & i & i+1 \\
    \begin{block}{c[cccc]}
        j & B-1 & B & \textcolor{red}{C-1} & \textcolor{red}{C+1}\\
        j+1 & B & B+2 & \textcolor{red}{C} & \textcolor{red}{C+2}\\
        i & \textcolor{red}{C-1} & \textcolor{red}{C} & A-1 & A\\
        i+1 & \textcolor{red}{C+1} & \textcolor{red}{C+2} & A & A+2\\
    \end{block}
\end{blockarray}\\
\xrightarrow{Cl(T-)}
\begin{blockarray}{ccccccccc}
    & j & j+1 & i & i+1 & j+v & j+v+1 & i+v & i+v+1\\
    \begin{block}{c[cccccccc]}
        j & B & B+1 & \textcolor{red}{C} & \textcolor{red}{\underline{C+2}} & B-1 & B & \textcolor{red}{C-1} & \textcolor{red}{\underline{C+1}}\\
        j+1 & B+1 & B+3 & \textcolor{red}{\underline{C+1}} & \textcolor{red}{C+3} & B+1 & B+2 & \textcolor{red}{\underline{C}} & \textcolor{red}{C+2} \\
        i & \textcolor{red}{C} & \textcolor{red}{\underline{C+1}} & A & A+1 & \textcolor{red}{C} & \textcolor{red}{\underline{C+1}} & A-1 & A \\
        i+1 & \textcolor{red}{\underline{C+2}} & \textcolor{red}{C+3} & A+1 & A+3 & \textcolor{red}{\underline{C+2}} & \textcolor{red}{C+3} & A+1 & A+2\\
        j+v & B-1 & B+1 & \textcolor{red}{C} & \textcolor{red}{\underline{C+2}} & B-1 & B & \textcolor{red}{C-1} & \textcolor{red}{\underline{C+1}}\\
        j+v+1 & B & B+2 & \textcolor{red}{\underline{C+1}} & \textcolor{red}{C+3} & B & B+2 & \textcolor{red}{\underline{C}} & \textcolor{red}{C+2}\\
        i+v & \textcolor{red}{C-1} & \textcolor{red}{\underline{C}} & A-1 & A+1 & \textcolor{red}{C-1} & \textcolor{red}{\underline{C}} & A+1 & A\\
        i+v+1 & \textcolor{red}{\underline{C+1}} & \textcolor{red}{C+2} & A & A+2 & \textcolor{red}{\underline{C+1}} & \textcolor{red}{C+2} & A & A+2\\
    \end{block}
\end{blockarray}
\end{multline}
so that the $8\times 8$ matrix is a submatrix of the lower right $2\times 2$ blocks of $\widetilde{Q}_{u/v}$. Now, a tedious (but easy) calculation using Theorem \ref{Linkthm} shows that one gets the same submatrix of $Q_{u/v}^L$ if $\xi_j\prec\xi_i$; if $\xi_i\prec\xi_j$, then the underlined entries in each of the red $2\times 2$ blocks must be permuted. However, permuting these entries in the corresponding $4\times 4$ submatrix of $Q_{(u-v)/v}^T$ to get $\widetilde{Q}_{(u-v)/v}^T$ remedies this issue. The cases involving different orientations of the two arcs are similar.
\end{proof}

Finally, we need to consider how the remaining blocks of $Q_{u/v}^L$ and $\widetilde{Q}_{u/v}$ relate to each other. In particular, we still need to handle the top-right and bottom-left $2\times 2$ blocks, which, in the case of $\widetilde{Q}_{u/v}$ are shifts of $Q_{+-}$ or $Q_{-+}$. We want to show that we can permute entries from these blocks in $\widetilde{Q}_{u/v}$ to get $Q_{u/v}^L$, so that $Q_{u/v}^L\sim \widetilde{Q}_{u/v}$. This leads us to the final step in the proof of Theorem \ref{Linkthm}.


\begin{lemmaN}
\label{LQpmperm}
    The matrices $Q_{u/v}^L$ and $Q_{u/v}$ satisfy $Q_{u/v}^L\sim Q_{u/v}$
\end{lemmaN}

\begin{proof}
As was indicated above, we need to show $Q_{u/v}^L\sim \widetilde{Q}_{u/v}$. Then the lemma (and Theorem \ref{Linkthm}) will follow since $\widetilde{Q}_{u/v}\sim Q_{u/v}$.

    First, we describe what permutations are required. If $\tau_{(u-v)/v}$ has $UP$ orientation, then each active intersection point $\xi_i$ doubles to give two intersection points $\lambda_{i'}$ and $\lambda_{i''}$ for $L_{u/v}$ such that $\lambda_{i'}\in \mathfrak{A}$ and $\lambda_{i''}\in\mathfrak{B}$ (so $i''-i'=u-v$), and each inactive intersection point $\xi_j$ is paired with two intersection points $\lambda_{j'}$ and $\lambda_{j''}$ for $L_{u/v}$ such that both are in $\mathfrak{C}$ or $\mathfrak{D}$ and $|j''-j'|=1$. We cannot say much more than this without considering several cases, due to the complexity of the system for labeling the $\lambda$'s. However, assuming we have applied the untwisting isotopy, if we cut $\overline{\alpha_{(u-v)/v}}$ near $X^+$ to get two parallel copies of $\alpha_{(u-v)/v}$ as we did in Section \ref{Knotquadformsec}, then we can say the following: if we assume $\lambda_{i'}$ and $\lambda_{j'}$ are on one of these parallel copies, and $\lambda_{i''}$ and $\lambda_{j''}$ are on the other, then the full set of required permutations are of the form $Q_{i'j''}\leftrightarrow Q_{i''j'}$ (and $Q_{j''i'}\leftrightarrow Q_{j'i''}$) where $\xi_j\prec\xi_i$. This is easy to verify by drawing pictures and applying the same arguments we have used numerous times already.

    To complete the proof, we need to determine an ordering of these permutations that is allowed by Theorem \ref{permthm}. For the sake of brevity, we will only provide the rule for determining the order. Suppose one needs to apply the two pairs of permutations $Q_{a'b''}\leftrightarrow Q_{a''b'}$/$Q_{b''a'}\leftrightarrow Q_{b'a''}$ and $Q_{c'd''}\leftrightarrow Q_{c''d'}$/$Q_{d''c'}\leftrightarrow Q_{d'c''}$ (so $\xi_b\prec \xi_a$ and $\xi_d\prec\xi_c$, with $\xi_a$ and $\xi_c$ active, and the other two inactive). Call these the $(a,b)$ and $(c,d)$ permutations, respectively. Then, if $\lambda_{b'}$ and $\lambda_{d'}$ are in the same block ($\mathfrak{C}$ or $\mathfrak{D}$), and if $b=d$, then $(a,b)$ should be performed before $(c,d)$ if $\xi_c\prec\xi_a$; if $b \neq d$, then $(a,b)$ should be performed before $(c,d)$ if $\xi_b\prec\xi_d$. 
    
    Otherwise, $\lambda_{b'}$ and $\lambda_{d'}$ are not in the same block. Without loss of generality, assume $\lambda_{b'},\lambda_{b''}\in \mathfrak{C}$ and $\lambda_{d'},\lambda_{d''}\in \mathfrak{D}$. If $d'-b'=d''-b''=v$ and we fix $b',b'',d',$ and $d''$ so that $b'-b''=d'-d''=1$, then the $(a,b)$ permutation is of the form $Q_{a'b''}\leftrightarrow Q_{((a'\pm (u-v))(b''+1))}$/$Q_{((b''+1)(a'\pm (u-v)))}\leftrightarrow Q_{b'a''}$, and similarly for $(c,d)$; the $(a,b)$ permutation should be performed first if $a''=a'+u-v$, and $(c,d)$ should be performed first if $a''=a-u+v$ (the same holds if we use $c''$ instead of $a''$). Otherwise, if we do not have this condition on $b',b'',d',$ and $d''$, then $(a,b)$ should be performed before $(c,d)$ if $\xi_b\prec\xi_d$. These rules are sufficient for determining the order of permutations giving $Q_{u/v}^L\sim \widetilde{Q}_{u/v}$.

\end{proof}

This concludes the proof of Theorem \ref{Linkthm}. As in Section \ref{knotssec}, we will now provide some examples.

\subsection{Example: Rational Torus Links}

By rational torus links, we mean the two-component links $L_{2u/1}=T(2,2u)$. Just as we saw for rational torus knots, Theorem \ref{Linkthm} predicts an easy formula for $S,A,$ and $Q$ in terms of the matrices we get for $\tau_{(2u-1)/1}$, which we briefly recall now.

\begin{figure}
    \centering
    \includegraphics[height=5cm]{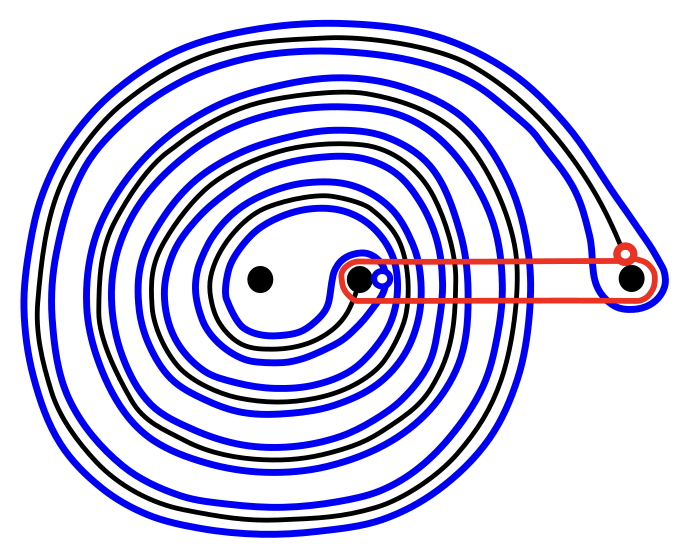}
    \caption{$\mathcal{D}(L_{8/1})$}
    \label{L81fig}
\end{figure}

In Example 6.2 of \cite{JHpI26}, it was shown that, given the same ordering of $\mathcal{G}_{n/1}^1=\mathfrak{X}\sqcup\mathfrak{Y}$ that we described earlier in Section \ref{linkssec} (order the active intersections, $\mathfrak{X}$, and inactive intersections, $\mathfrak{Y}$, independently with respect to $\prec$ with $\mathfrak{X}$ coming before $\mathfrak{Y}$), Proposition \ref{tangleHOMFLYpoly} gives 
\[
    \begin{cases}
        S_i=n-i+1\\
        A_i=0\\
        \begin{cases}
            Q_{ij}=n-k, \qquad & \text{if} \, (i,j)\in\mathcal{I}_k, 1 \leq k \leq n\\
            Q_{ij}=0,\qquad & \text{if} \, (i,j)\in\mathcal{I}_k, k=n+1,
        \end{cases}
    \end{cases}
    \]
where $\mathcal{I}_k$ is the same set defined in Section \ref{rattorusknots}. Observe that $|\mathfrak{X}|=n$ and $|\mathfrak{Y}|=1$.

Consider $L_{8/1}=\text{Cl}(T\tau_{7/1})$ as a concrete example. For the tangle $\tau_{7/1}$, we get

  \begin{figure}[H]
    \centering
\begin{minipage}[b]{.45\textwidth}

$\left[\begin{array}{ c | c }
   S & A \end{array}\right] =
  \left[\begin{array}{ c | c }
    7 & 0 \\
    6 & 0 \\
    5 & 0 \\
    4 & 0 \\
    3 & 0 \\
    2 & 0 \\
    1 & 0 \\
    \hline
    0 & 0 \\
   \end{array}\right]$
    \end{minipage}
    \begin{minipage}[b]{.45\textwidth}
    $ Q=\left[\begin{array}{ c c c c c c c | c }
    6 & 5 & 4 & 3 & 2 & 1 & 0 & 0\\
    5 & 5 & 4 & 3 & 2 & 1 & 0 & 0\\
    4 & 4 & 4 & 3 & 2 & 1 & 0 & 0\\
    3 & 3 & 3 & 3 & 2 & 1 & 0 & 0\\
    2 & 2 & 2 & 2 & 2 & 1 & 0 & 0\\
    1 & 1 & 1 & 1 & 1 & 1 & 0 & 0\\
    0 & 0 & 0 & 0 & 0 & 0 & 0 & 0\\
    \hline
    0 & 0 & 0 & 0 & 0 & 0 & 0 & 0
    \end{array}\right].$
    \end{minipage}
\end{figure}

by Proposition \ref{tangleHOMFLYpoly}. Then, if we apply Lemma \ref{WSclosureformulaL}, we get

\[
\left[\begin{array}{ c | c }
   S & A \end{array}\right]=\left[\begin{array}{ c c c c c c c | c c c c c c c | c | c}
    9 & 8 & 7 & 6 & 5 & 4 & 3 & 8 & 7 & 6 & 5 & 4 & 3 & 2 & 1 & 0\\
    \hline
    2 & 2 & 2 & 2 & 2 & 2 & 2 & 0 & 0 & 0 & 0 & 0 & 0 & 0 & 0 & 0
    \end{array}\right]^T
\]
and 
\[
Q_{8/1}= 
\left[\begin{array}{ c c c c c c c | c c c c c c c | c | c}
    5 & 4 & 3 & 2 & 1 & 0 & -1 & 6 & 6 & 5 & 4 & 3 & 2 & 1 & 0 & -1\\
    4 & 4 & 3 & 2 & 1 & 0 & -1 & 5 & 5 & 5 & 4 & 3 & 2 & 1 & 0 & -1\\
    3 & 3 & 3 & 2 & 1 & 0 & -1 & 4 & 4 & 4 & 4 & 3 & 2 & 1 & 0 & -1\\
    2 & 2 & 2 & 2 & 1 & 0 & -1 & 3 & 3 & 3 & 3 & 3 & 2 & 1 & 0 & -1\\
    1 & 1 & 1 & 1 & 1 & 0 & -1 & 2 & 2 & 2 & 2 & 2 & 2 & 1 & 0 & -1\\
    0 & 0 & 0 & 0 & 0 & 0 & -1 & 1 & 1 & 1 & 1 & 1 & 1 & 1 & 0 & -1\\
    -1 & -1 & -1 & -1 & -1 & -1 & -1 & 0 & 0 & 0 & 0 & 0 & 0 & 0 & 0 & -1\\
    \hline
    6 & 5 & 4 & 3 & 2 & 1 & 0 & 8 & 7 & 6 & 5 & 4 & 3 & 2 & 1 & 0\\
    6 & 5 & 4 & 3 & 2 & 1 & 0 & 7 & 7 & 6 & 5 & 4 & 3 & 2 & 1 & 0\\
    5 & 5 & 4 & 3 & 2 & 1 & 0 & 6 & 6 & 6 & 5 & 4 & 3 & 2 & 1 & 0\\
    4 & 4 & 4 & 3 & 2 & 1 & 0 & 5 & 5 & 5 & 5 & 4 & 3 & 2 & 1 & 0\\
    3 & 3 & 3 & 3 & 2 & 1 & 0 & 4 & 4 & 4 & 4 & 4 & 3 & 2 & 1 & 0\\
    2 & 2 & 2 & 2 & 2 & 1 & 0 & 3 & 3 & 3 & 3 & 3 & 3 & 2 & 1 & 0\\
    1 & 1 & 1 & 1 & 1 & 1 & 0 & 2 & 2 & 2 & 2 & 2 & 2 & 2 & 1 & 0\\
    \hline
    0 & 0 & 0 & 0 & 0 & 0 & 0 & 1 & 1 &1 & 1 & 1 & 1 & 1 & 1 & 0\\
    \hline
    -1 & -1 & -1 & -1 & -1 & -1 & -1 & 0 & 0 & 0 & 0 & 0 & 0 & 0 & 0 & 0
    \end{array}\right].
\]
We can also use Theorem \ref{Linkthm} to compute $S,A,$ and $Q$, and it turns out that these agree exactly with the data given above. Importantly, $Q_{8/1}^L=Q_{8/1}$, like we had for knots. The analogy with rational torus knots extends further to the following proposition.

\begin{propositionN}
\label{toruslinkprop}
For a rational torus link $L_{2n/1}$, we have $Q_{2n/1}^L=Q_{2n/1}$.
\end{propositionN}

Just as we said for Proposition \ref{ratknotforms}, this proposition follows immediately by observing that the simplicity of this family of links means that we do not need to apply any of the permutations to $Q_{2n/1}$ discussed in the proof of Theorem \ref{Linkthm}. 

\subsection{Example: $L_{8/3}$}

To conclude the section, we will consider a link obtained by applying $\text{Cl}(T-)$ to a rational tangle with $RI$ orientation. In particular, we will look at $L_{8/3}=\text{Cl}(T\tau_{5/3}$). Many of the figures early in this section used this link as an example, so it is recommended to refer back to Figures \ref{L83} and \ref{L83untwistfig} to see $\mathcal{D}(L_{8/3})$ and its untwisted version, respectively. Previously, we also saw how to label the $\lambda_i$'s via the standard ordering.

Here, we will only consider the matrix $Q_{8/3}^L$ because there are no complications with $S$ and $A$. We will ignore the correction terms, but observe that $\mu_1(L_{8/3})=-2, \mu_2(L_{8/3})=-1,$ and $\mu_3(L_{8/3})=1$. Computing $Q_{8/3}^L$  gives

\[
Q_{8/3}^L= 
\left[\begin{array}{ c c c c c | c c c c c | c c c | c c c}
    3 & 1 & 0 & 2 & 1 & 2 & 0 & \textcolor{red}{0} & \textcolor{red}{2} & 0 & 2 & \underline{0} & \underline{0} & 3 & 2 & 2\\
    1 & 0 & -1 & 1 & 0 & 1 & -1 & \textcolor{red}{-1} & \textcolor{red}{1} & -1 & 0 & -1 & -1 & 1 & \underline{1} & \underline{1}\\
    0 & -1 & -1 & 0 & 0 & \textcolor{red}{-1} & \textcolor{red}{-2} & -2 & -1 & -1 & 0 & -1 & -1 & 1 & \underline{1} & \underline{1}\\
    2 & 1 & 0 & 2 & 1 & \textcolor{red}{1} & \textcolor{red}{0} & 0 & 1 & 0 & 2 & \underline{0} & \underline{0} & 3 & 2 & 2\\
    1 & 0 & 0 & 1 & 1 & 1 & 0 & 0 & 1 & 0 & 0 & 0 & 0 & 1 & 1 & 1\\
    \hline
    2 & 1 & \textcolor{red}{-1} & \textcolor{red}{1} & 1 & 2 & 0 & -1 & 1 & 0 & 1 & \underline{-1} & \underline{-1} & 2 & 1 & 1\\
    0 & -1 & \textcolor{red}{-2} & \textcolor{red}{0} & 0 & 0 & -1 & -2 & 0 & -1 & -1 & -2 & -2 & 0 & \underline{0} & \underline{0}\\
    \textcolor{red}{0} & \textcolor{red}{-1} & -2 & 0 & 0 & -1 & -2 & -2 & -1 & -1 & -1 & -2 & -2 & 0 & \underline{0} & \underline{0}\\
    \textcolor{red}{2} & \textcolor{red}{1} & -1 & 1 & 1 & 1 & 0 & -1 & 1 & 0 & 1 & \underline{-1} & \underline{-1} & 2 & 1 & 1\\
    0 & -1 & -1 & 0 & 0 & 0 & -1 & -1 & 0 & 0 & -1 & -1 & -1 & 0 & 0 & 0\\
    \hline
    2 & 0 & 0 & 2 & 0 & 1 & -1 & -1 & 1 & -1 & 1 & 0 & -1 & 2 & 2 & 1\\
    \underline{0} & -1 & -1 & \underline{0} & 0 & \underline{-1} & -2 & -2 & \underline{-1} & -1 & 0 & 0 & -1 & 1 & 1 & 1\\
    \underline{0} & -1 & -1 & \underline{0} & 0 & \underline{-1} & -2 & -2 & \underline{-1} & -1 & -1 & -1 & -1 & 0 & 0 & 0\\
    \hline
    3 & 1 & 1 & 3 & 1 & 2 & 0 & 0 & 2 & 0 & 2 & 1 & 0 & 4 & 3 & 2\\
    2 & \underline{1} & \underline{1} & 2 & 1 & 1 & \underline{0} & \underline{0} & 1 & 0 & 2 & 1 & 0 & 3 & 3 & 2\\
    2 & \underline{1} & \underline{1} & 2 & 1 & 1 & \underline{0} & \underline{0} & 1 & 0 & 1 & 1 & 0 & 2 & 2 & 2
    \end{array}\right].
\]
without the shift up by $1$. The red parts denote the $2\times 2$ submatrices that get in the way of having the ``$+L$'' in the $(1,2)$-block and the ``$+U$'' in the $(2,1)$-block. This is resolved in the proof of Lemma \ref{Lbottomblockperm} by reordering the $\lambda_i$'s (the same principles used in the $UP$ case apply in the $RI$ case). In particular, we want to reorder by $\lambda_1\leftrightarrow \lambda_3$, $\lambda_2\leftrightarrow \lambda_4$, $\lambda_6\leftrightarrow \lambda_8$, and $\lambda_7\leftrightarrow \lambda_9$. This reordering has the effect of conjugating $Q_{8/3}^L$ by a permutation matrix, which clearly does not change the generating function, and it yields the desired structure for the $(1,2)$- and $(2,1)$-blocks.

Additionally, the underlined entries are the ones that still need to be permuted to agree with Lemma \ref{WSclosureformulaL}. Adapting the proof of Lemma \ref{LQpmperm} to the $RI$ case shows that the necessary permutations are the ones we should expect and they can be done in an order that satisfies the conditions of Theorem \ref{permthm}.

\section{Symmetric-Colored Polynomial Invariants}
\label{sympolysec}

\subsection{Symmetric-Colored HOMFLY-PT Polynomials}

In this paper, we have been focusing on the antisymmetric-colored HOMFLY-PT polynomials or, equivalently, the ones colored by single-column Young diagrams. However, we could have just as easily thought about the symmetric-colored polynomials, or the ones colored by single-row Young diagrams. 

\begin{definitionN}
    Given a link $L$ and a vector $\vec{\lambda}$ denoting the colorings of its components, let $\overline{P}^{\vec{\lambda}}_L(q,a)$ be the unreduced colored HOMFLY-PT polynomial for $L$ specified by $\vec{\lambda}$.
\end{definitionN}

This is a strong generalization of what we have been working with in this paper---one thing we have been assuming is that, when working with links, the components have the same coloring. In \cite{LZ10}, the authors discuss these colored HOMFLY-PT polynomials in this greater generality.

Let $\vec{\lambda}^T$ be the vector of Young diagrams obtained by taking the transpose of each Young diagram in $\vec{\lambda}$. Then, the relationship between $\overline{P}^{\vec{\lambda}}_L$ and $\overline{P}^{\vec{\lambda}^T}_L$, is given by the following Theorem.

\begin{theoremN}[Theorem 4.4, \cite{TVW19}]
    Given a link $L$ with $n$ components colored according to a vector $\vec{\lambda}=(\lambda_1,...,\lambda_n)$ of Young diagrams, then the following equation holds:
    \begin{equation}
        \overline{P}^{\vec{\lambda}}_L(q,a)=(-1)^c \overline{P}^{\vec{\lambda}^T}L(q,a)\bigg|_{q\mapsto q^{-1}},
    \end{equation}
    where $c=|\lambda_1|+...+|\lambda_i|$, the sum of the number of nodes in the Young diagrams.
\end{theoremN}

It should be noted that the theorem is stated in terms of unreduced colored HOMFLY-PT polynomials, but we have been working with the reduced ones. In the reduced context, the theorem implies that the antisymmetric- and symmetric-colored HOMFLY-PT polynomials are related by inverting $q$.

Given our specialized context in this paper, we have been writing $P_{u/v}^{\bigwedge^j}(q,a)$ to represent $P_{L_{u/v}}^{\vec{\lambda}}(q,a)$, where now we allow $L_{u/v}$ to be knot, and $\vec{\lambda}=((1)^j,(1)^j)$ if $L_{u/v}$ has two components; if it is a knot, we only need $\lambda=(1)^j$, the single-column Young diagram. Similarly, we can use $P_{u/v}^{S^j}(q,a)$ to denote $P_{L_{u/v}}^{\vec{\lambda}^T}(q,a)$, where $\vec{\lambda}$ is the partition for the symmetric-colored polynomials. Then, we have 
\[
P_{u/v}^{S^j}(q,a)=P_{u/v}^{\bigwedge^j}(q^{-1},a).
\]
Notice that in the $j=1$ case, this reflects how the uncolored HOMFLY-PT polynomial is symmetric with respect to inverting $q$, which is accounted for by the correction terms. Furthermore, because of this simple relationship between the antisymmetric- and symmetric-colored polynomials, we can use the same diagrams $\mathcal{D}^j(K_{u/v})$ and $\mathcal{D}^j(L_{u/v})$, for $j\in\{1,2\}$, to compute the quiver forms for the generating functions of the symmetric-colored HOMFLY-PT polynomials, in addition to the antisymmetric ones. The key difference is that the signs of all terms in $S$ and $Q$ must be changed, and all Pochhammer symbols are defined in terms of $q^{-1}$ rather than $q$. However, we can rewrite the Pochhammer symbols in terms of $q$ if we subtract one from all non-diagonal entries of the new $Q$. For example, in the case of a rational knot $K_{u/v}$, we have
\[
\sum_{j\geq 0} P_{u/v}^{S^j}(q,a)x^j=\sum_{\textbf{d}=(d_1,...,d_u)\in \mathbb{N}^u}(-q)^{S'\cdot \textbf{d}}a^{A'\cdot \textbf{d}}q^{\textbf{d} \cdot Q'\cdot \textbf{d}^t} {d_1+...+d_u\brack d_1,...,d_u}x^{\sum_{i=1}^u d_i},
\]
where $S'=-S$, $A'=A$, and $Q'=-Q-1+I$, for $S,A,$ and $Q$ the matrices computed by Theorem \ref{Knotsthm}, where $-1+I$ is the matrix with $0$'s along the diagonal and $-1$'s everywhere else.

\begin{example}
    We can compare our computations for the $T(2,2n+1)$ torus knots with the ones from \cite{K19}. In particular, one can easily determine that $-Q_{(2n+1)/1}^K-1+I$ is equal to the matrix computed for $T(2,2n+1)$ in \cite{K19} up to a reordering of indices and an overall shift up by $2n$.
\end{example}

\subsection{Colored Jones Polynomials}

It is worth saying a quick word about colored Jones polynomials and how they fit into this discussion. These single-variable polynomial invariants arise in many different conjectures from knot theory, including the Volume Conjecture \cite{M10,MM01,Mu25} and the AJ Conjecture \cite{LT11,L06}. They generalize the ordinary Jones polynomial by coloring strands in a link diagram with arbitrary irreducible representations of $U_q(\mathfrak{sl}_N)$. For a knot $K$ with diagram $D$, given one of these irreducible representations $S^j\mathbb{C}^2$, the $j$-fold symmetric product of the vector representation, the corresponding $j$-colored Jones polynomial of $K$ is computed by taking a $j$-cabling of $D$, and inserting the $j$th Jones-Wenzl projector. More details on the invariant can be found in the citations listed above for the two conjectures.

Importantly, the $j$-colored (reduced) Jones polynomial of a knot $K$, which we write as $V_K^j(q)$ is related to $P^{S^j}_K(q,a)$ by
\[
V_K^j(q)=P^{S^j}_K(q,a)\bigg|_{a\mapsto q^2}.
\]
In the context of rational knots, this leads to the following Corollary of Theorem \ref{Knotsthm}, where we are using the same geometric picture in $M$ and $\text{Conf}^2(M)$.

\begin{corollaryN}
\label{coloredJones}
 The generating function for the colored Jones polynomials for the rational knot $K_{u/v}=\text{Cl}(\tau_{u/v})$ can be written as
    \begin{equation}
        \sum_{j\geq 0} V_{K_{u/v}}^j(q)x^j=\sum_{\textbf{d}=(d_1,...,d_u)\in \mathbb{N}^u}(-q)^{H\cdot \textbf{d}}q^{\textbf{d} \cdot Q\cdot \textbf{d}^t} \frac{(q^{-2};q^{-2})j}{\prod_{i=1}^u (q^{-2};q^{-2})_{d_i}}x^{d_1+...+d_u}
    \end{equation}
    where $S$, $H$, and $Q$ may computed by the following formulas:
    \[
    \begin{cases}
        H_i=(3\Psi_{X^+}-\Psi_{\{X^-,Y\}})([\gamma_i])+2\mu_2(K_{u/v})-\mu_1(K_{u/v})\\
        \begin{cases}
            Q_{ii}=(3\Psi_{X^+}-\Psi_{\{X^-,Y\}})([\gamma_i])-\mu_3(K_{u/v})\\
            Q_{ij}=Q_{ii}+(2\Psi_{X^+}-\Phi)([\gamma_{j,i}]).
        \end{cases}
    \end{cases}
    \]   
\end{corollaryN}

\begin{proof}
    This follows immediately from Theorem \ref{Knotsthm} and
    \[
    V_{K_{u/v}}^j(q)=P^{\bigwedge^j}_{u/v}(q,a)\bigg|_{q\mapsto q^{-1},a\mapsto q^2},
    \] 
    realizing that $H_i\equiv -\Psi_{\{X^\pm,Y\}}([\gamma_x])-\mu_1(K_{u/v})$ (mod 2), where the latter is the negated version of $S_i$ coming from the (antisymmetric) colored HOMFLY-PT case. Note that $H=2A-S$.
\end{proof}

It is unclear at this time whether thinking of colored Jones polynomials in this way is particularly helpful for either of the named conjectures above, but this might be worth investigating.

\subsection{Example: Colored Jones Polynomials of the Figure-8 Knot}

We conclude by showing how the colored Jones polynomials can be computed for a concrete example, the Figure-8 knot, via the methods of this paper.

Recall Figure \ref{K52labelfig} of the geometric picture for $K_{5/2}$ with the $\kappa_i$ labeled. The figure is reproduced below.

\[
\vcenter{\hbox{\includegraphics[height=5cm,angle=0]{K52labeled.png}}}
\]

As in Section \ref{K52exsec}, we ignore the correction terms, as they correspond to a framing shift. Then, using Corollary \ref{coloredJones}, one computes

 \begin{figure}[h]
    \centering
\begin{minipage}[b]{.45\textwidth}

$
   H =
  \left[\begin{array}{ c }
    2 \\
    -1 \\
    0 \\
    1 \\
    -2 \\

    \end{array}\right]$
    \end{minipage}
    \begin{minipage}[b]{.45\textwidth}
    $ Q=\left[\begin{array}{ c c c c c}
   
    2 & 1 & 1 & 2 & 1 \\
    1 & -1 & 0 & 0 & -1 \\
    1 & 0 & 0 & 1 & 0\\
    2 & 0 & 1 & 1 & 0\\
    1 & -1 & 0 & 0 & -2\\
    
    \end{array}\right]$.
    \end{minipage}
\end{figure}

This can either be computed directly from the formulas or by using the fact that $Q$ is computed by taking the negative of the $Q$ computed by Theorem \ref{Knotsthm} and $H=2A-S$. From this, it follows that

\[
V_{4_1}^j(q)=\sum_{|\textbf{d|=j}}(-q)^{2d_1-d_2-2d_4+d_5}q^{2d_1^2-d_2^2-2d_4^2+d_5^2+2d_1d_2+2d_1d_3+4d_1d_4+2d_1d_5-2d_2d_5+2d_3d_4}\frac{(q^{-2};q^{-2})_j}{\prod_{i=1}^5 (q^{-2};q^{-2})_{d_i}},
\]

where $\textbf{d}\in \mathbb{N}^5$ and $|\textbf{d}|=d_1+...+d_5$. More concisely, we could write
\[
V_{4_1}^j(q)=\sum_{|\textbf{d|=j}}(-q)^{H\cdot\textbf{d}}q^{\textbf{d}\cdot Q\cdot \textbf{d}^T}{j \brack d_1,...,d_5}^-,
\]
where 
\[
{j \brack d_1,...,d_5}^-={j \brack d_1,...,d_5}\bigg|_{q\mapsto q^{-1}}.
\]
We can use our ``standard'' quantum multinomial if we replace $Q$ with $Q-1+I$.

\printbibliography
\end{document}